\documentclass[a4paper,10pt]{article}
\usepackage[normalem]{ulem}
\usepackage{amsmath,amsthm,amssymb,enumerate}
\usepackage[T1]{fontenc}
\usepackage{gensymb}
\usepackage[square,sort&compress,comma,numbers]{natbib}
\usepackage[sc]{mathpazo}
\usepackage{multirow} 
\usepackage{newtxtext,newtxmath}
\usepackage{subfig}
\usepackage{graphicx}
\usepackage{epstopdf}
\usepackage{cancel}
\usepackage[left=2.5cm, right=2.2cm, top=2cm]{geometry}
\usepackage{tikz}
\usepackage{caption}
\usetikzlibrary{shapes,calc}
\usepackage{verbatim}
\usepackage{mathrsfs}
\usepackage{accents}
\usepackage[utf8]{inputenc}
\usepackage{appendix}

\usetikzlibrary{decorations.pathmorphing}
\usetikzlibrary{decorations.pathreplacing}
\usetikzlibrary{positioning}
\usetikzlibrary{shapes}
\usetikzlibrary{arrows}
\usetikzlibrary{patterns}
\usetikzlibrary{fadings}
\usetikzlibrary{plotmarks}
\usetikzlibrary{calc}
\usetikzlibrary{intersections}
\tikzstyle{every picture}+=[font=\footnotesize]
\usepackage{paralist}
\usepackage{bbm}
\usepackage{latexsym}           
\usepackage{enumerate}
\usepackage{enumitem}

\setlist{noitemsep, topsep=0.8ex, partopsep=0pt
	, leftmargin=2em}
\setlist[1]{labelindent=\parindent}

\newlist{axioms}{enumerate}{1}
\setlist[axioms]{font=\bfseries}

\newlist{alphenum}{enumerate}{1}
\setlist[alphenum]{label=\textbf{(\alph*)}, leftmargin=3em}

\newlist{alphienum}{enumerate}{1}
\setlist[alphienum]{label=\textit{(\alph*)}}

\newlist{romanenum}{enumerate}{1}
\setlist[romanenum]{label=\textit{(\roman*)}}

\newlist{romaninenum}{enumerate*}{1}
\setlist[romaninenum]{label=\textit{(\roman*)}}
\usepackage[ruled]{algorithm2e}
\SetKwIF{If}{ElseIf}{Else}{if}{}{else if}{else}{endif}
\SetKwFor{For}{for}{}{endfor}
\usepackage[noabbrev, capitalise]{cleveref}
\usepackage{tabu}
\crefname{equation}{\unskip}{\unskip}
\creflabelformat{equation}{#2(#1)#3}
\usepackage[ruled]{algorithm2e}

\newtheorem{lemma}{Lemma}

  \SetFuncSty{textsc}
  \SetKwFunction{frun}{Run}
  \SetKwHangingKw{arun}{\frun}
  \SetKwFunction{fset}{Set}
  \SetKwHangingKw{aset}{\fset}
  \SetKwFunction{ffind}{Find}
  \SetKwHangingKw{afind}{\ffind}
  \SetKwFunction{fselect}{Select}
  \SetKwHangingKw{aselect}{\fselect}
  \SetKwFunction{fcompute}{Compute}
  \SetKwHangingKw{acompute}{\fcompute}
  \SetKwFunction{fsolve}{Solve}
  \SetKwHangingKw{asolve}{\fsolve}
  \SetKwFunction{fupdate}{Update}
  \SetKwHangingKw{aupdate}{\fupdate}
  \SetKwFunction{festimate}{Estimate}
  \SetKwHangingKw{aestimate}{\festimate}
  \SetKwFunction{fmark}{Mark}
  \SetKwHangingKw{amark}{\fmark}
   \SetKwFunction{fbreak}{Break}
   \SetKwHangingKw{abreak}{\fbreak}
  \SetKwFunction{frefine}{Refine}
  \SetKwHangingKw{arefine}{\frefine}
  \SetKwFunction{compute}{Compute}
  \SetKwFunction{set}{Set}

  {\algorithm}%
  {\endalgorithm}
  
\newtheorem{thm}{Theorem}[section]
\newtheorem{remark}{Remark}[section]

\newtheorem{lem}[thm]{Lemma}

\theoremstyle{definition}

\theoremstyle{remark}
\newtheorem{rem}{Remark}[section]
\newtheorem{example}{\bf Example}[section]

\numberwithin{equation}{section}

\newcommand{\bT}{\mathbb T}
\newcommand{\bV}{\text{\bf V}}

\newcommand{\cA}{\mathcal A}

\newcommand{\cE}{\mathcal E}

\newcommand{\cT}{\mathcal T}

\newcommand{\cJ}{\mathcal J}

\newcommand{\bz}{\boldsymbol{z}}

\newcommand{\bH}{\boldsymbol{H}}
\newcommand{\bL}{\boldsymbol{L}}

\newcommand{\hto}{H^2_0(\Omega)}

\newcommand{\NC}{\text{NC}}

\newcommand{\vket}{von K\'{a}rm\'{a}n equations }

\newcommand{\sit}{\sum_{T \in\mathcal{T}}\int_T}

\newcommand{\fl}{\;\text{ for all }}

\newcommand{\half}{\frac{1}{2}}
\newcommand{\trinl}{\ensuremath{|\!|\!|}}
\newcommand{\trinr}{\ensuremath{|\!|\!|}}

\newcommand{\dx}{{\rm\,dx}}

\newcommand{\ds}{{\rm\,ds}}

\newcommand{\M}{\text{M}}
\newcommand{\N}{\mathcal{N}}

\newcommand{\T}{\mathcal{T}}
\renewcommand{\P}{\mathcal{P}}

\newcommand{\hPsi}{\widehat{\Psi}}
\newcommand{\hTheta}{\widehat{\Theta}}
\newcommand{\bPsi}{\bar{\Psi}}
\newcommand{\bTheta}{\bar{\Theta}}
\newcommand{\bpsi}{\bar{\psi}}
\newcommand{\btheta}{\bar{\theta}}

\def \P{{{\cal P}}}

\allowdisplaybreaks

\def\mT{\mathbb{T}}
\def\cA{\mathcal{A}}

\def\O{\Omega}

\newcommand{\jc}{\mathrm{J}}

\def\cJ{\mathcal{J}}

\def\cT{\mathcal{T}}
\def\cE{{\mathcal{E}}(\Omega)}

\def\bchi{\boldsymbol{\chi}}
\def\brho{\boldsymbol{\rho}}
\def\bxi{\boldsymbol{\xi}}
\def\bg{\boldsymbol g}
\newcommand{\tnr}[1]{\trinl #1 \trinr}
\def\N{\mathbb{N}}

\newcommand{\norm}[1]{\left\Vert #1 \right\Vert}

\title{A posteriori error analysis for a distributed optimal control problem governed by the \vket}
\author{Sudipto Chowdhury\footnote{sudipto.choudhary@lnmiit.ac.in, Department of Mathematics, The LNM Institute of Information Technology, Jaipur 302031, India}, Asha K. Dond \footnote{ashadond@iisertvm.ac.in, School of Mathematics, Indian Institute of Science Education and Research Thiruvananthapuram 695551, India}, Neela Nataraj\footnote{neela@math.iitb.ac.in, Indian Institute of Technology Bombay, Powai, Mumbai 400076, India} and Devika Shylaja\footnote{011353@imail.iitm.ac.in, Department of Mathematics, Indian Institute of Technology Madras, Chennai 600036, India}}

\begin{document}
\maketitle
\begin{abstract} 
\noindent This article discusses numerical analysis of the distributed optimal control problem governed by the \vket defined on a polygonal domain in $\mathbb{R}^2$.  The state and adjoint variables are discretised using the nonconforming Morley finite element method and the control is discretized using piecewise constant functions. {\it A priori}  and  {\it a posteriori} error estimates are derived for the state, adjoint and control variables. The {\it a posteriori} error estimates are shown to be efficient. Numerical results that confirm the theoretical estimates are presented. 
\end{abstract}

\noindent{\bf Keywords:} {von K\'{a}rm\'{a}n equations, distributed control,  plate bending, non-linear, nonconforming, Morley FEM, {\it a~priori}, {\it a~posteriori}, error estimates}
\section{Introduction}

\subsection*{Problem formulation}

\noindent
Let $\Omega\subset\mathbb{R}^2$ be a polygonal domain and $\nu$ denotes the outward normal vector to the boundary $\partial\Omega$ of $\Omega$. This paper considers the  distributed control problem governed by the von K\'{a}rm\'{a}n equations stated below:
\begin{subequations}
	\begin{align}
	&  \min_{u \in U_{ad}} \cJ(\Psi, u) \,\,\, \textrm{ subject to } \label{cost}\\
	&  \Delta^2  \psi_1 =[\psi_1 ,\psi_2]+f +{\mathcal C}u, \,\,  \Delta^2  \psi_2 =- \half[\psi_1 ,\psi_1]\mbox{ in } \Omega, \label{state} \\
	& \psi_1=0,\,\frac{\partial \psi_1}{\partial \nu}=0\text{ and } \psi_2=0,\,\frac{\partial  \psi_2}{\partial \nu} =0
	\text{  on }\partial\Omega. \label{state3}
	\end{align}
\end{subequations}
Here the cost functional 
$\displaystyle \cJ(\Psi, u) :=\frac{1}{2}\trinl \Psi- \Psi_d \trinr_{\bL^2(\Omega)}^2  + \frac{\alpha}{2}\|{u}\|_{L^2(\omega)}^2 \: \dx$, the state variable  $\Psi:=(\psi_1,\psi_2)$, where $\psi_1$ and $\psi_2$ correspond to the displacement and Airy-stress,  $\Psi_d:={(\psi_{d,1}, \psi_{d,2})}\in \bL^2(\Omega):=L^2(\Omega)\times L^2(\Omega)$ is the prescribed desired state for $\Psi$, $\trinl \Psi- \Psi_d \trinr_{\bL^2(\Omega)}^2:=\sum_{i=1}^2\|\psi_i-\psi_{d,i}\|_{L^2(\Omega)}^2$, $\alpha>0$ is a fixed regularization parameter, $U_{ad} \subset L^2(\omega)$, $\omega \subset \Omega$ is a non-empty, closed, convex and bounded set of admissible controls defined by
$$ \displaystyle
U_{ad} = \{   u \in L^2(\omega) : u_a \le  u(x) \le u_b  \;  \mbox{ for almost every }  x \mbox { in }  \omega \},$$
$u_a \le  u_b \in {\mathbb R}$ are given, $\Delta^2$ denotes the fourth-order biharmonic operator, the von K\'{a}rm\'{a}n bracket $ [\eta,\chi]:=\eta_{xx}\chi_{yy}+\eta_{yy}\chi_{xx}-2\eta_{xy}\chi_{xy}={\rm cof}(D^2\eta):D^2\chi$ with the co-factor matrix 
${\rm cof}(D^2\eta)$  of $D^2\eta$,  $f \in L^2(\Omega)$, and {${\mathcal C }\in {\mathcal L}(L^2(\omega), L^2(\Omega))$ is the extension operator defined by 
	$
	{\mathcal C }u(x)=u(x)\mbox{ if }x\in\omega \mbox{ and }  {\mathcal C }u(x)=0 \mbox{ if }x\not\in\omega.$ }

\subsection*{Motivation}

The {\bf optimal control problem} governed by the \vket  \eqref{cost}-\eqref{state3} is analysed in \cite{ngr} for $C^1$ conforming finite elements. In \cite{SCNNDS}, {\it a priori} error estimates are derived under minimal regularity assumptions on the exact solution where the state and adjoint variables are discretised using Morley finite element methods (FEMs). The discrete trilinear form in the weak formulation \cite{SCNNDS} is derived after an integration by parts. In this article, a simplified form of  the trilinear form that involves the von K\'{a}rm\'{a}n bracket itself is considered. This choice of the trilinear form \cite{CCGMNN18} is appropriate for  both reliable and efficient {\it a posteriori} estimates. {\it  To the best of our knowledge, there are no results in literature that discuss {\it a posteriori} error analysis for the approximation of regular solutions of optimal control problems governed by von K\'{a}rm\'{a}n equations.} Recently, {\it a posteriori} error analysis for the optimal control problem governed by second-order stationary Navier-Stokes equations is studied in \cite{aposterioriNScontrol_2021} with conforming finite element method under smallness assumption on the data. The  trilinear form in \cite{aposterioriNScontrol_2021} vanishes whenever the second and third variables are equal, and satisfies the anti-symmetric property with respect to the second and third variables and this aids the {\it a posteriori} error analysis. This paper discusses approximation of regular solutions for fourth-order semi-linear problems without any smallness assumption on the data. Moreover, the trilinear form for \vket does not satisfy the properties stated above and hence leads to additional challenges in the analysis. 

\medskip
\noindent The {\bf \vket} \cite{CiarletPlates}  that describes the bending of very thin elastic plates offers challenges in its numerical approximation; mainly due to its nonlinearity and higher order nature; we refer to \cite{CiarletPlates, Knightly, Fife, Berger,BergerFife, BlumRannacher} and the references therein for the existence of solutions, regularity and bifurcation phenomena of the von K\'{a}rm\'{a}n equations. The numerical analysis of \vket has been studied using conforming FEMs in \cite{Brezzi,ng1},  nonconforming Morley FEM in \cite{ng2,carstensen2017nonconforming}, mixed FEMs  in \cite{Miyoshi,Chen2020AMF}, discontinuous Galerkin methods and $C^0$ interior penalty methods in \cite{brennernew,CCGMNN18}. 

\medskip

\noindent {\bf Nonconforming Morley} FEM based on piecewise quadratic polynomials in a triangle is more elegant, attractive and simpler for fourth-order problems. However, the convergence analysis offers a lot of novel challenges in the context of control problems governed by semilinear problems with trilinear nonlinearity since the discrete space $V_ {\M}$ is not a subspace of $H^2_0(\O)$. The adjoint variable in the control problem  satisfies a fourth-order linear problem with lower-order terms and its {\it a priori} and {\it a posteriori} analysis with Morley FEM offers additional difficulties. 

\medskip
\noindent The regularity results of \vket in \cite{BlumRannacher} extends to the regularity of the state and adjoint variables of the control problem \cite{ngr} and ensures that the  optimal state and adjoint variables belong to $\hto\cap H^{2+\gamma}(\Omega)$, where $\gamma\in (\half,1]$, referred to as the index of elliptic regularity, is determined by the interior angles of $\Omega$. Note that when $\Omega$ is convex, $\gamma=1$.

\smallskip

\noindent 

\subsection*{Contributions}

\noindent In continuous formulation (see \eqref{wform}) and the conforming FEM \cite{ngr}, the trilinear form $b(\bullet,\bullet,\bullet)$ is symmetric with respect to all the three variables; that makes the analysis simpler to a certain extent. However, for fourth-order systems, nonconforming Morley FEM is attractive and is {\it a method of choice} \cite{carstensen2017nonconforming} and this motivated the {\it a priori} analysis for the optimal control problem in  \cite{SCNNDS}. The expression for the discrete trilinear form $\displaystyle b_\NC(\bullet,\bullet,\bullet)$ in \cite{SCNNDS}  defined as $ \half\sit {\rm{cof}} (D^2 \eta_\M) D \chi_\M\cdot D\varphi_\M \dx$ for all Morley functions $ \eta_\M,\chi_\M \; \text{and} \; \varphi_\M$ is obtained after an integration by parts, where $\T$ denotes the triangulation of $\O$. This form is symmetric with respect to the second and third variables. Though this choice of trilinear form leads to optimal order error estimates for the optimal control problem \eqref{cost}-\eqref{state3}, it leads to terms that involve {\it averages} in the reliability analysis of the state equations (as in the case of Navier-Stokes equation considered in  \cite{carstensen2017nonconforming}). The {\it efficiency} estimates are unclear in this context. 
To overcome this, a more natural trilinear form $\displaystyle b_\NC(\eta_\M,\chi_\M,\varphi_\M)=-\half\sit [\eta_\M,\chi_\M]\varphi_\M \dx$  that  is symmetric with respect to the first and second variables is chosen in this article. The {\it a priori} and {\it a posteriori} analysis for the {\it state equations}  are  discussed in \cite{carstensen2017nonconforming,CN2020}. The {\it a posteriori} analysis for the fully discrete optimal control problem governed by \vket addressed in this article is novel and involves additional difficulties. For instance, the adjoint system in this case involves lower-order terms with leading biharmonic operators. {\it A posteriori} analysis for biharmonic operator with lower-order terms is a problem of independent interest. 

\medskip
\noindent Thus the contributions of this article can be summarized as follows.
\begin{itemize}
    \item 
    For a formulation that is different from that in \cite{SCNNDS}, optimal order {\it a priori} error estimates in energy norm when state and adjoint variables are approximated by Morley FEM and  linear order of convergence  for control variable in $L^2$ norm when control is approximated using piece-wise constants are outlined. 
    \item Reliable and efficient  {\it a posteriori} error estimates that drive the adaptive refinement for the optimal state and adjoint variables in the energy norm and control variable in the $L^2$ norm are developed. The approach followed in this paper provides a strategy for the nonconforming FEM analysis of  optimal control problems governed by higher-order semi-linear problems. 
    \item Several auxiliary results that are derived  will be of interest in other applications - for example, optimal control problems governed by Navier-Stokes problems in the stream-vorticity formulation.
    \item The paper illustrates results of computational experiments that validate both theoretical  {\it a priori} and {\it a posteriori} estimates for the optimal control problem under consideration. 
\end{itemize}


\subsection*{Organisation}
\noindent The remaining parts of this paper are organised as follows. Section \ref{sec.weak} presents the weak and the nonconforming finite element formulations for \eqref{cost}-\eqref{state3}. The state and adjoint variables are discretised using Morley finite elements and the control variable is discretised using piecewise constant functions. Section \ref{sec.discreteformulation} deals  some preliminaries related to Morley FEM. The boundedness properties of the discrete bilinear and trilinear forms that are crucial for the error analysis are  discussed in this section. {\it A priori} error estimates for the state, adjoint and control variables under minimal regularity assumptions on the exact solution are stated in Section \ref{sec.apriori}. Note that the analysis differs from \cite{SCNNDS} due to a different trilinear form. Section \ref{sec.reliability} develops reliable {\it a posteriori} estimates for the state, adjoint and control variables of the optimal control problem. Section \ref{sec.efficiency} establishes efficiency results for the optimal control problem. Results of numerical experiments that validate theoretical estimates are presented in Section \ref{sec.numericalresults}. Finally, details of proofs of some results stated in Section \ref{sec.apriori} are derived in the Appendix.

\subsection*{Notations}

\noindent
 Throughout the paper, standard notations on Lebesgue and
 Sobolev spaces and their norms are employed. The standard semi-norm and norm on $H^{s}(\Omega)$ (resp. $W^{s,p} (\Omega)$) for $s>0$ and $1 \le p \le \infty$ are denoted by $|\cdot|_{s}$ and $\|\cdot\|_{s}$ (resp. $|\cdot|_{s,p}$ and $\|\cdot\|_{s,p}$ ) and norm in $L^\infty(\O)$ is denoted by $\|\cdot \|_{0,\infty}$. The norm in $H^{-s}(\Omega)$ is denoted by $\|\cdot\|_{-s}$. The standard $L^2$ inner product and norm are denoted by $(\cdot, \cdot)$ and $\|\cdot\|.$ The notation $\|\cdot\|$ is also used to denote the operator norm and should be understood from the context. The notation ${\bH}^s(\Omega)$ (resp. ${\boldsymbol{L}}^p(\Omega)$) is used to denote the product space  
 $H^{s}(\Omega) \times H^s(\Omega)$ (resp. $L^p(\Omega) \times L^p(\Omega)$). For all $\Phi = (\varphi_1,\varphi_2) \in {\bH}^s(\Omega) \; ( \text{resp. } {\boldsymbol L}^2(\Omega))$, the product space is equipped with the norm
 $
 \trinl{\Phi}\trinr_{s}:=(\| \varphi_1\|_s^2 +\|\varphi_2\|_s^2)^{1/2} \; 
 (\text{ resp. } \trinl{\Phi}\trinr:=(\| \varphi_1\|^2 +\|\varphi_2\|^2)^{1/2}). \; 
 $
 The notation $a\lesssim b$ (resp. $a \gtrsim  b$) means there exists a generic mesh independent constant $C$ such that $a\leq Cb$ (resp. $a\ge Cb$).  The positive constants $C$ appearing in the inequalities denote generic constants which do not depend on the mesh-size. 
 
 \section{Weak and Finite Element Formulations}\label{sec.weak}
\noindent 
 In this section, the weak and Morley FEM formulations for \eqref{cost}-\eqref{state3} and some auxiliary results are presented.
 
 \subsection{Weak Formulation}
 
 \medskip \noindent Let $V:=H^2_0(\Omega)$,  $\bV=V\times V$, the bilinear (resp. trilinear) form $a(\bullet,\bullet):V\times V\rightarrow \mathbb{R}$ (resp. $b(\bullet,\bullet,\bullet):V\times V\times V\rightarrow \mathbb{R}$) be defined by
 \begin{align*}
 a(\varphi_1,\varphi_2):=\int_{\Omega}D^2\varphi_1:D^2\varphi_2\; \dx \;  \; 
 (\mbox{resp. }b(\varphi_1,\varphi_2,\varphi_3):=-\frac{1}{2}\int_{\Omega}[ \varphi_1,\varphi_2]\varphi_3\; \dx).
 \end{align*}
  For all $\varphi_1, \varphi_2,\varphi_3\in V$, the bilinear  and trilinear forms $a(\bullet, \bullet)$ and $b(\bullet, \bullet, \bullet)$ satisfy
 \begin{equation*}
 |{a}(\varphi_1, \varphi_2)| \leq \|\varphi_1\|_2 \: \|\varphi_2\|_2, \; \; |{a}(\varphi_1,\varphi_1)| \geq \|   {\varphi_1} \|_2^2 \; 
 \; {\rm and }  \; |b(\varphi_1,\varphi_2,\varphi_3)|  \lesssim \|\varphi_1\|_2 \: \| \varphi_2 \|_2 \: \|\varphi_3\|_2.
 \end{equation*}
 \medskip\noindent 
 The weak formulation that corresponds to \eqref{cost}-\eqref{state3} seeks $(\Psi, u) \in  \bV \times  U_{ad}$  such that 
 \begin{subequations}\label{wform}
 	\begin{align}
 	&  \min_{(\Psi, \,u) \in  \bV \times  U_{ad}} \cJ(\Psi, u) \,\,\, \textrm{ subject to } \label{cost1}\\
 	& a(\psi_1,\varphi_1)+ b(\psi_1,\psi_2,\varphi_1)+b(\psi_{2},\psi_{1},\varphi_1)=(f+\mathcal{C}u,\varphi_1)   \; \; \fl \varphi_1\in V\label{wforma},\\
 	& a(\psi_2,\varphi_2)-b(\psi_1,\psi_1,\varphi_2)   =0       \; \;     \fl \varphi_2 \in V\label{wformb}.
 	\end{align}
 \end{subequations}
 For a given $u \in L^2(\omega)$, \eqref{wforma}-\eqref{wformb} possesses at least one solution \cite{Knightly}. 
 \medskip
For all  ${\boldsymbol \xi}=(\xi_1,\xi_2)$, $\Phi=(\varphi_1,\varphi_2)$, ${\boldsymbol \eta}=(\eta_1,\eta_2) \in \bV$, the operator form for \eqref{wforma}-\eqref{wformb} is 
 \begin{equation}
 \Psi \in \bV, \; \; {\mathcal A} \Psi +{\mathcal B}(\Psi) = {\bf F} +{\mathbf C}{\bf u}  \mbox{ in } \bV', \label{of}
 \end{equation}
 with ${\mathcal A} \in {\mathcal L} (\bV, \bV')$ defined by
 $
 \langle {\mathcal A} {\boldsymbol \xi} ,\Phi   \rangle$ \footnote{The subscripts in the duality pairings are omitted for notational convenience.}$=
 A ({\boldsymbol \xi} ,\Phi)= a(\xi_1,\varphi_1) + a (\xi_2, \varphi_2)$, 
 ${\mathcal B}$ from $\bV$ to
 $\bV'$ defined by 
 $\langle\mathcal{B}({\boldsymbol \eta}),\Phi\rangle= B( {\boldsymbol \eta}, {\boldsymbol \eta}, {\Phi}  )$ where $B( {\boldsymbol \eta}, {\Phi},{\boldsymbol \xi}  )=b(\eta_1,\varphi_2, \xi_1) +b(\eta_2,\varphi_1, \xi_1) - b(\eta_1,\varphi_1, \xi_2)$, $  {\bf F} = \left( \begin{array}{c}
 f  \\
 0  \end{array} \right), $ ${\mathbf C}{\bf u} =
 \left( \begin{array}{c}
 {\mathcal C}u \\
 0  \end{array} \right),$ $
 {\bf u} =\left( \begin{array}{c}
 u \\
 0  \end{array} \right)$,
 { and } 
$ ({\bf F}+ {\mathbf C}{\bf u} ,\Phi):=(f+{\mathcal C}u,\varphi_1)$.
 \smallskip
 
 \noindent
The state equations in \eqref{wforma}-\eqref{wformb} can be written as
$ N(\Psi; \Phi):=  A(\Psi,\Phi)+B(\Psi,\Psi,\Phi)-( {\bf F} +{\mathbf C}{\bf u},\Phi)=0 \mbox{ for all } \Phi \in \bV.$ The first and second-order Fr\'{e}chet derivatives of $ N(\Psi)$ at $\Psi$ in the direction $\bxi$ are given by
$DN(\Psi;\bxi,\Phi):=\langle{\mathcal A}{\boldsymbol \xi}+{\mathcal B}'(\Psi) {\boldsymbol \xi}, \Phi \rangle$
and $D^2N(\Psi;\bxi,\bxi,\Phi):= \langle {\mathcal B}''(\bxi,{ {\boldsymbol \xi}}), \Phi \big \rangle$, where
the operators ${\mathcal B}'(\Psi) \in {\mathcal L}(\bV, \bV')${\footnote {The same notation $'$ is used either to denote the Fr\'echet derivative of an operator or the dual of a space, but the context helps to clarify its precise meaning.}} and $,  {\mathcal B}''(\Psi,{ {\boldsymbol \xi}})\in {\mathcal L}(\bV\times \bV, \bV')$ are given by 
$\langle {\mathcal B}'(\Psi){ {\boldsymbol \xi}}, \Phi \rangle :=
2B(\Psi, { {\boldsymbol \xi}}, \Phi)$ and $\big \langle {\mathcal B}''(\Psi,{ {\boldsymbol \xi}}), \Phi \big \rangle :=
2B(\Psi, { {\boldsymbol \xi}}, \Phi)$.
 
 \medskip \noindent
 For a given $u \in L^2(\omega)$,  a solution $\Psi$ of \eqref{wforma}-\eqref{wformb} is said to be {\it regular} \cite[Definition 2.1]{SCNNDS} if the linearized form is well-posed. In this case, the pair $(\Psi,u)$ also is referred to as a regular  solution to \eqref{state}-\eqref{state3}. 
 The pair $(\bar \Psi, \bar u) \in  
 	\bV \times U_{ad}$ is a {\it local solution} \cite{cmj} to \eqref{wform} if and only if $(\bar \Psi, \bar u)$ satisfies \eqref{wforma}-\eqref{wformb} and there exist neighbourhoods  ${\mathcal O}(\bar\Psi)$ of $\bar \Psi$ in $\bV$ and ${\mathcal O}(\bar u)$  of ${\bar u}$ in $L^2(\omega)$ such that 
 	$\cJ({\bar \Psi},{\bar u} ) \le \cJ(\Psi, u)$ for all pairs $(\Psi, u) \in {\mathcal O}(\bar\Psi)  \times (U_{ad} \cap {\mathcal O}(\bar u))$ that satisfy \eqref{wforma}-\eqref{wformb}.
 \begin{thm}\label{th2.5}\cite{cmj}
 	Let $(\bar{\Psi}, \bar{u}) \in \bV \times L^2(\omega)$
 	be a regular solution to
 	\eqref{wform}. 
 	Then there exist an open ball
 	${\mathcal O}(\bar { u})$ 
 	of $\bar{ u }$ in $L^2(\omega)$, an open ball
 	${\mathcal O}({\bar \Psi})$ of
 	$\bar \Psi$ in $\bV$, and a mapping $G$ from
 	${\mathcal O}(\bar u)$ to ${\mathcal O}(\bar\Psi)$ of class $C^\infty$,
 	such that, for all $u\in {\mathcal O}(\bar u)$, $\Psi_{u}=G({u})$  is the unique solution in ${\mathcal
 		O}(\bar\Psi)$ to  $\eqref{of}$. {}{Thus, $G'(u)= ({\mathcal A} + {\mathcal B}'(\Psi_{u}))^{-1}$ is uniformly bounded from a smaller ball into a smaller ball} {}$($these smaller balls are still denoted by $ {\mathcal O}(\bar u)$ and ${\mathcal O}(\bar\Psi)$ for notational simplicity$)$. 
 	Moreover,  if
 	$ G'({u}){v} =: \mathbf{z}_{{v}} \in
 	\bV$ and 
 	$ G''({u}) {v}^2 =: \mathbf{w} \in
 	\bV$, then $\mathbf{z}_{v}$ and $\mathbf{w}$ satisfy 
 	\begin{eqnarray}
 	{\mathcal A}\mathbf{z}_{{v}} +
 	{\mathcal B}'(\Psi_{u})\mathbf{z}_{{v}} =
 	{\mathbf C}\mathbf{v}\quad \mathrm{in\ } \bV',
 	\; 
 	{\mathcal A}{\mathbf{w}} + {\mathcal B}'(\Psi_{u}){\mathbf{w}} +{\mathcal B}''({\mathbf{z}}_v,{\mathbf{z}}_v)
 	= 0\quad \mathrm{in\ } \bV', \label{E2.7}
 	\end{eqnarray}
 where ${\mathcal A} + {\mathcal B}'(\Psi_{u})$ is an isomorphism from $\bV$
 	into ${\bV}'$ for all $u\in {\mathcal O}(\bar u)$. Moreover,
 		$\|{\mathcal A} + {\mathcal B}'(\Psi_{u})\|_{{\mathcal L}(\bV, \bV')}$ and $\|({\mathcal A} + {\mathcal B}'(\Psi_{u}))^{-1}\|_{{\mathcal L}(\bV', \bV)}$ are uniformly bounded. Also, $  \|\mathbf{z}_{{v}}\|_{2}\le \| G'({u}) \|_{{\mathcal L}(L^2(\omega), H^2(\Omega))}\|v\|_{L^2(\omega)}. $ \qed
 \end{thm}
\begin{rem}
  	{}{The dependence of $\Psi$ with respect to $u$ is made explicit with the notation $\Psi_{u}$ only when it is necessary.}{}
  \end{rem}
\begin{rem}
 \noindent The regular solution $\bPsi$  to \eqref{wform} satisfies the inf-sup condition
  \begin{align}\label{inf-sup_apost_psi}
  0<\beta:=\inf_{\substack{\bxi\in \bV\\ \trinl\bxi\trinr_2=1}}\sup_{\substack{\Phi\in \bV\\ \trinl\Phi\trinr_2=1}}\langle  {\mathcal A}{\boldsymbol \xi}+{\mathcal B}'(\bPsi) {\boldsymbol \xi},\Phi\rangle, \text{ and this leads to }  \|(\mathcal{A}+\mathcal{B}'({\bPsi}))^{-1}\|_{\mathcal{L}(\bV',\bV)} = 1/\beta.
  \end{align}
  \end{rem}
\noindent 
The existence of a solution to \eqref{wform} can be obtained using standard arguments of considering a minimizing sequence, which is bounded in $\bV \times L^2(\omega)$, and passing to the limit \cite{Lions1, hou, fredi2010}.
 {}{
 	\begin{lem}[\it{a priori bounds, regularity and convergence}]\cite[Lemmas 2.7 {,2.9} \& 2.10]{ngr}\cite[Theorem 2.1]{CN2020}\label{ap}\noindent
 		\begin{itemize}
 			\item[(a)]For $f\in H^{-1}(\Omega)$ and $u \in L^2(\omega)$, the solution $\Psi$ of \eqref{wforma}-\eqref{wformb} belongs to $\bV \cap {\bH}^{2 + \gamma} (\Omega)$,  $\gamma \in (1/2,1]$, and  satisfies the $a~priori$ bounds
 			$
 			\trinl\Psi\trinr_2  \lesssim (\|f\|_{-1} + \|u\|_{L^2(\omega)}), \; 
 			\trinl\Psi\trinr_{2+\gamma}   \lesssim  (\|f\|^3_{-1}+ \|u\|^3_{L^2(\omega)} +
 			\|f\|_{-1} + \|u\|_{L^2(\omega)}). 
 			$
 			\item[(b)] The solution ${\bf z}_{ v }$ of the linearized problem \eqref{E2.7} also belongs to $\bV \cap{\bH}^{2 + \gamma}(\Omega)$, and  satisfies the $a~priori$ bound
 			$\trinl  \mathbf{z}_v \trinr_{2 + \gamma} \lesssim \|v\|_{L^2(\omega)}.$
 			\item[(c)]
 			Let $({\bar \Psi}, {\bar u})$ be a regular solution to \eqref{wform} and  $({u_k})_k$
 			be a sequence in ${\mathcal O}(\bar u)$ weakly converging to
 			$\bar u$ in $L^2(\omega)$. Let $\Psi_{u_k}$ be the
 			solution to \eqref{of} in ${\mathcal O}(\bar\Psi)$
 			that corresponds to $u_k$.  Then,  $(\Psi_{u_k})_k$ converges to $\bar \Psi$ in
 			$\bV$. \qed
 		\end{itemize} 
 	\end{lem}
\noindent {\it Local solutions} $(\bar \Psi, \bar u)$ to \eqref{wform}  such that the pair is a {\it regular solution} to \eqref{of} are approximated in this article. The {\it optimality system} for the optimal control problem \eqref{wform} is:
 	\begin{subequations} \label{opt_con}
 		\begin{align}
 		&   {A}({\bar \Psi},\Phi)+{B}({\bar \Psi},{\bar \Psi},\Phi)=(\mbox{\bf F} + {\bf C}{\bar{\bf u}}, \Phi) \fl \Phi \in \bV \; \;  \rm{( State \; equations)}\label{state_eq}
 		\\  
 		& {A}(\Phi, {\bar\Theta})+2{B}({\bar \Psi},\Phi,{\bar \Theta})=(\bar \Psi- \Psi_d, \Phi) \fl \Phi \in \bV \; \; \rm{(Adjoint \; equations)}\label{adj_eq}
 		\\
 		&  \left({\mathbf C}^* {\bar \Theta} + \alpha {\bf \bar u} ,  {\bf u} - {\bf \bar u} \right)_{{\boldsymbol{L}}^2(\omega)} \ge 0  \fl  {\bf u} =(u,0)^T, \; {u \in U_{ad}} \; \;  \rm{(First \: order \: optimality \: condition)} \label{opt3}
 		\end{align}
 	\end{subequations}
 	where ${\bar \Theta}$ is the adjoint state and $\mathbf{C}^{*}$ denotes the adjoint of $\mathbf{C}$. For almost all $x \in \Omega$, the optimal control ${\bf \bar u}(x):=({\bar u}(x),0)$ in \eqref{opt3} satisfies
 	\begin{equation}\label{rep}
 {{ \bf \bar u}(x)= \Pi_{[u_a, u_b]} \left( - \frac{1}{\alpha} ({\mathbf C}^* {\bar \Theta} ) \right),}
 	\end{equation}
 	where {{$\bar \Theta=( \bar\theta_{1},\bar\theta_{2} )$ and }}the projection operator $\Pi_{[u_{a},u_{b}]}$ is defined by $\Pi_{[u_{a},u_{b}]}(g):= \min\{u_{b}, \max \{u_{a},g\} \}.$

\subsection {Discrete formulation}

Let $\T$ be an admissible and regular triangulation of the domain $\Omega$ into simplices in $\mathbb R^2$, $h_T$ be the diameter of  $T \in \T$ and $h:=\max_{T \in \T} h_T$. For a non-negative integer $k \in \N_0$, ${\mathcal P}_{k}(\mathcal{T})$ denotes the space of piece-wise polynomials of degree at most equal to $k$. Let $\Pi_k$ denote the $L^2$ projection onto the space of piece-wise polynomials ${\mathcal P}_k(\mathcal{T})$. The oscillation of $f$ in $\cT$ reads ${\rm{osc}}_k(f,\cT)=\|h_{}^2(f-\Pi_k f)\|$ for $k \in \N_0$. 

 	
 	\smallskip
 	
 	
 	\smallskip
 	\noindent
 	The nonconforming {\it Morley element space} $V_\M$  is defined by
 	\begin{align*}
 	V_\M& :=\{ v_\M\in {\mathcal P}_2(\T)|
 	 v_\M \text{ is continuous at the interior vertices}  \text{ and vanishes at the vertices } 
 	 \text{ of }\partial \Omega; \:
 	D_{\NC}{ v_\M} \\
 	&\text{ is continuous at the midpoints of interior edges and }
 	\text{vanishes at the  midpoints of the edges of } \partial \Omega
 	\},
 	\end{align*}
 	and  is equipped with the norm $\|\bullet\|_{\NC}$ defined by  $ \|\varphi\|_{\NC}= \bigg(\sum_{T\in \mathcal{T}}\|D^2_\NC\varphi\|_{L^2(T)}^2\bigg)^{1/2}$.  Here $D_{\NC} \bullet$ and  $D^2_\NC \bullet$ denote the piecewise gradient and Hessian of the arguments on triangles  $T \in \mathcal{T}$. 
 	For $\varphi \in V$,  $\|\varphi\|_2=\|\varphi\|_\NC$ and thus $\| \bullet \|_\NC$ denotes the norm in $V+V_\M$.  Let  $\bV_\M:= V_\M \times V_\M$ and for $\Phi=(\varphi_1,\varphi_2) \in \bV_\M, \; \trinl \Phi \trinr_\NC^2:=\|\varphi_1\|_{\NC}^2+\|\varphi_2\|_{\NC}^2.$
 For a non-negative integer $m$, and $\Phi =(\varphi_1,\varphi_2) \in W^{m,p}(\Omega;\T)$,  where $ W^{m,p}(\Omega;\T)$ denotes the broken Sobolev space with respect to $\T$, $\trinl \Phi \trinr_{m,p,h}^2:=|\varphi_1|_{m,p,h}^2+|\varphi_2|_{m,p,h}^2$,  and $|\varphi_i|_{m,p,h}=(\sum_{T\in \mathcal{T}}|\varphi_i|^{p}_{m,p,T})^{1/p}$, $i=1,2$; with $|\cdot|_{m,p,T}$ denoting the usual semi-norm in $W^{m,p}(T)$. When $p=2$, the notation is abbreviated as $|\cdot|_{m,h}$ and $\trinl\cdot\trinr_{m,h}$.
 
 
 \medskip

 \noindent For all $ \eta_\M,\chi_\M \; \text{and} \; \varphi_\M\in V_ {\M}$, define the discrete bilinear and trilinear forms by
 $$ a_\NC(\eta_\M,\chi_\M):=\sit D^2 \eta_\M:D^2\chi_\M \dx \mbox{ and }  b_\NC(\eta_\M,\chi_\M,\varphi_\M):=-\half\sit [\eta_\M,\chi_\M]\varphi_\M \dx.$$ 
 
 \noindent Similarly, for $ \Xi_\M=(\xi_{1},\xi_{2}),$ 
 $\Theta_\M=(\theta_{1},\theta_{2})$, $ \Phi_\M=(\varphi_{1},\varphi_{2})\in  \bV_\M$, define
 \begin{align*}
 & A_\NC(\Theta_\M,\Phi_\M):=a_\NC(\theta_1,\varphi_1)+a_\NC(\theta_2,\varphi_2), \; F_\NC(\Phi_\M):=\sit f\varphi_1\dx  \; \; \text{and}\\
 &B_\NC(\Xi_\M,\Theta_\M,\Phi_\M):=b_\NC(\xi_{1},\theta_{2},\varphi_{1})+b_\NC(\xi_{2},\theta_{1},\varphi_{1})-b_\NC(\xi_{1},\theta_{1},\varphi_{2}). 
 \end{align*} 
 The above definitions of the bilinear and trilinear forms are meaningful for functions in $V+V_{\M}$ 
 (resp. ${\bf V} + {\bf V}_{\M}$). {{Note that for all ${ \xi},{ \theta},\varphi\in V$, $a_\NC(\xi,\theta)=a(\xi,\theta)$ and $b_\NC(\xi,\theta,\varphi)=b(\xi,\theta,\varphi)$.}}
 \medskip

\noindent  Analogous to the definition of nonlinear operator $\mathcal{B}:\bV\rightarrow\bV'$, define the discrete counterparts $\mathcal{B}_{\NC}:\bV+\bV_\M\rightarrow  (\bV+\bV_\M)'$
 as $
 \langle\mathcal{B}_{\NC}(\Psi),\Phi\rangle =B_\NC(\Psi,\Psi,\Phi) \fl \Psi, \Phi \in \bV+\bV_\M.
 $
 The  Fr\'echet derivative of $\mathcal{B}_\NC$ around $\Psi$ at the direction of ${\boldsymbol \xi}$ is denoted by $\mathcal{B}_\NC'(\Psi)({\boldsymbol \xi})$ and is
 \begin{equation}\label{B.derivative}
 \langle\mathcal{B}_{\NC}'(\Psi)({\boldsymbol \xi}),\Phi\rangle =2B_\NC(\Psi,{\boldsymbol \xi},\Phi)\quad  \fl \Psi, \Phi,{\boldsymbol \xi} \in \bV+\bV_\M.
 \end{equation}
 \noindent	The {\it admissible space for discrete controls} is 
$ U_{h, ad}:=\left\{{{u}}\in L^2(\omega):  u|_T\in \mathcal{P}_0(T), \; u_a \leq  u \leq
 u_b \mbox{ for all }T\in
 \mathcal{T} \right\}.$
The discrete control problem associated with \eqref{wform} reads
 \begin{subequations}\label{discrete_cost}
 	\begin{align} 
 	&  \min_{(\Psi_\M, {u}_h) \in  \bV_\M \times  U_{h,ad}} \cJ(\Psi_\M, u_h) \,\,\, \textrm{ subject to } \\
 	&  {A}_\NC(\Psi_\M,\Phi_\M)+{B}_\NC(\Psi_\M,\Psi_\M,\Phi_\M)=({\bf F} + {\mathbf C}{\bf u}_h, \Phi_\M) \fl \Phi_\M\in \bV_\M. \label{dis_cost_b}
 	\end{align}
 \end{subequations}
 The {\it discrete first order optimality system} that comprises  of the discrete state and adjoint equations and the first order optimality condition corresponding  to \eqref{discrete_cost} is
 \begin{subequations}\label{discrete.opt}
 	\begin{align}
 	&  A_\NC({\bar \Psi_\M},\Phi_\M)+B_\NC({\bar \Psi_\M}, {\bar \Psi_\M},\Phi_\M)=({\bf F} +  {\mathbf C}{{\bar{\bf u}_h}}, \Phi_\M) \fl \Phi_\M \in \bV_\M   \label{state_eq1}
 	\\  
 	& A_\NC(\Phi_\M, {\bar\Theta_\M})+2B_\NC({\bar \Psi_\M}, \Phi_\M, {\bar \Theta_\M})=(\bar \Psi_\M- \Psi_d, \Phi_\M) \fl \Phi_\M \in \bV_\M \label{adj_eq1}
 	\\
 	&  \left( {\mathbf C}^* {\bar \Theta_\M} + \alpha {{\bf \bar u}_h} ,  {{\bf u}_h} - {{\bf \bar u}_h} \right)_{} \ge 0  \fl {\bf u}_h =(u_h,0)^T, \; {u_h \in U_{h,ad}},\label{opt31}
 	\end{align}
 \end{subequations}
 where $\bar \Theta_\M \in \bV_\M $ denotes the discrete adjoint variable that corresponds to the optimal state variable ${\bar \Psi_\M} \in \bV_\M$.
  \subsection{ Auxiliary results}\label{sec.discreteformulation}
 This section presents some auxiliary results that are useful to establish the proof of both the {\it a priori} as well as {\it a posteriori} error estimates. These are very crucial properties of Morley interpolation and companion operators that aids the analysis. Some properties of the discrete bilinear and trilinear forms that will be used throughout in the article are also presented.
\begin{lem}[\it Morley interpolation operator] \label{Morley_Interpolation} 
	\cite{CCDGJH14, DG_Morley_Eigen, CCP} 
	For $v \in V $, the Morley interpolation operator 
	$I_{\rm M}: V \rightarrow V_{\rm M}$ defined by
	$
	(I_{\rm M} v)(z)=v(z) \text{ for any } \text{ vertex $z$ of } \T \text{ and } 
	\int_E\frac{\partial I_{\rm M} v}{\partial \nu_E}\ds=\int_E\frac{\partial v}{\partial \nu_E}\ds \text{ for any  edge } E  \text{ of } \T $
	satisfies $(a)$ the integral mean property $D^2_{\rm NC} I_{\rm M} =\Pi_0 D_{\rm NC}^2$ of the Hessian, $(b)  \sum_{m=0}^1h^{m-2}|(1-I_{\rm M}) v|_{H^m(T)} \leq C_I |(1-I_{\rm M}) v|_{H^2(T)} =C_I(\| D^2v\|_{L^2(T)}-\| D^2I_{\rm M} v\|_{L^2(T)})$ for all $v \in H^2(T)$ and $T \in \cT$, \text{ and }  
	$(c)\|(1-I_{\rm M})v\|_{\rm NC}\lesssim h^{\gamma}||v\|_{2+\gamma}$ for all $v \in V \cap H^{2+\gamma}(\O) $.

\end{lem}
\begin{lem}[\it Companion operator]\label{hctenrich} \cite{DG_Morley_Eigen, CCP} For any $v_{\rm M} \in 
	V_{{\rm M}}$,  there exists $J:{V}_{{\rm M}}\to V$ such that 
	\begin{align*}
		&(a) \; I_{\rm M} J v_{\rm M}= v_{\rm M} \; \text{ for all } v_{\rm M} \in V_{{\rm M}}, \quad  
    (b)  \; \Pi_0((1 - J)  v_{\rm M}) =0, \quad 
	(c) \; \Pi_0D_{\rm NC}^2((1-J) v_{\rm M}) =0, \\
	& (d)\;  \| h_{T}^{-2}((1-J)v_{\rm M})\| + \| h_{T}^{-1}D_{\rm NC}((1-J)v_{\rm M})\| + \| D^2_{\rm NC}((1-J)v_{\rm M})\| \le \Lambda_\jc \min_{v \in V}\| D_{\rm NC}^2(v_{\rm M}-v)\|, \\
& (e)   
\sum_{m=0}^{2}h_{T}^{2m-4}\|(1-J)v_{\rm M}\|^2_{H^m(T)} \leq C_J^2 \sum_{E \in \mathcal{E}(\Omega(T))}
	h_E \| [D_{\rm NC}^2v_{\rm M} ]_E  \tau_E\|^2_{L^2(E)}  \lesssim \min_{v \in V} \|D^2_{\rm NC}(v_{\rm M} -v)\|^2_{L^2(\Omega(T))}.
	\end{align*}
	Here ${\cal N}(T)$  denotes the set of vertices of  $T \in \T$  and patch 
	$\Omega(T):=\text { int } \left(\cup_{z \in {\cal N}(T)} \cup \T(z)\right)$, $\T(z)$ denotes the triangles that share the vertex $z$ and $\mathcal{E}(\Omega(T))$ denotes the edges in $\Omega(T)$,
	$\tau_E$ denotes the unit tangential vector to the edge $E$ and $[\phi]_E$ denotes the jump of a function $\phi$ across the edge $E$. \
\end{lem}
\noindent For vector-valued functions, the interpolation and companion operators are to be understood component-wise.
\begin{lem}[\it Bounds for $A_{\rm NC}(\bullet,\bullet)$]\cite[ Lemmas 4.2, 4.3]{ScbSungZhang}\label{Anc.bound} If $\bchi \in \bH^{2+\gamma}(\Omega)$,  $\Phi \in \bV\cap\bH^{2+\gamma}(\Omega)$ and $\Phi_{\rm M} \in \bV_{\rm M}$, then
$ (a) \; A_{\rm NC}(\bchi,(1-J)\Phi_{\rm M})\lesssim h^\gamma\trinl \bchi \trinr_{2+\gamma}\trinl \Phi_{\rm M} \trinr_{\rm NC}$ $
(b) \; A_{\rm NC}(\bchi,(1-I_{\rm M})\Phi)\lesssim h^{2\gamma}\trinl \bchi \trinr_{2+\gamma}\trinl \Phi \trinr_{2+\gamma} $
and $ (c) \; A_{\rm{NC}}(\Phi_{\rm{M}},\Phi_{\rm{M}}) =  \trinl\Phi_{\rm{M}}\trinr^{2}_{\rm{NC}}.$
\end{lem}
\begin{lem}[\it{Lower bounds for discrete norms}]{\cite[Lemma 4.7]{carstensen2017nonconforming}}\label{staest}
	For all $\Phi \in \bV+\bV_{\rm{M}}$, it holds that
	$(a) \; \trinl\Phi\trinr_{0,\infty}\lesssim \trinl\Phi\trinr_{\rm{NC}}
 \; and \;  (b) \; \trinl\Phi\trinr_{1,2,h}\lesssim\trinl\Phi\trinr_{\rm{NC}}.$
\end{lem}
\begin{lem}[\it{Bounds for $B_{\rm NC}(\bullet,\bullet,\bullet)$}]\label{Bnc bound}
	The boundedness  properties stated below hold:
	\begin{align*}
	(a)  \; &B_{\rm{NC}}(\bchi,\boldsymbol{\lambda},\Phi)\lesssim \trinl\bchi\trinr_{\rm{NC}}\trinl\boldsymbol{\lambda}\trinr_{\rm{NC}}\trinl\Phi\trinr_{0,\infty}\fl\bchi, \boldsymbol{\lambda}, \Phi\in \bV+\bV_{\rm{M}}.\\
	(b)  \; &B_{\rm{NC}}(\bchi,\boldsymbol{\lambda},\Phi)\lesssim \trinl\bchi\trinr_{\rm{NC}}\trinl\boldsymbol{\lambda}\trinr_{\rm{NC}}\trinl\Phi\trinr_{\rm{NC}}\fl\bchi, \boldsymbol{\lambda}, \Phi \in \bV+\bV_{\rm{M}}.\\
	(c) \; &B_{\rm{NC}}(\bchi,\boldsymbol{\lambda},\Phi)\lesssim \trinl\bchi\trinr_{2+\gamma}\trinl\boldsymbol{\lambda}\trinr_{\rm{NC}}\trinl\Phi\trinr_{0,4,h}\fl \bchi\in\bH^{2+\gamma}(\Omega), \boldsymbol{\lambda}, \Phi\in \bV + \bV_{\rm{M}}.\\
	(d) \; &B_{\rm{NC}}(\bchi,\boldsymbol{\lambda},\Phi)\lesssim \trinl\bchi\trinr_{2+\gamma}\trinl\boldsymbol{\lambda}\trinr_{\rm{NC}}\trinl\Phi\trinr_{1}\fl \bchi\in\bH^{2+\gamma}(\Omega), \boldsymbol{\lambda} \in \bV + \bV_{\rm{M}},  \Phi\in H^1_0(\O).	\\
	(e) \; &B_{\rm{NC}}(\bchi,\boldsymbol{\lambda},\Phi)\lesssim \trinl\bchi\trinr_{2+\gamma}\trinl\boldsymbol{\lambda}\trinr_{2+\gamma}\trinl\Phi\trinr_{0,2,h}\fl \bchi,  \boldsymbol{\lambda} \in\bH^{2+\gamma}(\Omega), \Phi\in \bV + \bV_{\rm{M}}.
	\end{align*}
\end{lem}
\begin{proof}
	The  first bound  follows from the definition of 
	$B_{\NC}(\bullet,\bullet, \bullet)$ and the generalised H\"older's inequality. The  bound in $(b)$ follows from  $(a)$ and Lemma \ref{staest}.$(a)$. For $\chi \in H^{2+\gamma}(\O), \lambda$ and $ \phi \in V+V_\M$, $(c)$ follows from the definition of 
	$B_{\NC}(\bullet,\bullet, \bullet)$, the estimate
	$ \sum_{T\in\mathcal{T}}\int_{T} [\chi,\lambda]\phi \dx\lesssim\|\chi\|_{2,4}\|\lambda\|_{\NC}\|\phi\|_{0,4,h},$ 
	and the continuous Sobolev imbedding $H^{2+\gamma}(\Omega) \hookrightarrow W^{2,4}(\Omega)$. The  bound in  $(d)$  follows from  $(c)$ and the continuous Sobolev embedding $H^1(\O) \hookrightarrow L^4(\O)$. 	
	The last bound follows from
	$ \sum_{T\in\mathcal{T}}\int_{T} [\chi,\lambda]\phi \dx\lesssim\|\chi\|_{2,4}\|\lambda\|_{2,4}\|\phi\|_{0,2,h},$ 
	and $H^{2+\gamma} (\Omega) \hookrightarrow W^{2,4}(\Omega)$ where $\chi, \lambda \in H^{2+\gamma}(\O)$ and $ \phi \in V+V_\M$.
\end{proof}
\begin{lem}\label{Bnc bound adjointH1error}
For  $\Psi,\bchi, \Theta \in  \bV \cap \bH^{2+\gamma}(\O)$ and $\Psi_{\rm M} \in \bV_{\rm M}$,  $$B_{\rm NC}(\Psi-\Psi_{\rm M},\bchi, \Theta) \lesssim \left(h^\gamma\trinl\Psi-\Psi_{\rm M}\trinr_{\rm NC}+\trinl\Psi-\Psi_{\rm M} \trinr\right)\trinl \bchi \trinr_{2+\gamma}\trinl \Theta \trinr_{2+\gamma}.$$
\end{lem}
\begin{proof}
Since the piecewise second derivatives of $I_\M\bchi$ are constants, the definition of $B_\NC(\bullet,\bullet,\bullet)$ and Lemma \ref{hctenrich}.$(c)$ show $B_{\NC}((J-1)\Psi_{\M},I_\M \bchi, \P_0\Theta)=0$. This and elementary algebra lead to
\begin{align}
& B_{\rm NC}(\Psi-\Psi_{\rm M},\bchi, \Theta)	=B(\Psi-J\Psi_{\M},\bchi, \Theta)+B_{\NC}((J-1)\Psi_{\M},\bchi, \Theta)\nonumber\\
&\quad =B(\Psi-J\Psi_{\M},\bchi, \Theta)+B_{\NC}((J-1)\Psi_\M,(1-I_\M) \bchi, \Theta)+B_{\NC}((J-1)\Psi_{\M},I_\M \bchi, (1-\P_0) \Theta).\label{bnc.adjoint}
\end{align}
The definition of $B(\bullet,\bullet,\bullet)$, the symmetry of $b(\bullet,\bullet,\bullet)$ in the first and third variables, Lemma \ref{Bnc bound}.$(e)$, triangle inequality with $\Psi_\M$ and Lemma \ref{hctenrich}.$(d)$ with $v=\Psi$ lead to 
$B(\Psi-J\Psi_{\M},\bchi, \Theta) \lesssim \left(h^2\trinl \Psi-\Psi_{\M} \trinr_\NC+\trinl \Psi-\Psi_{\M}\trinr\right)\trinl \bchi \trinr_{2+\gamma}\trinl \Theta \trinr_{2+\gamma}.$ Lemmas \ref{Bnc bound}.$(b)$, \ref{Morley_Interpolation}.$(c)$ and \ref{hctenrich}.$(d)$ with $v=\Psi$ result in $B_{\NC}((J-1)\Psi_{\M},(1-I_\M) \bchi , \Theta)\lesssim h^\gamma\trinl \bchi \trinr_{2+\gamma}\trinl \Theta \trinr_{2} \times\trinl \Psi-\Psi_{\M} \trinr_\NC.$ Lemmas \ref{Bnc bound}.$(a)$, \ref{Morley_Interpolation}.$(b)$, \ref{hctenrich}.$(d)$ with $v=\Psi$, {
{projection estimate in $L^\infty(\cT)$  \cite[Proposition 1.135]{ErnJLU_2004}}} and the global Sobolev embedding $H^{2+\gamma}(\Omega)\hookrightarrow W^{1,\infty}(\Omega)$ imply $B_{\NC}((J-1)\Psi_{\M},I_\M \bchi, (1-\P_0)\Theta) \lesssim h \trinl \bchi \trinr_{2}\trinl \Theta \trinr_{2+\gamma} \trinl \Psi - \Psi_{\M} \trinr_\NC.$ 
A substitution of the last three bounds  in \eqref{bnc.adjoint} concludes the proof.
\end{proof}

\section{A priori error estimates}\label{sec.apriori}
This section deals with the {\it a priori} error estimates for the state, adjoint and control variables under minimal regularity assumptions on the exact solution. The proof of the results that differ from  \cite{SCNNDS} owing of the choice of the alternate discrete trilinear form are discussed in Appendix. 
\medskip

\noindent For a given ${\bf F}$, fixed control $u\in U_{ad}$ and ${\bf u}=(u,0)$,   consider the auxiliary state equation that seeks  $\Psi_u \in \bV$ such that 
\begin{align}\label{wv1}
{A}(\Psi_u,\Phi) +{B}(\Psi_u, \Psi_u,\Phi) = ( {\bf F+{\bf C} {\bf u}},\Phi )  \mbox{ for all } \Phi \in \bV.
\end{align}
\noindent The nonconforming Morley finite element (FE) approximation  to (\ref{wv1}) seeks $\Psi_{u,\M} \in \bV_\M$ such that 
\begin{align} 
&  A_{\NC}(\Psi_{u,\M},\Phi_\M)+B_{\NC}(\Psi_{u,\M},\Psi_{u,\M},\Phi_\M)=({\bf F} + {\mathbf C}{\bf u}, \Phi_\M) \; 
\mbox{ for all } \Phi_\M \in \bV_\M. \label{wform2}
\end{align}
%
%
 \noindent
  The next result on the existence, uniqueness and error estimates of the auxiliary state equation is proved with the help of Lemma~\ref{aux} given in the Appendix.  The proofs that are available in \cite{SCNNDS, CN2020} are skipped. Note that a modified proof of Lemma~\ref{aux} is presented and it utilises the properties of the companion operator to obtain sharper bounds  in  comparison to \cite[Lemma 3.12]{SCNNDS}.
\begin{thm}[\it{Existence, uniqueness and error estimates}] \label{ee1}\noindent
	\begin{itemize}
		\item [(i)]Let $({\bar \Psi}, {\bar u}) \in \bV \times L^2(\omega) $ be a regular solution to \eqref{wform}.
		Then, there exist $\rho_1, \rho_2>0$ such that, for a sufficiently small choice of the discretization parameter and ${u} \in B_{\rho_2}({\bar u})$, \eqref{wform2} admits a unique solution in 
		$B_{\rho_1}(\bar \Psi)$, where ${u} \in B_{\rho_2}({\bar u})$ (resp. ${\Psi} \in B_{\rho_1}({\bar \Psi})$ ) implies  $\|u - {\bar u}\|_{L^2(\omega)} \le \rho_2$ (resp.  $\trinl \Psi - {\bar \Psi} \trinr_{\rm{NC}} \le \rho_1$).
		\item[(ii)]	Let $({\bar \Psi}, {\bar u}) \in \bV \times L^2(\omega)$ be a regular solution to \eqref{wform}.
		Then, for ${u} \in B_{\rho_2}({\bar u})$ and a sufficiently small choice of the discretization parameter, the solutions $\Psi_{u}$ and $\Psi_{u,\rm {M}}$ to  \eqref{wv1} and \eqref{wform2} satisfy the energy and broken $H^1$ norm estimates 
	$	(a) \; \trinl\Psi_{u}-\Psi_{u,{\rm M}}\trinr_{\rm {NC}} \lesssim \|(1 - I_{\rm M})\Psi_u\|_{\rm NC}+{\rm{osc_1}}(f+\mathcal{C}u+[\psi_{u,1},\psi_{u,2}],\T) +{\rm{osc_1}}([\psi_{u,1},\psi_{u,1}],\T),$ 
	$ (b) \;  \; \trinl\Psi_{u}-\Psi_{u,{\rm M}}\trinr_{1,2,h} \lesssim   h^{\gamma}(\trinl\Psi_{u}-\Psi_{u,{\rm M}}\trinr_{\rm {NC}} +{\rm{osc}}_m(f+\mathcal{C}u,\T)) \mbox{ for each } m \in \N_0.$
		\item[(iii)] 	For $u, \:{\hat u} \in B_{\rho_2}({\bar u}),$ and a sufficiently small choice of the discretization parameter,   the solutions $\Psi_{u}$ and $\Psi_{{\hat u},{\rm M}}$ to \eqref{wv1} and \eqref{wform2}, with controls chosen as $u$ and ${\hat u}$ respectively, satisfy
		$\trinl \Psi_{u}- \Psi_{{\hat u},{\rm M}}\trinr_{\rm NC} \lesssim  h^{\gamma} + \|u- {\hat u}\|_{L^2(\omega)}.$
	\end{itemize}
	Here $\gamma \in (1/2,1]$ is the elliptic regularity index. \qed
\end{thm}
\noindent The proof of $(i)$ and $(iii)$ can be found in \cite[Theorem 3.8 $(i)$ and Lemma 3.9]{SCNNDS}. The error estimate in energy and piecewise $H^1$ norms given by $(ii)(a)$-$(b)$ are established in \cite[Theorem 3.1]{CN2020}.
\begin{rem}\label{rem.state}
 The well-known result for the biharmonic problem for the approximation using Morley nonconforming FEM which states that the $L^2$ error estimate cannot be further improved than that of $H^1$ error estimate \cite{HuShi} extends to \vket and thus Theorem \ref{ee1}.$(ii)$, and Lemma \ref{Morley_Interpolation}.$(c)$ show 
	$	(a) \; \trinl\Psi_{u}-\Psi_{u,{\rm M}}\trinr_{\rm {NC}} \lesssim  h^\gamma, $ \\$\; (b) \;  \; \trinl\Psi_{u}-\Psi_{u,{\rm M}}\trinr_{1,2,h} \lesssim   h^{2\gamma}$, and $ (c) \;  \; \trinl\Psi_{u}-\Psi_{u,{\rm M}}\trinr\lesssim   h^{2\gamma}.$ \qed
\end{rem}
\noindent The auxiliary problem corresponding to the adjoint equations seeks $\Theta_{u}\in \bV$ such that
\begin{align} \label{auxiadje_12}
{A}(\Phi, \Theta_u) +2{B}(\Psi_{u}, \Phi, \Theta_u)=(\Psi_{u}-\Psi_{d},\Phi) \mbox{ for all } \Phi\in \bV,\end{align}
where  $\Psi_{u} \in \bV$ is the solution to \eqref{wv1}. A Morley FE discretization corresponding to  \eqref{auxiadje_12} seeks $\Theta_{u,\M}\in \bV_\M$ such that, for all $\Phi_{\M}\in \bV_{\M}$, 
\begin{align} \label{auxiadje_1_dis}
{A}_{\NC}(\Phi_\M, \Theta_{u,\M})+2{B}_{\NC}(\Psi_{u,\M},\Phi_\M,  \Theta_{u,\M})=(\Psi_{u,\M}- \Psi_d, \Phi_\M).
\end{align}
%

\noindent The existence, uniqueness and convergence results stated in the next theorem follow analogous to that of \cite[Theorems 4.1, 4.2 $(a)$]{SCNNDS} and is skipped for brevity.
\begin{thm}[{\it Existence, uniqueness and energy error estimate}] \label{aux_dis1} 
Let $({\bar \Psi}, {\bar u}) \in \bV \times L^2(\omega)$ be a regular solution to \eqref{wform}. Then, (i)  there exist $0 < \rho_3 \le \rho_2$  such that, for a sufficiently small choice of discretization parameter  and ${u} \in B_{\rho_3}({\bar u})$, \eqref{auxiadje_1_dis} admits a unique solution, (ii)  for $u\in B_{\rho_3}(\bar{u})$ and  a sufficiently small choice of discretization parameter, the solutions $\Theta_{u}$ and $\Theta_{u,{\rm M}}$ of \eqref{auxiadje_12} and \eqref{auxiadje_1_dis} satisfy the  energy norm error estimate:	
$\,\trinl\Theta_{u}-\Theta_{u,{\rm M}}\trinr_{\rm NC}\lesssim \trinl \Psi_u-\Psi_{u, \rm M}  \trinr_{\rm NC}+h^\gamma (\| \psi_{u,1} - \psi_{d,1}-[\psi_{u,1}, \theta_{u,2}]+ [ \psi_{u,2},\theta_{u,1}]\|
		+\| \psi_{u,2}-\psi_{d,2}+[\psi_{u,1}, \theta_{u,1}]\|)$,
		where $\Psi_u$ (resp. $\Psi_{u,\rm M}$) solves \eqref{wv1} (resp. \eqref{wform2}) and $\gamma\in(\frac{1}{2},1]$ is the index of the elliptic regularity.		
\end{thm}


\noindent The proof of {\it a priori} $H^1$ error estimate stated in the next theorem for adjoint variables is a non-trivial modification of the corresponding result in \cite{SCNNDS}  and is presented in the Appendix.  The form of the error estimate will be useful in the adaptive convergence study that is planned for future.
\begin{thm}[{\it piecewise $H^1$ error estimate}]\label{thm.adj.energy}
	Let $(\bar{\Psi},\bar{u}$) $\in \bV\times L^2(\omega)$ be a regular solution to \eqref{wform}. Then for a sufficiently small choice of the discretization parameter, the solutions $\Theta_{u}$ and $\Theta_{u,{\rm M}}$ of \eqref{auxiadje_12} and \eqref{auxiadje_1_dis} satisfy 
\begin{align*}
&\,\trinl \Theta_{u}- \Theta_{u,{\rm M}}\trinr_{1,2,h} 
	\lesssim h^\gamma\left(\trinl \Psi_u-\Psi_{u,{\rm M}}\trinr_{\rm NC}+\trinl\Theta_u-\Theta_{u,{\rm M}} \trinr_{\rm NC}+\trinl(1-\P_0)\Theta_u\trinr_{0,\infty}\right) \\
		&\quad \hspace{3cm} +\trinl\Psi_u-\Psi_{u,{\rm M}}\trinr_{\rm NC} \trinl\Theta_u-\Theta_{u,{\rm M}}\trinr_{\rm NC} 
+ \trinl\Psi_u-\Psi_{u,{\rm M}} \trinr +h^{2+\gamma}{\rm{osc}}_0(\Psi_{u}-\Psi_{d}),
	\end{align*}
	where  $\Psi_u$ (resp. $\Psi_{u,\rm M}$) solves \eqref{wv1} (resp. \eqref{wform2}) and $\gamma\in(\frac{1}{2},1]$ is the index of the elliptic regularity.
\end{thm}
 \begin{rem}\label{rem.adjoint}
 The {{projection estimate in $L^\infty(\cT)$ \cite[Proposition 1.135]{ErnJLU_2004}}} and global Sobolev embedding $H^{2+\gamma}(\Omega)\hookrightarrow W^{1,\infty}(\Omega)$ imply $\trinl(1-\P_0)\Theta_u\trinr_{0,\infty} \lesssim h \trinl \Theta_u \trinr_{2+\gamma}$. This, Theorems \ref{aux_dis1} and \ref{thm.adj.energy}, and Remark \ref{rem.state} show that\\
		$	(a) \; \trinl\Theta_{u}-\Theta_{u,{\rm M}}\trinr_{\rm {NC}} \lesssim  h^\gamma, \; (b) \;  \; \trinl\Theta_{u}-\Theta_{u,{\rm M}}\trinr_{1,2,h} \lesssim   h^{2\gamma}, \;\text{ and }  (c) \;  \; \trinl\Theta_{u}-\Theta_{u,{\rm M}}\trinr\lesssim   h^{2\gamma}.$ \qed	
\end{rem}
\medskip


\noindent
 For the error estimates of nonlinear control problem, second order sufficient optimality conditions are employed. For a detailed discussion, we refer to \cite[Section 2.3]{ngr} and \cite[Section 3.2]{cmj}.
	\begin{thm}[{\it A priori error estimates}]\cite[Theorem 5.1]{SCNNDS} \label{conv.apriori}
	Let  $({\bar \Psi}, {\bar u})$ be a regular solution to \eqref{wform}  and  $\{({\bar \Psi}_{\rm{M}}, {\bar u}_h)  \}_{h \le
		h_1}$  be a solution to  \eqref{discrete_cost} converging
	to $({\bar \Psi}, {\bar u})$ in $\bV \times L^2(\omega)$, for a sufficiently small mesh-size $h$ with ${\bar u}_h \in B_{\rho_3}(\bar u)$ as in Theorem \ref{aux_dis1}. Let $\bar \Theta $ and ${\bar \Theta}_{\rm{M}}$ be the corresponding continuous and discrete adjoint  variables, respectively.  Then, for a sufficiently small choice of the discretization parameter, it holds
$(a)\; \trinl \bar \Psi- {\bar \Psi}_{\rm M} \trinr_{\rm NC} \lesssim h^{\gamma}$, $ (b) \; \trinl \bar \Theta- {\bar \Theta}_{\rm M} \trinr_{\rm NC} \lesssim h^{\gamma}, \text{ and }  	(c) \; \|\bar{{u}}-\bar{{u}}_h\|_{L^2(\omega)}\lesssim  h,$
	$\gamma \in (1/2,1]$ being the index of elliptic regularity.
\end{thm}
\section{Reliability Analysis}\label{sec.reliability}
This section deals with the reliability analysis for the {\it a posteriori} error estimator for the optimal control problem \eqref{wform}. Let $\bT$ be the set of all admissible triangulations $\cT$. Given any $0 <\delta<1$, let $\bT(\delta)$ be the set of all triangulations $\cT$ with mesh-size $\le \delta$ for all triangles $T \in \cT$ with area $|T|$. Assume that $\omega \subset \O$ is a polygonal domain and that $\T$ restricted to $\omega$ yields a triangulation for $\omega$.

\noindent The main result of this section is stated first in Theorem \ref{thm.global.reliability}. The proof is presented at the end of this section. Define the auxiliary variable $\widetilde{u}_h$ by
\begin{align}\label{defn_post_proc}
\widetilde{{u}}_h:=\Pi_{[u_{a},u_{b}]}\big(-\frac{1}{\alpha}(\mathcal{C}^*\bar{\theta}_{\M,1})\big),
\end{align}
where $\bar{\Theta}_{\M}=(\bar{\theta}_{\M,1}, \bar{\theta}_{\M,2})$ is the discrete adjoint variable corresponding to the control $\bar{u}_h$.

\noindent For $K \in \cT$ and $E \in \cE$, define
\begin{subequations}\label{def.estimators}
\begin{align} 
&\eta^2_{K,\bar\Psi_{\rm{M}}}:= h_K^4\left(\| f+\mathcal{C}\bar u_h+[\bar \psi_{\M,1},\bar \psi_{\M,2}]\|_{L^2(K)}^2+\|[\bar \psi_{\M,1}, \bar \psi_{\M,1}]\|_{L^2(K)}^2\right), \quad \eta^2_{K,\bar u_h}:=\| \widetilde{u}_h-\bar u_h \|^2_{L^2(K)}, \label{def.est.statecontrol.volume}
\\
&\eta^2_{K,\rm{res},\bTheta_{\rm{M}}}:=h_K^4\left(\|\bar \psi_{\M,1}-\psi_{d,1} -[\bar \psi_{\M,1}, \bar \theta_{\M,2}]+ [\bar \psi_{\M,2},\bar \theta_{\M,1}]\|_{L^2(K)}^2
+\|\bar \psi_{\M,2}-\psi_{d,2}+[\bar \psi_{\M,1}, \bar \theta_{\M,1}]\|_{L^2(K)}^2\right),\label{def.est.adjoint.volume1}\\
& \eta^2_{K,\P_0,\bTheta_{\rm{M}}}:= \| D^2\bpsi_{\M,1}(1-\P_0)\btheta_{\M,2}\|^2_{{L^{2}(K)}} +\| D^2\bpsi_{\M,2}(1-\P_0)\btheta_{\M,1}\|^2_{{L^{2}(K)}}+ \| D^2\bpsi_{\M,1}(1-\P_0)\btheta_{\M,1}\|^2_{{L^{2}(K)}},\label{def.est.adjoint.volume2}\\
&\eta^2_{K,\bTheta_{\rm{M}}}:= \eta^2_{K,\rm{res},\bTheta_{\rm{M}}} + \eta^2_{K,\P_0,\bTheta_{\rm{M}}}, \label{def.est.adjoint.volume} \\
&\eta^2_{E,\bPsi_{\rm{M}}}:= h_E\left(\|[D^2 \bar \psi_{\M,1}\tau_E]_E\|_{L^2(E)}^2+\|[D^2 \bar \psi_{\M,2}\tau_E]_E\|_{L^2(E)}^2\right), \text{ and } \label{def.est.state.edge}\\
&\eta^2_{E,\bTheta_{\rm{M}}}:= h_E\left(\|[D^2 \bar \theta_{\M,1}\tau_E]_E\|_{L^2(E)}^2+\|[D^2 \bar \theta_{\M,2}\tau_E]_E\|_{L^2(E)}^2\right).\label{def.est.adjoint.edge}
\end{align}
\end{subequations}

\begin{thm}[{\bf Reliability for the control problem}]
	\label{thm.global.reliability}
	Let $(\bPsi,\bTheta,\bar{\bf u})$ (resp. $(\bPsi_{\rm{M}},\bTheta_{\rm{M}},\bar{\bf u}_h)$) solve the optimality system \eqref{opt_con} (resp.  \eqref{discrete.opt}). Then for a sufficiently small choice of the mesh size $h$, there exists an $h-$independent positive constant $C_{\rm rel}$ such that
	\begin{align}\label{global.reliability}
	\tnr{\bPsi-\bPsi_{\rm{M}}}_{\rm NC}^2+	\tnr{\bTheta-\bTheta_{\rm{M}}}_{\rm NC}^2&+\|\bar{u} - {\bar{u}}_h\|_{L^2(\omega)}^2
	\le C_{\rm rel}^2 \eta^2, \mbox{ where } \eta^2:=\eta^2_{\rm ST}+\eta^2_{\rm AD}+\eta^2_{\rm CON}, \text{ with }
	\end{align}
	$$
	\eta^2_{\rm ST}=	\sum_{K\in\cT }\eta^2_{K,\bPsi_{\rm{M}}}+\sum_{E\in\cE }\eta^2_{E,\bPsi_{\rm{M}}},~
	\eta^2_{\rm AD}:=\sum_{K\in\cT }\eta^2_{K,\bTheta_{\rm{M}}}+\sum_{E\in\cE }\eta^2_{E,\bTheta_{\rm{M}}},
~\eta^2_{\rm CON}:=\sum_{K\in\cT }\eta^2_{K,\bar u_h}.
$$
\end{thm}
\subsection{A posteriori error analysis for the state equations}
 Let $(\bPsi,\bar u)$ be a regular solution to \eqref{wform} and $\hPsi \in \bV$ solves the auxiliary state equation 
 \begin{align}\label{state.aux}
 {A}(\hPsi,\Phi) +{B}(\hPsi, \hPsi,\Phi) = ( {\bf F}+{\bf C} \bar{\bf u}_h,\Phi )  \mbox{ for all } \Phi \in \bV,
 \end{align}
where $\bar{\bf u}_h=(\bar{ u}_h,0)^T$ is the discrete control in \eqref{discrete.opt}. Since $\bar \Psi$ is a regular solution, for a sufficiently small choice of the mesh size $h$,  $\bar u_h \in {\mathcal O}(\bar u)$ from  Theorem \ref{conv.apriori}.$(c)$ and hence
Theorem \ref{th2.5} yields $\hPsi$ is regular.  That is, 
\begin{align}\label{inf-sup_apost}
0<\hat{{\beta}}:=\inf_{\substack{\bxi\in \bV\\ \trinl\bxi\trinr_2=1}}\sup_{\substack{\Phi\in \bV\\ \trinl\Phi\trinr_2=1}}DN(\hPsi;\bxi,\Phi). 
\end{align}
Note that $\hPsi$ solves the \vket \eqref{state.aux} and its Morley FE approximation seeks $\bPsi_\M$ given by \eqref{state_eq1}. Let $C_b$ denotes the boundedness constant of $B_\NC(\bullet,\bullet,\bullet)$ absorbed in $`\lesssim'$ of Lemma \ref{Bnc bound}.$(b)$. Suppose $\varepsilon, \delta >0$  are chosen smaller such that, for any $\cT \in \bT(\delta)$, exactly
one discrete solution $(\bPsi_{\rm{M}},\bTheta_{\rm{M}},\bar{\bf u}_h)$ solve the optimality system \eqref{discrete.opt} such that Remark~\ref{rem.state}.$(a)$ and Theorem \ref{conv.apriori}.$(a)-(b)$  hold with $\varepsilon \le \min\{ {\hat{\beta}}/({2C_b(1+\Lambda_\jc))}, \beta/(4C_b)\}$ and 
\begin{align}\label{varepsilon}
\|\hPsi - \bPsi_{ \M}\|_{\NC}+\|\bPsi - \bPsi_{ \M}\|_{\NC}+\|\bTheta - \bTheta_{ \M}\|_{\NC} \le \varepsilon,
\end{align} 
where $\bTheta$ solves \eqref{adj_eq} and $\beta$ (resp. $\hat{\beta}$) is the inf-sup constant in \eqref{inf-sup_apost_psi} (resp. \eqref{inf-sup_apost}).

 \begin{thm}[Reliability for the state variable]\label{thm.state.reliability}
 Let $(\bar{\Psi},\bar{u})$ $\in \bV\times L^2(\omega)$ be a regular solution to \eqref{wform} and $({\bar \Psi}_{\rm{M}},   {\bar u}_h)$ $\in \bV_{\rm M} \times U_{h,ad}$ solves \eqref{discrete_cost}. 
 Then for a sufficiently small choice of mesh-size $h$, there exists an $h-$independent positive constant $C_{\rm ST, rel}$ such that
  \begin{equation}\label{state.reliability}
 	 \tnr{\bar{\Psi}-\bPsi_{\rm{M}}}_{\rm{NC}}^2\le C_{\rm ST, rel}^2 \bigg(\sum_{K\in\cT }\eta^2_{K,\bPsi_{\rm{M}}}+\sum_{E\in\cE }\eta^2_{E,\bPsi_{\rm{M}}}+\|\bar u - \bar {u}_h \|_{L^2(\omega)}^2\bigg).
 	 \end{equation}
 \end{thm}
 \begin{proof} The proof adapts the ideas of  \cite{carstensen2017nonconforming} for the control problem. The terms $\tnr{\widehat{\Psi}-\bPsi_\M}_{\rm NC}$ and $\tnr{\bar{\Psi}-\hPsi}_{2}$ are estimated and then a triangle inequality completes the proof.
The inf-sup condition \eqref{inf-sup_apost} implies that for any  $0<\epsilon_1<\hat{\beta}$ there exists some $\Phi\in \bV$ with $\trinl\Phi\trinr_2=1$ and
\begin{align}\label{inf_sup_epsilon}
(\hat{\beta}-\epsilon_1)\trinl \hPsi-J\bPsi_\M\trinr_2\leq DN(\hPsi; \hPsi-J\bPsi_\M,\Phi).
\end{align}
Since $N(\bullet)$ is quadratic,  the finite Taylor series is exact  and hence 
\begin{align*}
N( J\bPsi_\M;\Phi)=DN(\hPsi;J\bPsi_\M-\hPsi,\Phi)+\half D^2N(\hPsi; J\bPsi_\M-\hPsi, J\bPsi_\M-\hPsi,\Phi).
\end{align*}
This with $D^2N(\hPsi;\hPsi-J\bPsi_\M,\hPsi-J\bPsi_\M,\Phi)=2B(\hPsi-J\bPsi_\M,\hPsi-J\bPsi_\M,\Phi)$,  \eqref{inf_sup_epsilon} and  Lemma \ref{Bnc bound}.$(b)$ show
\begin{align}
(\hat{\beta}-\epsilon_1)\trinl \hPsi-J\bPsi_\M\trinr_2&
\leq |N(J\bPsi_\M;\Phi)|+C_b\trinl \hPsi-J\bPsi_\M\trinr_2^2.\label{est.state}
\end{align}
 A triangle inequality, \eqref{varepsilon}, Lemma \ref{hctenrich}.$(d)$ with $v=\hPsi$ and $\varepsilon \le \hat{\beta}/(2C_b(1+\Lambda_\jc))$ imply 
\begin{equation}\label{est.1}
\trinl \hPsi-J\bPsi_\M \trinr_{\NC}\leq \trinl \hPsi-\bPsi_\M \trinr_\NC+\trinl (1-J)\bPsi_\M\trinr_{\NC} \le (1+\Lambda_\jc)\varepsilon\le {
\hat{\beta}}/{2C_b}.
\end{equation}
With $\epsilon_1\searrow 0$, \eqref{est.state} and \eqref{est.1} result in
$
\frac{\hat{\beta}}{2}\trinl \hPsi-J\bPsi_\M\trinr_2\leq |N(J\bPsi_\M;\Phi)|.
$
This eventually shows that
\begin{align}\label{app_thm}
\trinl\hPsi-\bPsi_{\M}\trinr_{\NC}\le 2 \hat{\beta}^{-1} |N(J\bPsi_{\M};\Phi)|+\trinl (J-1)\bPsi_{\M}\trinr_{\NC}.
\end{align}
%
The definition of $N(\bullet)$, \eqref{state_eq1}  and rearrangements lead to 
\begin{align}\label{s1.6} & N(J\bPsi_{\M};\Phi)=A(J\bPsi_{\M},\Phi)+B(J\bPsi_{\M},J\bPsi_{\M},\Phi)-({\bf F} +{\mathbf C}{\bf u}_h,\Phi)\notag\\
&\quad =A_{\NC}((J-1)\bPsi_{\M},\Phi)+A_\NC(\bPsi_{\M},(1-I_{\M})\Phi)+B_\NC((J-1)\bPsi_{\M},J\bPsi_{\M},\Phi)\notag\\
&\; \; \quad+B_{\NC}(\bPsi_{\M},(J-1)\bPsi_{\M},\Phi) +B_\NC(\bPsi_{\M},\bPsi_{\M}, (1-I_{\M})\Phi)-({\bf F} +{\mathbf C}{\bf u}_h, (1-I_{\M})\Phi)
=: \sum_{i=1}^{6}S_i.
\end{align}
A Cauchy-Schwarz inequality proves $S_1 \le\trinl (J-1)\bPsi_{\M}\trinr_{\NC}.$ Since the piece-wise second derivatives of $\bPsi_\M$ are constants, Lemma \ref{Morley_Interpolation}.$(a)$ implies $S_2=0$. The triangle inequalities, Lemma \ref{hctenrich}.$(d)$ with $v=\hPsi$, \eqref{varepsilon} and Lemma \ref{ap}.$(a)$ prove
\begin{align}\label{dissta}
\|J \bPsi_\M\|_2 + \|\bPsi_\M\|_{\NC} &\le \|(J-1) \bPsi_\M\|_{\NC}+ 2(\|\hPsi-\bPsi_\M \|_{\NC} +\|\hPsi\|_2) 
\le (2+\Lambda_\jc)\varepsilon+2\|\bPsi\|_2 :=M_1.
\end{align}
  Lemma \ref{Bnc bound}.$(b)$ and \eqref{dissta} show $S_3+S_4 \le C_bM_1 \trinl (J-1)\bPsi_{\M}\trinr_{\NC}$. The definition of $B_\NC(\bullet,\bullet,\bullet)$, a Cauchy-Schwarz inequality and Lemma \ref{Morley_Interpolation}.$(b)$ prove
$S_5+S_6 \le C_{I} \sum_{K\in\cT }h_K^2\left(\| f+\mathcal{C}u_h+[\bar \psi_{\M,1},\bar \psi_{\M,2}]\|_{L^2(K)}+\|[\bar \psi_{\M,1}, \bar \psi_{\M,1}]\|_{L^2(K)}\right).$ A substitution of $S_1$-$S_6$ in \eqref{s1.6} and then in \eqref{app_thm} with Lemma \ref{hctenrich}.$(e)$, the definitions \eqref{def.est.statecontrol.volume} and \eqref{def.est.state.edge} result in
	\begin{align}\label{staterel_estimator.t1}
	\trinl\hPsi-\bPsi_{\M}\trinr_{\NC}^2& \leq \tilde{C}_{\rm ST, rel}^2\big(\sum_{K\in\cT }\eta^2_{K,\bPsi_{\rm{M}}}+\sum_{E\in\cE }\eta^2_{E,\bPsi_{\rm{M}}}\big),
	\end{align}
with the constant $\tilde{C}_{\rm ST, rel}^2:=C_J^2\left(1+2\hat{\beta}^{-1}(1+C_bM_1)\right)^2+4\hat{\beta}^{-2}C_I^2$. Theorem \ref{th2.5} for \eqref{state_eq} and \eqref{state.aux} yield $G(\bar u)=\bPsi$, $G(\bar u_h)=\hPsi$. Also, if $ G'({u}){v} =: \mathbf{z}_{{v}} \in
 \bV$, then $\mathbf{z}_{v}$ satisfies
${\mathcal A}\mathbf{z}_{{v}} +
 {\mathcal B}'(\Psi)\mathbf{z}_{{v}} =
 {\mathbf C}\mathbf{v}\quad \mathrm{in\ } \bV',$
 where $\Psi=G(u)$ and $u$, $v$ belong to the interior of $ {\mathcal O}({\bar u})$. 
Theorem \ref{th2.5} proves the uniform boundedness of $\trinl (\mathcal{A}+\mathcal{B}'(\Psi_{u}))^{-1}\trinr_{\mathcal{L}(\bV',\bV)}$ whenever $u\in {\mathcal O}({\bar u})$.
 Hence, for $u_t= {\bar u_h} + t(\bar u- \bar{u}_h)$ and $\Psi_{t}=G(u_{t})$, mean value theorem, Theorem \ref{th2.5} 
 and  $\bar u_h \in {\mathcal O}({\bar u})$ show
 \begin{align*} 
 & \trinl \bPsi- \hPsi \trinr_2  
 =\trinl { \int_{0}^1} G'(u_t)({\mathbf C}(\bar{ \bf u}-\bar{\bf {u}}_h)))  \: {\rm dt} \trinr_{2} 
 =  \trinl { \int_{0}^1} (\mathcal{A}+\mathcal{B}'(\Psi_{t}))^{-1}({\mathbf C}(\bar {\bf u}-\bar {\bf {u}}_h))  \: {\rm dt} \trinr_{2} \lesssim \|\bar u - \bar {u}_h \|_{L^2(\omega)}.
 \end{align*}
 \noindent A combination of \eqref{staterel_estimator.t1} and the last displayed result  with a triangle inequality concludes the proof.
\end{proof}
\subsection{A posteriori error analysis for the adjoint equations}	\label{sec:apost_adj}
\noindent The auxiliary problem that corresponds to the adjoint equations seeks $\hTheta\in \bV$ such that
\begin{align} \label{auxiadje_1}
{A}(\Phi, \hTheta) +2{B}_\NC(\bar\Psi_\M, \Phi, \hTheta)=(\bar \Psi_{\M}-\Psi_{d},\Phi) \mbox{ for all } \Phi\in \bV,
\end{align}
where  $\bar\Psi_\M \in \bV_\M$ is the solution to \eqref{state_eq1}. 
%
 Since $\bPsi$ is a regular solution to \eqref{wform}, the adjoint of the operator in \eqref{inf-sup_apost_psi} satisfies the inf-sup condition given by 
\begin{equation}\label{inf_suP_adjoint}
	{\beta}=\inf_{\substack{\bxi\in \bV\\ \trinl\bxi\trinr_2=1}}\sup_{\substack{\Phi\in \bV\\ \trinl\Phi\trinr_2=1}}\langle  {\mathcal A}{\Phi}+{\mathcal B}'(\bPsi){\Phi},{\boldsymbol \xi}\rangle, \quad \trinl \bTheta \trinr_2 \le {\beta}^{-1}\trinl \bPsi-\Psi_d\trinr
\end{equation}
with the last inequality derived from \eqref{adj_eq}. An introduction of $\bPsi$, the first inequality of \eqref{inf_suP_adjoint}, Lemma \ref{Bnc bound}.$(b)$, \eqref{varepsilon} and $\varepsilon \le \beta/(4C_b)$ show that for any  $0 <\epsilon_2<{\beta}$, there exists some $\Phi\in \bV$ with $\trinl\Phi\trinr_2=1$ such that
\begin{align}
{A}(\Phi, \hTheta) +&2{B}_\NC(\bar\Psi_\M, \Phi, \hTheta)
={A}(\Phi, \hTheta) +2{B}(\bPsi, \Phi, \hTheta)+2{B}_\NC(\bar\Psi_\M-\bPsi, \Phi, \hTheta) \nonumber \\
& \ge  ({\beta}-\epsilon_2-2C_b\trinl \bar\Psi_\M-\bPsi\trinr_\NC) \trinl\hTheta\trinr_2\ge ({\beta}-2C_b\varepsilon)\trinl\hTheta\trinr_2 \ge \frac{\beta}{2}\trinl\hTheta\trinr_2\label{thetahat.infsup}
\end{align}
with $\epsilon_2\searrow 0$ in the second last step of the inequality above. 
This shows the wellposedness of \eqref{auxiadje_1}. A combination of \eqref{auxiadje_1} and \eqref{thetahat.infsup} leads to a bound for the solution of $\hTheta $ of \eqref{auxiadje_1} as
\begin{equation}\label{apriori.discreteadj1}
\trinl \hTheta \trinr_{2} \le 2{\beta}^{-1}\trinl  \bPsi_{\rm{M}} -\Psi_d\trinr_{}.
\end{equation}
\medskip
\noindent 
For $\Psi, \Phi \in \bV + \bV_\M$, define {\it linear operators}  $\mathcal{F}_{\Psi}$ and $\mathcal{F}_{\Psi,\NC}\in \mathcal{L}(\bV+\bV_{\M})$  by
\begin{align}\label{adjlinop}
\mathcal{F}_{\Psi}(\Phi) = \Phi+ T[{\mathcal B}_{\NC}'(\Psi)^* (\Phi)] \text{ and } \mathcal{F}_{\Psi,\NC}(\Phi) = \Phi+T_\NC[{\mathcal B}'_{\NC}(\Psi)^* (\Phi)],
\end{align}
where ${\mathcal B}'_\NC(\Psi)^*$ is the adjoint operator corresponding to  ${\mathcal B}'_{\NC}(\Psi)$ and the bounded linear operator $T(\bullet)$ (resp. $T_\NC(\bullet)$) solves the biharmonic system of equations in the sense that for the load $\bg \in \bV'$ (resp. $\bg \in \bV_\M'$), $A(T\bg,\Phi)=\langle\bg,\Phi\rangle$ for all $\Phi \in \bV$ (resp. $A_\NC(T_\NC \bg,\Phi_\M)=\langle\bg,\Phi_\M\rangle$ for all $\Phi_\M \in \bV_\M$). A detailed discussion of these operators is provided in Appendix.

\noindent  The next lemma  (proved in Appendix) is utilized in the proof of Theorem \ref{thm.adjoint.reliability}. 
\begin{lem}[\it Uniform boundeness of $\mathcal{F}^{-1}_{\Psi_u}$]\label{automorphism}  
	If  $\bar \Psi \in \bV $ is a regular solution to \eqref{wform}, then $\mathcal{F}_{\Psi_u}$ is an automorphism on $\bV+\bV_{\rm M}$,  whenever $u$ is sufficiently close to $\bar u$. Moreover,  $\|\mathcal{F}_{\Psi_u}^{-1}\|_{\mathcal{L}(\bV+\bV_{\rm M})}\lesssim 1+\|(\mathcal{A}+\mathcal{B}
		'(\bar{\Psi}))^{-1}\|_{\mathcal{L}(\bV',\bV)}\trinl\bar\Psi\trinr_{2}$. 
\end{lem}
 \begin{thm}[Reliability for the adjoint variable]\label{thm.adjoint.reliability}
Let $(\bPsi,\bTheta,\bar{\bf u})$ (resp. $(\bPsi_{\rm{M}},\bTheta_{\rm{M}},\bar{\bf u}_h)$) solve the optimality system \eqref{opt_con} (resp. \eqref{discrete.opt}). Then for sufficiently small mesh-size $h$, there exists an $h-$independent positive constant $C_{\rm AD, rel}$ such that
 	\begin{equation}\label{adjoint.reliability}
 	\tnr{\bTheta-\bTheta_{\rm{M}}}_{\rm NC}^2\le C_{\rm AD, rel}^2\bigg(\sum_{K\in\cT }\eta^2_{K,\bPsi_{\rm{M}}}+ \sum_{K\in\cT }\eta^2_{K,\bTheta_{\rm{M}}}+\sum_{E\in\cE }\eta^2_{E,\bPsi_{\rm{M}}}+\sum_{E\in\cE }\eta^2_{E,\bTheta_{\rm{M}}}+\|\bar u - \bar {u}_h \|_{L^2(\omega)}^2\bigg).
 	\end{equation}
 	\end{thm}
 \begin{proof}
 The terms $\tnr{\widehat{\Theta}-\bTheta_\M}_{\rm NC}$ and $\tnr{\bar{\Theta}-\hTheta}_{2}$ are estimated and then a triangle inequality completes the proof. The inf-sup condition \eqref{inf_suP_adjoint} implies for any $0<\epsilon_3<{\beta}$ there exists some  $\Phi \in \bV$ with $\trinl \Phi \trinr_2 =1$ and
$$ ({\beta}-\epsilon_3)\trinl \hTheta-J \bTheta_\M\trinr_2 \le A( \hTheta-J \bTheta_\M, \Phi)+2B_\NC(\bPsi_\M,\Phi, \hTheta-J \bTheta_\M)+2B_\NC(\bPsi-\bPsi_\M,\Phi, \hTheta-J \bTheta_\M).$$
Since $\trinl \bar\Psi-\bPsi_\M\trinr_\NC\le \varepsilon \le {\beta}/(4C_b)$, Lemma \ref{Bnc bound}.$(b)$ for the last term in the right hand side of the above inequality shows
$$ ({\beta}/2-\epsilon_3)\trinl \hTheta-J \bTheta_\M\trinr_2 \le A( \hTheta-J \bTheta_\M, \Phi)+2B_\NC(\bPsi_\M,\Phi, \hTheta-J \bTheta_\M).$$
 This, \eqref{auxiadje_1}, \eqref{adj_eq1}  and simple manipulation eventually lead to
\begin{align}
& ({\beta}/2-\epsilon_3)\trinl \hTheta-J \bTheta_\M\trinr_2 \le (\bar \Psi_{\M}-\Psi_{d},\Phi)-A(J \bTheta_\M, \Phi)-2B_\NC(\bPsi_\M,\Phi, J \bTheta_\M) =(\bar \Psi_{\M}-\Psi_{d},(1-I_\M) \Phi)\nonumber\\
&\quad -A_\NC((J -1)\bTheta_\M , \Phi)+A_\NC(\bTheta_\M, (I_\M-1) \Phi ) -2B_\NC(\bPsi_\M,\Phi, J \bTheta_\M)+2B_\NC(\bPsi_\M,I_\M\Phi,\bTheta_\M) \nonumber \\
& =(\bar \Psi_{\M}-\Psi_{d},(1-I_\M) \Phi)-A_\NC((J -1)\bTheta_\M, \Phi)+A_\NC(\bTheta_\M, (I_\M -1)\Phi) +2B_\NC(\bPsi_\M,\Phi, (1-J) \bTheta_\M)\nonumber\\
&\quad+2B_\NC(\bPsi_\M, (I_\M-1)\Phi,\bTheta_\M)=:\sum_{i=1}^{5}S_i. \label{s1.5.adj}
\end{align} 
A Cauchy-Schwarz inequality shows that $S_2 \le \trinl (J-1)\bTheta_{\M}\trinr_{\NC}.$ Since the piecewise second derivatives of $\bTheta_\M$ are constants, Lemma \ref{Morley_Interpolation}.$(a)$ implies $S_3=0$. Lemma \ref{Bnc bound}.$(b)$ and \eqref{dissta} prove $S_4 \le C_bM_1 \trinl (J-1)\bTheta_{\M}\trinr_{\NC}$. 
\noindent The orthogonality property of $J$ in Lemma \ref{hctenrich}.$(c)$ proves $B_\NC(\bPsi_\M,(1-J)I_\M\Phi,\P_0\bTheta_\M)=0$. This and elementary algebra lead to
\begin{align}\label{s5.adj}
S_5/2&=B_\NC(\bPsi_\M,(1-J)I_\M\Phi,(1-\P_0) \bTheta_\M)+B_\NC(\bPsi_\M,(J I_\M-1) \Phi,(1-J) \bTheta_\M) \nonumber\\
&\qquad +B_\NC((1-J)\bPsi_\M,(J I_\M-1) \Phi,J \bTheta_\M)+B(J\bPsi_\M,(J I_\M -1)\Phi,J \bTheta_\M).
\end{align}
Triangle inequalities, Lemmas \ref{hctenrich}.$(d)$ with $v=\bTheta$, \eqref{varepsilon}, the second inequality of \eqref{inf_suP_adjoint} and Lemma \ref{ap}.$(a)$ show that 
\begin{equation}\label{companion.adjoint}
\trinl J \bTheta_\M\trinr_2+\trinl \bTheta_\M \trinr_\NC\le \trinl (J-1) \bTheta_\M\trinr_\NC+2(\trinl \bTheta-\bTheta_\M\trinr_\NC+\trinl \bTheta \trinr_2)\le (2+\Lambda_\jc)\varepsilon+2\trinl \bTheta\trinr_2:=M_2.
\end{equation}
Lemmas \ref{Morley_Interpolation}.$(b)$ and \ref{hctenrich}.$(d)$ with $v=\Phi$ verify
\begin{equation}\label{jim}
\trinl (J I_\M -1)\Phi\trinr_\NC\le \trinl (J-1) I_\M \Phi\trinr_\NC+\trinl (I_\M -1)\Phi\trinr_\NC\ \le (\Lambda_\jc +1)\trinl (I_\M -1)\Phi\trinr_\NC \le C_{I} (\Lambda_\jc +1).
\end{equation}
The first three terms in the right-hand side of \eqref{s5.adj} are estimated now. The definition of $B_\NC(\bullet,\bullet,\bullet)$, the Cauchy-Schwarz inequality, \eqref{jim} and the definition \eqref{def.est.adjoint.volume2} prove
\begin{equation}\label{s51}
B_\NC(\bPsi_\M,(1-J)I_\M\Phi,(1-\P_0) \bTheta_\M) \le C_{I} (\Lambda_\jc +1)\big(\sum_{K\in\cT }\eta_{K,\P_0,\bTheta_{\rm{M}}}^2\big)^{1/2}.
\end{equation}
Lemma \ref{Bnc bound}.$(b)$, \eqref{dissta}, \eqref{companion.adjoint}, \eqref{jim},  Lemma \ref{hctenrich}.$(e)$ and the definition \eqref{def.est.state.edge}-\eqref{def.est.adjoint.edge} show
\begin{align}
B_\NC(\bPsi_\M,(J I_\M-1) \Phi,(1-J) \bTheta_\M)&\le C_bC_{I}C_J (\Lambda_\jc +1)M_1\big(\sum_{E\in\cE }\eta^2_{E,\bTheta_{\rm{M}}}\big)^{1/2},\label{s52}\\
 B_\NC((1-J)\bPsi_\M,(J I_\M-1) \Phi,J \bTheta_\M)&\le C_bC_{I}C_J (\Lambda_\jc +1)M_2\big(\sum_{E\in\cE }\eta^2_{E,\bPsi_{\rm{M}}}\big)^{1/2}.\label{s53}
\end{align}
The last term in the right hand side of \eqref{s5.adj} is estimated in its scalar version and details are provided for better clarity. The symmetry of $b(\bullet,\bullet,\bullet)$ with respect to the second and third variables, and an introduction of $\bpsi_{\M,1}$ and $\btheta_{\M,1}$ imply that the first term in the expansion can be rewritten as
\begin{align}
b(J\bpsi_{\M,1},(J I_\M-1) \phi_2,J \btheta_{\M,1})
&= b_\NC((J-1)\bpsi_{\M,1},J \btheta_{\M,1},(J I_\M-1) \phi_2)+b_\NC(\bpsi_{\M,1},(J-1) \btheta_{\M,1},(J I_\M -1)\phi_2)\nonumber\\
& \quad +b_\NC(\bpsi_{\M,1},\btheta_{\M,1},(I_\M-1) \phi_2)\label{scalar.b}
\end{align}
with $b_\NC(\bpsi_{\M,1},\btheta_{\M,1},(J-1)I_\M \phi_2)=0$ from Lemma \ref{hctenrich}.$(b)$ in the last step. Lemma \ref{Bnc bound}.$(b)$ (in its scalar version), \eqref{dissta}, \eqref{companion.adjoint}-\eqref{jim},  Lemma \ref{hctenrich}.$(e)$ and \eqref{def.est.state.edge}-\eqref{def.est.adjoint.edge} leads to bounds of the first and second terms in the right hand side of \eqref{scalar.b}. 
The third term in the right-hand side of \eqref{scalar.b} is combined with the scalar form of $S_1$ as
\begin{equation}\label{s5.s1}
2b_\NC(\bpsi_{\M,1},\btheta_{\M,1},(I_\M-1) \phi_2)+(\bpsi_{\M,2}-\psi_{d,2},(I_\M-1) \phi_2) \le C_I h^2\big(\sum_{K \in \cT}\|\bar \psi_{\M,2}-\psi_{d,2}+[\bar \psi_{\M,1}, \bar \theta_{\M,1}]\|_{L^2(K)}^2\big)^{1/2}
\end{equation}
 with the Cauchy-Schwarz inequality and Lemma \ref{Morley_Interpolation}.$(b)$. The remaining two terms in the expansion of $B_\NC(\bullet,\bullet,\bullet)$ are dealt with in an analogous way. 
A substitution of \eqref{s51}-\eqref{s5.s1} in $S_1+S_5$ and then the resulting estimates with $S_2$-$S_4$ in \eqref{s1.5.adj}, triangle inequality with $J \bTheta_\M$, Lemma \ref{hctenrich}.$(e)$ and the definitions \eqref{def.est.adjoint.volume}-\eqref{def.est.adjoint.edge} show
\begin{align}
\trinl \hTheta-\bTheta_\M\trinr_2^2 &\le \tilde{C}_{\rm AD, rel}^2\big(\sum_{K\in\cT }\eta^2_{K,\bTheta_{\rm{M}}}+\sum_{E\in\cE }\eta^2_{E,\bPsi_{\rm{M}}}+\sum_{E\in\cE }\eta^2_{E,\bTheta_{\rm{M}}}\big)\label{adj.t1.final}
\end{align}
\noindent with $\epsilon_3 \searrow 0$, and $\tilde{C}_{\rm AD, rel}^2:=4\beta^{-2}\big(C_J^2((\beta/2+1+C_bM_1)+8C_IC_bM_1(\Lambda_\jc+1))^2+C_J^2(8C_IC_bM_2(\Lambda_\jc+1))^2+C_I^2(1+4(\Lambda_\jc+1)^2)\big).$ 
%
%
%
The uniform boundedness property of $\mathcal{F}^{-1}_{\bPsi_{}}$ in Lemma \ref{automorphism} implies $\tnr{\hTheta-\bar{\Theta}}_{2}=\tnr{\mathcal{F}_{\bPsi_{}}^{-1}\mathcal{F}_{\bPsi_{}}(\hTheta-\bar{\Theta})}_{2}\le \|\mathcal{F}_{\bar\Psi}^{-1}\|_{\mathcal{L}(\bV+\bV_{\rm M})} \tnr{\mathcal{F}_{\bPsi_{}}(\hTheta-\bar{\Theta})}_{\NC}.$
The definition of $\mathcal{F}_{\bPsi_{}}$ given by \eqref{adjlinop}, \eqref{adj_eq} and \eqref{auxiadje_1} show that
\begin{align}
&\mathcal{F}_{\bPsi_{}}(\bar{\Theta}-\hTheta)=T(\bPsi-\Psi_d)-\mathcal{F}_{\bPsi_{}}({\hTheta})=T(\bPsi-\Psi_d)- \hTheta- T[{\mathcal B}_{\NC}'(\bPsi)^* (\hTheta)]\nonumber\\
&=T(\bPsi-\bPsi_\M)+ T[{\mathcal B}_{\NC}'(\bPsi_\M-\bPsi)^* (\hTheta)].\nonumber
\end{align}
Hence, Lemmas \ref{staest}.$(a)$ with absorbed constant in $`\lesssim'$ denoted as $\tilde{C}_b$, \ref{Bnc bound}.$(b)$, \eqref{apriori.discreteadj1} and Theorem \ref{thm.state.reliability} prove
\begin{align}\label{adj.t2}
\tnr{\hTheta-\bar{\Theta}}_{2} 
& \le  C_{\rm ST,rel}\|\mathcal{F}_{\bar\Psi}^{-1}\|_{\mathcal{L}(\bV+\bV_{\rm M})}\|T\|(\tilde{C}_b+2C_b{\beta}^{-1})\big(\sum_{K\in\cT }\eta^2_{K,\bPsi_{\rm{M}}}+\sum_{E\in\cE }\eta^2_{E,\bPsi_{\rm{M}}}+\|\bar u - \bar {u}_h \|_{L^2(\omega)}^2\big)^{1/2}.
\end{align}
The combination of \eqref{adj.t1.final} and  \eqref{adj.t2} concludes the proof.
\end{proof}
\begin{remark}
	\begin{itemize}
		\item[$(a)$] Note that the terms involoving $\P_0$ in the reliability estimator of adjoint equations $ \eta^2_{K,\bTheta_{\rm{M}}}$ of \eqref{def.est.adjoint.volume} are due to the combined effect of non-conformity of the method plus linear lower-order terms.

\item[$(b)$] It is possible to avoid the terms involving $\P_0$ in the reliability estimator $ \eta^2_{K,\bTheta_{\rm{M}}}$ of \eqref{def.est.adjoint.volume} which comes from $S_5=B_{\rm NC}(\bPsi_{\rm M},(I_{\rm M}-1)\Phi,\bTheta_{\rm M})$ in \eqref{s1.5.adj} with piece-wise integration by parts. However, this leads to several average terms in the edge estimators that are not residuals (in addition to the volume terms). The efficiency analysis for this is still open. A similar observation for the Navier-Stokes equation in the stream-vorticity formulation can be found in \cite[Remark 4.12]{carstensen2017nonconforming}.
	\end{itemize}
\end{remark} 

\subsection{A posteriori error analysis for the control variable}

Recall the auxiliary variable $\widetilde{u}_h$ given in \eqref{defn_post_proc}. This computable variable helps to derive the reliability estimate for the control variable. 
A key property in favor of the definition of $\widetilde{u}_h \in U_{\rm{ad}}$ is that $\widetilde{u}_h$ satisfies the optimality condition
	\begin{equation}\label{opt_con_ppt}
	\left({\mathbf C}^* {\bar \Theta_{\M}} + \alpha {\bf \widetilde{u}}_h ,  {\bf u} - {\bf \widetilde{u}}_h \right)_{{\boldsymbol{L}}^2(\omega)} \ge 0  \fl  {\bf u} =(u,0)^T, \; {u \in U_{ad}}.
	\end{equation}
 \noindent Define for $u,v \in U_{ad}$,\,$j'(u)v:=( {\mathcal C}^*{\theta_{{u,1}}} + \alpha { u},v)_{L^2(\omega)},$ where $j:  U_{ad} \cap {\mathcal O}(\bar u) \rightarrow {\mathbb R}$ is the reduced cost functional defined by $j(u):= \cJ(G(u), u)$ and $G(u) = \Psi_{u} =({\psi_{u,1}, \psi_{u,2}})\in \bV$ is the unique solution to \eqref{of} corresponding to $u$.
\begin{lem}[an auxiliary control estimate]\cite[Lemma 4.16]{cmj}\label{lem.auxiliary.control}
	Let $(\bar \Psi, \bar u)$ be a regular solution to \eqref{wform} and $(\bar \Psi, \bar \Theta, \bar u)$ solves \eqref{opt_con} that satisfies the sufficient second order optimality condition. Let $\widetilde{u}_h$ be defined as in \eqref{defn_post_proc}. Then for sufficiently small mesh-size $h$, there exists a $\delta_1>0$ such that 
	$	\| \bar u-\widetilde{u}_h\|_{L^2(\omega)}^2 \le2\delta_1^{-1} (j'(\widetilde{u}_h)-j'(\bar{u}))(\widetilde{u}_h-\bar{u}).$
\end{lem}

\noindent
Let $\widetilde{\Psi}$ and $\widetilde{\Theta}$ be the auxiliary continuous state and adjoint variables associated with the control $\widetilde{u}_h$.  That is, $\fl \Phi \in \bV$,  seek $(\widetilde{\Psi}, \widetilde{\Theta}) \in \bV \times \bV$  such that
\begin{align*}
&   {A}({\widetilde \Psi},\Phi)+{B}({\widetilde \Psi},{\widetilde\Psi},\Phi)=(\mbox{\bf F} + {\bf C}{\widetilde{\bf u}_h}, \Phi) \; \text{ and }\; 
 {A}(\Phi, {\widetilde\Theta})+2{B}({\widetilde \Psi},\Phi,{\widetilde \Theta})=(\widetilde \Psi- \Psi_d, \Phi). \label{aux_state_eq_tilde}
\end{align*}

\begin{thm}[Reliability for the control variable]\label{thm.control.reliability}
Let $(\bPsi,\bTheta,\bar{\bf u})$ (resp. $(\bPsi_{\rm{M}},\bTheta_{\rm{M}},\bar{\bf u}_h)$) solve the optimality system \eqref{opt_con} (resp.  \eqref{discrete.opt}). Then for a sufficiently small choice of the mesh size $h$,  there exists an $h-$independent positive constant $C_{\rm CON, rel}$ such that  
	\begin{equation}\label{control.reliability}
	\|\bar{u} - {\bar{u}}_h\|_{L^2(\omega)}^2 \le C_{\rm CON, rel}^2\bigg( \sum_{K\in\cT }\eta^2_{K,\bPsi_{\rm{M}}}+\sum_{K\in\cT }\eta^2_{K,\bTheta_{\rm{M}}}+ \sum_{K\in\cT }\eta^2_{K,\bar u_h}+\sum_{E\in\cE }\eta^2_{E,\bPsi_{\rm{M}}}+\sum_{E\in\cE }\eta^2_{E,\bTheta_{\rm{M}}}\bigg).
	\end{equation}
	\end{thm}
\begin{proof}
	  The triangle inequality with $\widetilde{u}_h$ and the definition \eqref{def.est.statecontrol.volume} lead to $ \| \bar u-\bar u_h \|_{L^2(\omega)}^2 \le \| \bar u-\widetilde{u}_h \|_{L^2(\omega)}^2 +\sum_{K \in \T}\eta^2_{K,\bar u_h}$. The continuous optimality condition \eqref{opt3} with $u=\widetilde{u}_h$ and \eqref{opt_con_ppt} with $u=\bar u$ show that
	$$j'(\bar u)(\widetilde{u}_h-\bar u) \ge 0, \quad 	-\left({\mathcal C}^* {\bar \theta_{\M,1}} + \alpha { \widetilde{u}_h} ,   { \widetilde{u}_h} -{\bar u} \right) \ge 0.$$
	These bounds, Lemma \ref{lem.auxiliary.control} and the definition of $j'(\bullet)$ lead to
	\begin{align*}
	\frac{\delta_1}{2}\| \bar u-\widetilde{u}_h\|_{L^2(\omega)}^2& \le  (j'(\widetilde{u}_h)-j'(\bar{u}))(\widetilde{u}_h-\bar{u})\le j'(\widetilde{u}_h)(\widetilde{u}_h-\bar{u}) \\
		&\le j'(\widetilde{u}_h)(\widetilde{u}_h-\bar{u})-\left({\mathcal C}^* {\bar \theta_{\M,1}} + \alpha { \widetilde{u}_h} ,   { \widetilde{u}_h} -{\bar u} \right)   =({\mathcal C}^*(\widetilde{\theta}-\bar \theta_{\M,1}), \widetilde{u}_h -{\bar u} ).
	\end{align*}
	Therefore, a Cauchy-Schwarz inequality results in
	$\| \bar u-\widetilde{u}_h\|_{L^2(\omega)}\le 2\delta_1^{-1} \trinl \widetilde{\Theta}-\bTheta_\M \trinr.$ 
	A triangle inequality that introduces $\hTheta$, Poincar\'e inequality with constant $C_{p}$, Lemma \ref{staest}.$(a)$ and \eqref{adj.t1.final} yield
	\begin{equation}\label{control.t1}
		\frac{\delta_1^2}{4}	\| \bar u-\widetilde{u}_h\|_{L^2(\omega)}^2 \le C_{p}^2  \trinl \widetilde{\Theta}-\hTheta\trinr_\NC^2 + \tilde{C}_b^2\tilde{C}_{\rm AD,rel}^2\bigg(\sum_{K\in\cT }\eta^2_{K,\bTheta_{\rm{M}}}+\sum_{E\in\cE }\eta^2_{E,\bTheta_{\rm{M}}}+\sum_{E\in\cE }\eta^2_{E,\bPsi_{\rm{M}}}\bigg).
	\end{equation}
The definitions \eqref{rep}, \eqref{defn_post_proc}, the Lipschitz property of operator $\Pi_{[u_{a},u_{b}]}$ and Lemma \ref{staest}.$(a)$ show
$\| \bar u-\widetilde{u}_h\|_{L^2(\omega)}  \leq \alpha^{-1} \tnr{\bTheta-\bar{\Theta}_{\rm{M}}}_{\bL^2(\omega)}\le \alpha^{-1}\tilde{C}_b \tnr{\bTheta-\bar{\Theta}_{\rm{M}}}_{\NC}.$ Hence, \eqref{varepsilon} implies $\| \bar u-\widetilde{u}_h\|_{L^2(\omega)} \le \alpha^{-1}\tilde{C}_b\varepsilon$. This, the estimate in \eqref{adj.t2} with $(\bTheta, \bPsi,\bar u)$ replaced by $(\widetilde{\Theta},\widetilde{\Psi},\widetilde{u}_h)$ and the definition \eqref{def.est.statecontrol.volume} show that
\begin{align*}
\trinl \widetilde{\Theta}-\hTheta\trinr_\NC \le C_{\rm ST,rel}\|\mathcal{F}_{\widetilde \Psi}^{-1}\|_{\mathcal{L}(\bV+\bV_{\rm M})}\|T\|(\tilde{C}_b+2C_b{\beta}^{-1})\big(\sum_{K\in\cT }\eta^2_{K,\bPsi_{\rm{M}}}+\sum_{E\in\cE }\eta^2_{E,\bPsi_{\rm{M}}}+\sum_{K \in \T}\eta^2_{K,\bar u_h}\big)^{1/2}.
	\end{align*}
A  substitution of this in \eqref{control.t1} concludes the proof.
\end{proof}

	\begin{proof}[Proof of Theorem \ref{thm.global.reliability}]
		The proofs follows from a combination of Theorems \ref{thm.state.reliability}, \ref{thm.adjoint.reliability} and \ref{thm.control.reliability} for any $\cT \in \bT(\delta)$ which satisfies \eqref{varepsilon}.
	\end{proof}
\section{Efficiency}\label{sec.efficiency}
This section deals with the {\it a posteriori} efficient error estimates for the control problem. The local efficiency proofs are based on the standard bubble function techniques \cite{verfurth}, \cite[Lemma 5.3]{CCGMNN18}. The combined result is stated first.
\begin{thm}{\bf(Efficiency)}\label{thm.global.efficiency}
	Let $(\bPsi,\bTheta,\bar{\bf u})$ (resp. $(\bPsi_{\rm{M}},\bTheta_{\rm{M}},\bar{\bf u}_h)$) solve the optimality system \eqref{opt_con} (resp. \eqref{discrete.opt}). Then, there exists a positive constant $C_{\rm eff}$ independent of $h$ such that
	\begin{align}
	\eta^2 &\le C_{\rm eff}^2\bigg( \tnr{\bPsi-\bPsi_{\rm{M}}}_{\rm NC}^2+	\tnr{\bTheta-\bTheta_{\rm{M}}}_{\rm NC}^2+\|\bar{u} - {\bar{u}}_h\|_{L^2(\omega)}^2 +\sum_{K \in \cT}h_K^4\left( \norm{f-f_h}_{L^2(K)}^2+\trinl\Psi_{d}-\Psi_{d,h}\trinr_{\bL^2(K)}^2\right)\nonumber\\
	&\qquad + \trinl (1-I_{\rm M}) \bTheta \trinr_{1,2,h}^2 +\trinl(1-I_{\rm M}) \bTheta\trinr_{0,\infty}^2+\trinl(1-\P_0)\bTheta\trinr_{0,\infty}^2\bigg),\nonumber
	\end{align}
	{{where $f_h$ (resp. $\Psi_{d,h}$) denotes the piece-wise average of $f$ (resp. $\Psi_{d}$) and $\eta^2:=\eta^2_{\rm ST}+\eta^2_{\rm AD}+\eta^2_{\rm CON}$ is the estimator defined in \eqref{global.reliability}.}}
\end{thm}

\begin{lem}[Local efficiency for state estimator]
	\label{thm.local.efficiency}
	Let $(\bPsi,\bTheta,\bar{\bf u})$ (resp. $(\bPsi_{\rm{M}},\bTheta_{\rm{M}},\bar{\bf u}_h)$) solve the optimality system \eqref{opt_con} (resp. \eqref{discrete.opt}). Then,
	\begin{align*}
	&  \eta_{K,\bPsi_{\rm{M}}}
	\lesssim \tnr{D^2(\bPsi-\bPsi_{\rm{M}})}_{\bL^2(K)}+h_K^2\left(\|\bar{u} - {\bar{u}}_h\|_{L^2(K)}+ \norm{f-f_h}_{L^2(K)}\right),\quad 
	\eta_{E,\bPsi_{\rm{M}}}\lesssim \tnr{D^2_{\rm NC}(\bPsi-\bPsi_{\rm{M}})}_{\bL^2(\Omega(K))},
		\end{align*}
	where $K \in \cT$, $E \in \mathcal{E}(\Omega(K))$ and $f_h$ denotes the piece-wise average of $f$.
\end{lem}
\begin{proof}
For each element $K\in \mathcal T$,  it holds that
\begin{align}\label{effi.1}
h_K^2\| f+\mathcal{C}u_h&+[\bar \psi_{\M,1},\bar \psi_{\M,2}]\|_{L^2(K)}  + h_K^2||[\bar \psi_{\M,1}, \bar \psi_{\M,1}]\|_{L^2(K)}\lesssim h_K^2 \norm{f-f_h}_{L^2(K)}\nonumber\\
+& h_K^2 \norm{\bar u-\bar u_h}_{L^2(K)} +\trinl D^2(\bar \Psi-\bar \Psi_{\M})\trinr_{\bL^2(K)} +\trinl D^2\bar \Psi\trinr_{\bL^2(K)}\trinl D^2(\bar \Psi-\bar \Psi_{\M})\trinr_{\bL^2(K)}. 
\end{align}
The proof of \eqref{effi.1} imitates the standard bubble functions arguments as in \cite[Lemma 5.3]{CCGMNN18}. In the proof therein for the first term in the left hand side of \eqref{effi.1}, set $\sigma:=(f_h+\mathcal{C}u_h+[\bar \psi_{\M,1},\bar \psi_{\M,2}])b_K^2$ in $K$, and zero in $\Omega \setminus K$, where $b_K$ denotes the standard interior bubble function \cite{verfurth}. Then the state equation \eqref{state_eq} with the test function $(\sigma,0)$, $\Delta^2 \bar \psi_{\M,1} =0$ and $\sigma\in H^2_0(K)$ prove \eqref{effi.1}.
The term $ ||[\bar \psi_{\M,1}, \bar \psi_{\M,1}]\|_{L^2(K)}$ can be estimated similar to the above analysis.

\smallskip

 \noindent For the edge estimator term, Lemma \ref{hctenrich}.$(e)$ with $v=\bpsi_{\M,1}$ implies, for $E \in \mathcal{E}(\Omega(K))$,
\begin{align}\label{edge_effi}
h_E \|[D^2 \bar \psi_{\M,1}\tau_E]_E\|_{L^2(E)}^2 
\lesssim \|D^2_{\rm NC}(\bar \psi_{\M,1} - \bar \psi_{1})\|^2_{L^2(\Omega(K))}.
\end{align}
Analogous arguments lead to similar result for the  edge estimator $ \|[D^2 \bar \psi_{\M,2}\tau_E]_E\|_{L^2(E)}^2$.
\end{proof}
\begin{lem}[Local efficiency for adjoint estimator]
	\label{thm.local.efficiency2}
		Let $(\bPsi,\bTheta,\bar{\bf u})$ (resp. $(\bPsi_{\rm{M}},\bTheta_{\rm{M}},\bar{\bf u}_h)$) solve the optimality system \eqref{opt_con} (resp. \eqref{discrete.opt}). Then,
	\begin{align}
	& \eta_{K,\bTheta_{\rm{M}}}
	\lesssim \tnr{D^2(\bPsi-\bPsi_{\rm{M}})}_{\bL^2(K)}+	\tnr{D^2(\bTheta-\bTheta_{\rm{M}})}_{\bL^2(K)}+h^2_K\trinl\Psi_{d}-\Psi_{d,h}\trinr_{\bL^2(K)}+\tnr{\bPsi-\bPsi_{\rm{M}}}_{\bL^2(K)} +\trinl\nabla(\bTheta_{\rm M}-\bTheta) \trinr_{\bL^2(K)}\nonumber\\
	&\quad + \trinl\nabla(1-I_{\rm M}) \bTheta\trinr_{\bL^2(K)} +\trinl(1-I_{\rm M}) \bTheta\trinr_{\bL^\infty(K)}+\trinl(1-\P_0)\bTheta \trinr_{\bL^\infty(K)},\;  \eta_{E,\bTheta_{\rm{M}}}\lesssim \tnr{D^2_{\rm NC}(\bTheta-\bTheta_{\rm{M}})}_{\bL^2(\Omega(K))},\nonumber
	\end{align}
	where $K \in \cT$, $E \in \mathcal{E}(\Omega(K))$ and $\Psi_{d,h}$ denotes the piece-wise average of $\Psi_d$.
\end{lem}
\begin{proof}
	For each element $K\in \mathcal T$,  it can be shown that
	\begin{align}\label{effi.4}
	h_K^2&\|\bar \psi_{\M,1}-\psi_{d,1} -[\bar \psi_{\M,1}, \bar \theta_{\M,2}]+ [\bar \psi_{\M,2},\bar \theta_{\M,1}]\|_{L^2(K)}+h_K^2\|\bar \psi_{\M,2}-\psi_{d,2}+[\bar \psi_{\M,1}, \bar \theta_{\M,1}]\|_{L^2(K)}\le h_K^2 \trinl{\Psi_{d}-\Psi_{d,h}}\trinr_{\bL^2(K)}\nonumber\\
	&+h_K^2\trinl\bar \Psi-\bar \Psi_{\M}\trinr_{\bL^2(K)}+\trinl D^2\bar \Theta\trinr_{\bL^2(K)}\trinl D^2(\bar \Psi-\bar \Psi_{\M})\trinr_{\bL^2(K)}+(1+\trinl D^2\bar \Psi\trinr_{\bL^2(K)})\trinl D^2(\bar \Theta-\bar \Theta_{\M})\trinr_{\bL^2(K)}.
	\end{align}
The proof of \eqref{effi.4} follows from the standard bubble functions technique. In the proof therein for the first term in the left hand side of \eqref{effi.4}, set $\sigma:=(\bar \psi_{\M,1}-\psi_{d,h,1} -[\bar \psi_{\M,1}, \bar \theta_{\M,2}]+ [\bar \psi_{\M,2},\bar \theta_{\M,1}])b_K^2 $ in $K$, and zero in $\Omega \setminus K$. The adjoint system (\ref{adj_eq})  with the test function $(\sigma,0)$, and the symmetric property of $b(\bullet,\bullet,\bullet)$ show that
$
\int_K D^2\bar\theta_1:D^2\sigma\dx-\int_K(\bar \psi_{1}-\psi_{d,1})\sigma\dx +\int_K ([\bar \psi_{1}, \bar \theta_{2}]- [\bar \psi_{2},\bar \theta_{1}] )\sigma\dx =0.
$
The combination of this, $\Delta^2 \bar\theta_{\M,1}=0$ and the arguments in the proof of \cite[Lemma 5.3]{CCGMNN18} prove \eqref{effi.4}. The estimates for the second term in the left hand side of \eqref{effi.4} is analogous to that of the first term. 
	
	\smallskip
	
	\noindent Consider $\| D^2\bpsi_{\M,1}(1-\P_0)\btheta_{\M,1}\|_{L^2(K)}$, $K \in \cT$.  The H\"older's inequality shows that
	\begin{equation}\label{step3}
	\| D^2\bpsi_{\M,1}(1-\P_0)\btheta_{\M,1}\|_{L^2(K)} \le  \| D^2\bpsi_{\M,1}\|_{L^2(K)} \|(1-\P_0)\btheta_{\M,1}\|_{L^\infty(K)}.
	\end{equation}
	\noindent The triangle inequality with $\P_0 I_\M \btheta_{1}$ leads to $\|(1-\P_0) \btheta_{\M,1}\|_{L^\infty(K)}\le \|(1-\P_0)(\btheta_{\M,1}-I_\M \btheta_{1}) \|_{L^\infty(K)}+\|(1-\P_0) I_\M \btheta_{1} \|_{L^\infty(K)}$. The inverse inequality \cite[Theroem 3.2.6]{Ciarlet} for the first term, triangle inequality with $(1-\P_0)\btheta_1$ for the second term, {  {projection estimate for $\P_0$ in $L^2(K)$ \cite[Proposition 1.135]{ErnJLU_2004}}} and the boundedness property of $\P_0$ prove
	\begin{align*}
	\|(1-\P_0) &\btheta_{\M,1}\|_{L^\infty(K)}
	\lesssim h^{-1} \|(1-\P_0)(\btheta_{\M,1}-I_\M \btheta_{1}) \|_{L^2(K)}+\|(I_\M-1) \btheta_{1}\|_{L^\infty(K)}\\
	&\hspace{4cm}+\|(1-\P_0)\btheta_{1} \|_{L^\infty(K)}+\|\P_0(\btheta_{1}-  I_\M \btheta_{1}) \|_{L^\infty(K)}\\
	&\lesssim ( \|\nabla(\btheta_{\M,1}-\btheta_{1}) \|_{L^2(K)}+ \|\nabla(1-I_\M) \btheta_{1}\|_{L^2(K)} )+\|((1-I_\M) \btheta_{1}\|_{L^\infty(K)}+\|(1-\P_0)\btheta_{1} \|_{L^\infty(K)}.
	\end{align*}
	This with \eqref{step3} result in
	\begin{align}
	\| D^2\bpsi_{\M,1}(1-\P_0)\btheta_{\M,1}\|_{L^2(K)} &\lesssim  \| D^2\bpsi_{\M,1}\trinr_{L^2(K)} \Big(\|\nabla(\btheta_{\M,1}-\btheta_{1}) \|_{L^2(K)}+ \|\nabla(1-I_\M)\btheta_{1}) \|_{L^2(K)}\nonumber\\
	&\qquad \qquad +\|(1-I_\M) \btheta_{1}\|_{L^\infty(K)}+\|(1-\P_0)\btheta_{1} \|_{L^\infty(K)}\Big).\label{efficiency.p0}
	\end{align}
	From \eqref{dissta}, $\trinl D^2\bPsi_{\M}\trinr_{L^2(K)} \le M$. The estimates for the remaining terms $ \| D^2\bpsi_{\M,1}(1-\P_0)\btheta_{\M,2}\|_{L^2(K)},~~\| D^2\bpsi_{\M,2}(1-\P_0)\btheta_{\M,1}\|_{L^2(K)}$
	follow from similar arguments and hence the details are omitted for brevity.
Lemma \ref{hctenrich}.$(e)$ leads to the desired estimate for the edge estimator $\eta_{E,\bTheta_{\rm{M}}}$. Analogous terms as the last term in the right hand side of \eqref{efficiency.p0} are dealt with in \cite[Theorem 4.10]{SCTGAKN2015}.
\end{proof}
\begin{lem}[Local efficiency for control estimator]
	\label{thm.local.efficiency3}
	Let $(\bPsi,\bTheta,\bar{\bf u})$ (resp. $(\bPsi_{\rm{M}},\bTheta_{\rm{M}},\bar{\bf u}_h)$) solve the optimality system \eqref{opt_con} (resp. \eqref{discrete.opt}). Then, 
	$\eta_{K,\bar u_h} \le \alpha^{-1}\tnr{\bTheta-\bTheta_{\rm{M}}}_{\bL^2(K)}+\|\bar{u} - {\bar{u}}_h\|_{L^2(K)}.$

\end{lem}
\begin{proof} The definitions \eqref{rep}, \eqref{defn_post_proc} and the Lipschitz property of operator $\Pi_{[u_{a},u_{b}]}$ show
$$\| \bar u-\widetilde{u}_h\|_{L^2(K)} \le \|  \Pi_{[u_{a},u_{b}]}( - \alpha^{-1}(\mathcal{C}^*({\bar \theta}_{1}-\bar{\theta}_{\M,1})))\|_{L^2(K)} \leq \alpha^{-1} \tnr{\bTheta-\bar{\Theta}_{\rm{M}}}_{\bL^2(K)}.$$
	This and  a triangle inequality prove $
	\| \widetilde{u}_h-\bar u_h \|_{L^2(K)} \leq \alpha^{-1} \tnr{\bTheta-\bar{\Theta}_{\rm{M}}}_{\bL^2(K)}+\| \bar{u}-\bar u_h \|_{L^2(K)} $ and concludes the proof of local efficiency for the control variable.
\end{proof}
\begin{proof}[\bf Proof of Theorem \ref{thm.global.efficiency}] Recall the definition of the complete estimator $\eta$ from (\ref{global.reliability}). The summation over all the element and edges of the triangulation $\mathcal T$, and the local efficiency  results in Lemmas \ref{thm.local.efficiency}-\ref{thm.local.efficiency3} show that
\begin{align}
\eta^2&=\sum_{K\in\cT }\eta^2_{K,\bPsi_{\rm{M}}}+\sum_{K\in\cT }\eta^2_{K,\bTheta_{\rm{M}}}+ \sum_{K\in\cT }\eta^2_{K,\bar u_h}+\sum_{E\in\cE }\eta^2_{E,\bPsi_{\rm{M}}}+\sum_{E\in\cE }\eta^2_{E,\bTheta_{\rm{M}}}\nonumber\\
&\lesssim \tnr{\bPsi-\bPsi_{\rm{M}}}_{\rm NC}^2+	\tnr{\bTheta-\bTheta_{\rm{M}}}_{\rm NC}^2+\|\bar{u} - {\bar{u}}_h\|_{L^2(\omega)}^2 +\sum_{K \in \cT}h_k^4 ( \norm{f-f_h}_{L^2(K)}^2+\trinl \Psi_{d}-\Psi_{d,h}\trinr_{L^2(K)}^2) +\tnr{\bPsi-\bPsi_{\rm{M}}}_{}^2 \nonumber\\
& \qquad +	\tnr{\bTheta-\bTheta_{\rm{M}}}_{}^2+\tnr{\bTheta-\bTheta_{\rm{M}}}_{1,2,h}^2+ \trinl (1-I_{\rm M}) \bTheta \trinr_{1,2,h}^2
 +\trinl(1-I_{\rm M}) \bTheta\trinr_{0,\infty}^2+\trinl(1-\P_0)\bTheta\trinr_{0,\infty}^2.\nonumber
\end{align}
This and Lemma \ref{staest} result in
\begin{align}
\eta^2&\lesssim \tnr{\bPsi-\bPsi_{\rm{M}}}_{\rm NC}^2+	\tnr{\bTheta-\bTheta_{\rm{M}}}_{\rm NC}^2+\|\bar{u} - {\bar{u}}_h\|_{L^2(\omega)}^2 +\sum_{K \in \cT}h_k^4\left( \norm{f-f_h}_{L^2(K)}^2+\trinl \Psi_{d}-\Psi_{d,h}\trinr_{L^2(K)}^2\right)\nonumber\\
&\qquad + \trinl (1-I_{\rm M})\bTheta \trinr_{1,2,h}^2 +\trinl(1-I_{\rm M}) \bTheta\trinr_{0,\infty}^2+\trinl(1-\P_0)\bTheta\trinr_{0,\infty}^2.\nonumber
\end{align}
Here the  constant absorbed in $`\lesssim'$ depends on the shape-regularity of $\mathcal T$. This concludes the proof. 
\end{proof}


\section{Numerical results}\label{sec.numericalresults}
The results of the numerical experiments that support the {\it a priori} and {\it a posteriori} estimates are presented in this section. 
\subsection{Preliminaries}\label{preliminaries2}
\noindent The state and adjoint variables are discretised using the Morley FE and the control variable is discretised using piecewise constant functions. The discrete solution $( \bar{\Psi}_\M, \bar{\Theta}_\M, {\bar{ u}}_h)$ is computed using a combination of Newtons' method in an inner loop and primal dual active set strategy in an outer loop, see \cite[Section 6.1]{SCNNDS} for the details of the implementation procedure for the a priori case for a different choice of the trilinear form. 
The initial guess for $(\bPsi_\M,\bTheta_\M)$ in the Newton's iterative scheme is chosen as the discrete solution to the biharmonic part of the discrete state and adjoint equations in \eqref{state_eq1} and \eqref{adj_eq1}. At each iteration of primal dual active set algorithm, the Newtons' method converges in ten iterations when the errors between final level and the penultimate level in Euclidean norm  is less than $10^{-9}$. The primal dual active set algorithm terminates within four steps.

\smallskip
\noindent
The numerical experiments are performed over the uniform and adaptive refinements. The uniform mesh refinement has been done by red-refinement criteria, where each triangle is subdivided into four sub-triangles by connecting the midpoints of the edges. The standard adaptive algorithm  Solve-Estimate-Mark-Refinement \cite{verfurth,CCGMNN18} is used for the adaptive refinement, which is described in Section \ref{sec.adaptivemesh}.

\noindent
 Let $\bPsi_\ell$ be the discrete solution $\bPsi_\M$ at the $\ell$th level for $\ell=1,2,3,..$ and define $e_\ell(\bPsi):=\trinl \bPsi-\bPsi_\ell \trinr_{\NC}.$
The order of convergence in the energy norm at $\ell$th level for $\bPsi$ is computed as
${\rm{Order(\ell)}}:={\rm{log}}\big({e_\ell(\bPsi)}/{e_{\ell+1}(\bPsi)}\big)/{\rm{log}}(h_{\ell}/ h_{\ell+1})$ (resp. $ {\rm{Order(\ell)}}:={\rm{log}}\big({e_\ell(\bPsi)}/{e_{\ell+1}(\bPsi)}\big)/{\rm{log}}(\rm{NDOF}_{\ell}/ \rm{NDOF}_{\ell+1})$)  for uniform refinements (resp. {adaptive refinements}), 
where $h_{\ell}$ and $\rm{NDOF}_\ell$ denote the mesh-size and  number of degrees of freedom at $\ell$th level  triangulation $\mathcal{T_\ell}$. The total number of degrees of freedom is NDOF $:=  2 \text{ dim }(\bV_\M)+ \text{ dim }(U_{h,ad})$. Finally, the total error is a sum of $\trinl\bPsi-\bPsi_{\M}\trinr_{\NC}$, $\tnr{\bTheta-\bTheta_{\rm{M}}}_{\rm NC}$ and $\|\bar{u} - {\bar{u}}_h\|$.

\smallskip
\noindent Two examples  are presented to illustrate the {\it a priori} and {\it a posteriori} reliability and efficiency estimates with $\omega=\Omega$ so that ${\mathcal C}=\rm{I}$. The first example is considered over unit square domain where the solution of \vket is sufficiently smooth and the second example is over an L-shaped domain where the solution of \vket belongs to $\bV \cap \bH^{2+\gamma}(\Omega)$ with $\gamma \approx 0.5445$.
\subsection{Uniform refinement}
\begin{example}\label{eg1} {\bf (Convex Domain)} 
Let the computational domain be $\Omega=(0,1)^2$. The model problem is constructed in such a way that the exact solution is known. The data in the distributed optimal control problem are chosen as 
$
\bar{\psi}_1=\bar{\psi}_2=\sin^2(\pi x)\sin^2(\pi y),\;  \bar{\theta}_1=\bar{\theta}_2=x^2 y^2 (1-x)^2 (1-y)^2, \; \bar{u}(x)=\Pi_{[-750,-50]}(-{1}/{\alpha} \: \bar{\theta}_1(x)), 
$
$\alpha=10^{-5},$ where
$\bar{\Psi}=(\bar{\psi}_1,\bar{\psi}_2)$ and $\bar{\Theta}=(\bar{\theta}_1,\bar{\theta}_2)$ denote the optimal state and adjoint variables.
The source terms $f, g$ and observation $\bar{\Psi}_d=({\bar{\psi}_{d,1},\bar{\psi}_{d,2}})$ for $\bar{\Psi}$ are then computed using $
f=\Delta^2\bar{\psi}_1-[\bar{\psi}_1,\bar{\psi}_2]- \bar{u},\quad g=\Delta^2\bar{\psi}_2+\half[\bar{\psi}_1,\bar{\psi}_1]$ and 
${\bar{\psi}_{d,1}}=\bar{\psi}_1-\Delta^2\bar{\theta}_1,\quad {\bar{\psi}_{d,2}}=\bar{\psi}_2-\Delta^2\bar{\theta}_2+[\bar{\psi}_1,\bar{\theta}_1].$
\end{example}
\noindent The relative errors and orders of convergence for the state, adjoint and control variables and the combined relative error and order of convergence are presented in Table \ref{table.squaredomain}. Since $\O$ is convex, Theorem \ref{conv.apriori} predicts linear order of convergence for the state and adjoint variables (resp. control variable) in the energy (resp. $L^2$) norm. These theoretical rates of convergence are confirmed by the numerical outputs. 
\begin{table}[ht]
		\footnotesize
		\begin{center}
			\begin{tabular}{|c||c|c||c|c ||c| c|| c|c|}
					\hline 
	$h$	& $\frac{\trinl\bPsi-\bPsi_{\M}\trinr_{\NC}}{\trinl\bPsi \trinr_{\NC}}	$ & Order& $\frac{\tnr{\bTheta-\bTheta_{\rm{M}}}_{\rm NC}}{\tnr{\bTheta}_{\rm NC}}$& Order& $\frac{\|\bar{u} - {\bar{u}}_h\|}{\|\bar{u}\|}$& Order& Total Error & Order\\
				\hline 				\hline 
				  0.2500 &   1.208162 & -- &   1.793806 & -- &   1.638183 & -- &   1.574966 & --\\
				0.1250 &   0.654690 &       0.88 &   0.730500 &       1.30 &   0.581509 &       1.49 &   0.592373 &       1.41 \\
				0.0625 &   0.357561 &       0.87 &   0.377143 &       0.95 &   0.175353 &       1.73 &   0.202300 &       1.55 \\
				0.0312 &   0.183915 &       0.96 &   0.190428 &       0.99 &   0.055375 &       1.66 &   0.074381 &       1.44 \\
				0.0156 &   0.092662 &       0.99 &   0.095447 &       1.00 &   0.021294 &       1.38 &   0.031846 &       1.22 \\
				0.0078 &   0.046422 &       1.00 &   0.047753 &       1.00 &   0.009674 &       1.14 &   0.015107 &       1.08 \\
				\hline 
			\end{tabular}
		\end{center}
		\vspace{-0.2in}
		\caption{\small Errors and orders of convergence for state, adjoint and control variables in Example \ref{eg1}}
		\label{table.squaredomain}
	\end{table}
	
\begin{example}\label{eg2} {\bf (Non-convex Domain)} 
Consider the non-convex L-shaped domain  $\Omega=(-1,1)^2 \setminus\big{(}[0,1)\times(-1,0]\big{)}$. The source terms $f,g$ and the observation $\Psi_d=({\psi_{d,1},\psi_{d,2}})$ are chosen such that the model problem has the exact singular solution {{\cite[Section 3.4.1]{Grisvard}}} given by 
$
\bar{\psi}_1=\bar{\psi}_2=\bar{\theta}_1=\bar{\theta}_2=(r^2 \cos^2\theta-1)^2 (r^2 \sin^2\theta-1)^2 r^{1+ \gamma}g_{\gamma,\omega}(\theta)
$
where $ \gamma\approx 0.5444837367$ is a non-characteristic 
root of $\sin^2( \gamma\omega) =  \gamma^2\sin^2(\omega)$, $\omega=\frac{3\pi}{2}$, and
$g_{\gamma,\omega}(\theta)=(\frac{1}{\gamma-1}\sin ((\gamma-1)\omega)-\frac{1}{ \gamma+1}\sin(( \gamma+1)\omega))(\cos(( \gamma-1)\theta)-\cos(( \gamma+1)\theta))$ 
$-(\frac{1}{\gamma-1}\sin(( \gamma-1)\theta)-\frac{1}{ \gamma+1}\sin(( \gamma+1)\theta))
(\cos(( \gamma-1)\omega)-\cos(( \gamma+1)\omega)).$ 
The exact control $\bar{u}$ is chosen as $\bar{u}(x)={\Pi}_{[-600,-50]} ( -
1/ \alpha \: \bar{\theta}_1(x) )$, where $\alpha=10^{-3}$. 
\end{example}
\begin{table}[ht]
		\footnotesize
		\begin{center}
			\begin{tabular}{|c|c||c|c||c|c ||c | c|| c|c|}
					\hline
				$h$ &NDOF	& $\frac{\trinl\bPsi-\bPsi_{\M}\trinr_{\NC}}{\trinl\bPsi \trinr_{\NC}}	$ & Order& $\frac{\tnr{\bTheta-\bTheta_{\rm{M}}}_{\rm NC}}{\tnr{\bTheta}_{\rm NC}}$& Order& $\frac{\|\bar{u} - {\bar{u}}_h\|}{\|\bar{u}\|}$& Order& Total Error & Order\\
				\hline 				\hline 
				
				   0.3536 &156 &   1.371575 & -- &   1.355646 & -- &   0.760376 & -- &   0.812881 & --\\
				0.1768 &740&   0.875686 &       0.65 &   0.906115 &       0.58 &   0.498261 &       0.61 &   0.532436 &       0.61 \\
				0.0884 & 3204 &  0.502780 &       0.80 &   0.508682 &       0.83 &   0.197497 &       1.34 &   0.224325 &       1.25 \\
				0.0442 & 13316 &   0.270684 &       0.89 &   0.268696 &       0.92 &   0.072470 &       1.45 &   0.089636 &       1.32 \\
				0.0221 & 54276 &  0.143731 &       0.91 &   0.141920 &       0.92 &   0.029279 &       1.31 &   0.039162 &       1.19 \\

				\hline 
			\end{tabular} 
			\vspace{-0.1in}
			\caption{\small Errors and orders of convergence  for state, adjoint and control variables in Example \ref{eg2}}
			\label{table.Lshapeddomain}
		\end{center}
		\vspace{-0.2in}
	\end{table} 
\noindent Table \ref{table.Lshapeddomain} shows error estimates and the convergence rates of the state, adjoint and  control variables. Since $\Omega$ is non-convex, only suboptimal orders of convergence for the state and adjoint variables in the energy norm are obtained as predicted by Theorem~\ref{conv.apriori}.   The numerical results show a better convergence rate for control which probably indicates that the numerical performance is carried out in the non-asymptotic region. 
%

\subsection{Adaptive mesh refinement}\label{sec.adaptivemesh}
The  standard adaptive algorithm: Solve-Estimate-Mark-Refine is used for the adaptive mesh-refinement. The total estimator $ \eta^2:=\eta^2_{\rm ST}+\eta^2_{\rm AD}+\eta^2_{\rm CON}$ is considered in the adaptive algorithm.

{
	\footnotesize
	\begin{algorithm}[H]
		\renewcommand{\thealgocf}{}
		\SetAlgoLined
{
		Set the initial triangulation $\mathcal T_0$\;
		Set the maximum number of iteration ${\rm Max}_{\ell}$\;
		\While{$\ell < {\rm Max}_{\ell}$}{
		 {\bf Solve}: Compute the solution  $( \bar{\Psi}_\M, \bar{\Theta}_\M, {\bar{ u}}_h)$ over the triangulation $\mathcal T_\ell$  using Newtons' method and primal dual active set strategy\;
		 {\bf Estimate}: Compute the complete estimator $\eta_\ell^2$ from (\ref{global.reliability})\;
		 {\bf Mark}: Mark a minimal subset $\mathcal M_\ell \subset \mathcal T_\ell$ by D\"{o}rfler marking criteria\;
		 \[ 0.2\sum_{K\in \mathcal T_\ell} \eta_\ell^2(K)  \leq \sum_{K\in \mathcal M_\ell}\eta_\ell^2(K)\]
		 {\bf Refine}:  Compute the closure of $\mathcal M_\ell$ and genrate new triangulation $\mathcal T_{\ell+1}$ using the newest vertex bisection\;
		 		Update the triangulation\;}
	}
\caption{Adaptive Mesh-refinement Algorithm }
\label{algo1}
\end{algorithm}
}
\medskip

\noindent {\bf Convex Domain:} Consider Example \ref{eg1}. This is a test case over the square domain with a smooth exact solution, performed to test the performance of the adaptive estimator for the uniform refinement. Table \ref{table.sqestidomain} depicts the convergence history  of  the estimators (defined in \eqref{def.estimators}) for the uniform refinements for the state, adjoint and control estimators. The combined error and estimator's convergence  are also computed. It is observed  that the individual errors and estimators as well as the combined error have linear order of convergence. Hence, the theoretical rates of convergence are confirmed
by these numerical outputs. 

\begin{table}[ht]
		\footnotesize
		\begin{center}
			\begin{tabular}{|c||c|c||c|c ||c | c|| c|c|}
					\hline
				$h$	& $\eta^2_{\rm ST}$ & Order& $\eta^2_{\rm AD}$  & Order& $\eta^2_{\rm CON}$ & Order& $\eta \quad$ & Order\\
				\hline \hline
   0.2500 &  71.094778 & -- &   1.431732 & -- &  84.529580 & -- & 110.461610 & --\\
    0.1250 &  42.932902 &       0.73 &   0.246138 &       2.54 &  31.588806 &       1.42 &  53.302414 &       1.05 \\
    0.0625 &  25.524713 &       0.75 &   0.114321 &       1.11 &  13.233906 &       1.26 &  28.751701 &       0.89 \\
    0.0312 &  13.597806 &       0.91 &   0.058131 &       0.98 &   6.187437 &       1.10 &  14.939481 &       0.94 \\
    0.0156 &   6.936296 &       0.97 &   0.029364 &       0.99 &   3.031687 &       1.03 &   7.569953 &       0.98 \\
    0.0078 &   3.490596 &       0.99 &   0.014749 &       0.99 &   1.508851 &       1.01 &   3.802777 &       0.99 \\
\hline
\end{tabular}
\end{center}
\vspace{-0.2in}
\caption{\small Estimator and order of convergence for state, adjoint and control variables in Example \ref{eg1}}
\label{table.sqestidomain}
\end{table}
\noindent {\bf Non-convex Domain: } 
This numerical experiment is performed over the non-convex domain (Example \ref{eg2}) with the exact solution has a singularity at the origin. The numerical experiment starts on the initial mesh with 24 triangles, and then adaptive refinements are carried out using  Algorithm \ref{algo1}. 
\begin{figure}[ht!]
	\begin{tabular}{cc}
	\includegraphics[width=8cm,height=6cm]{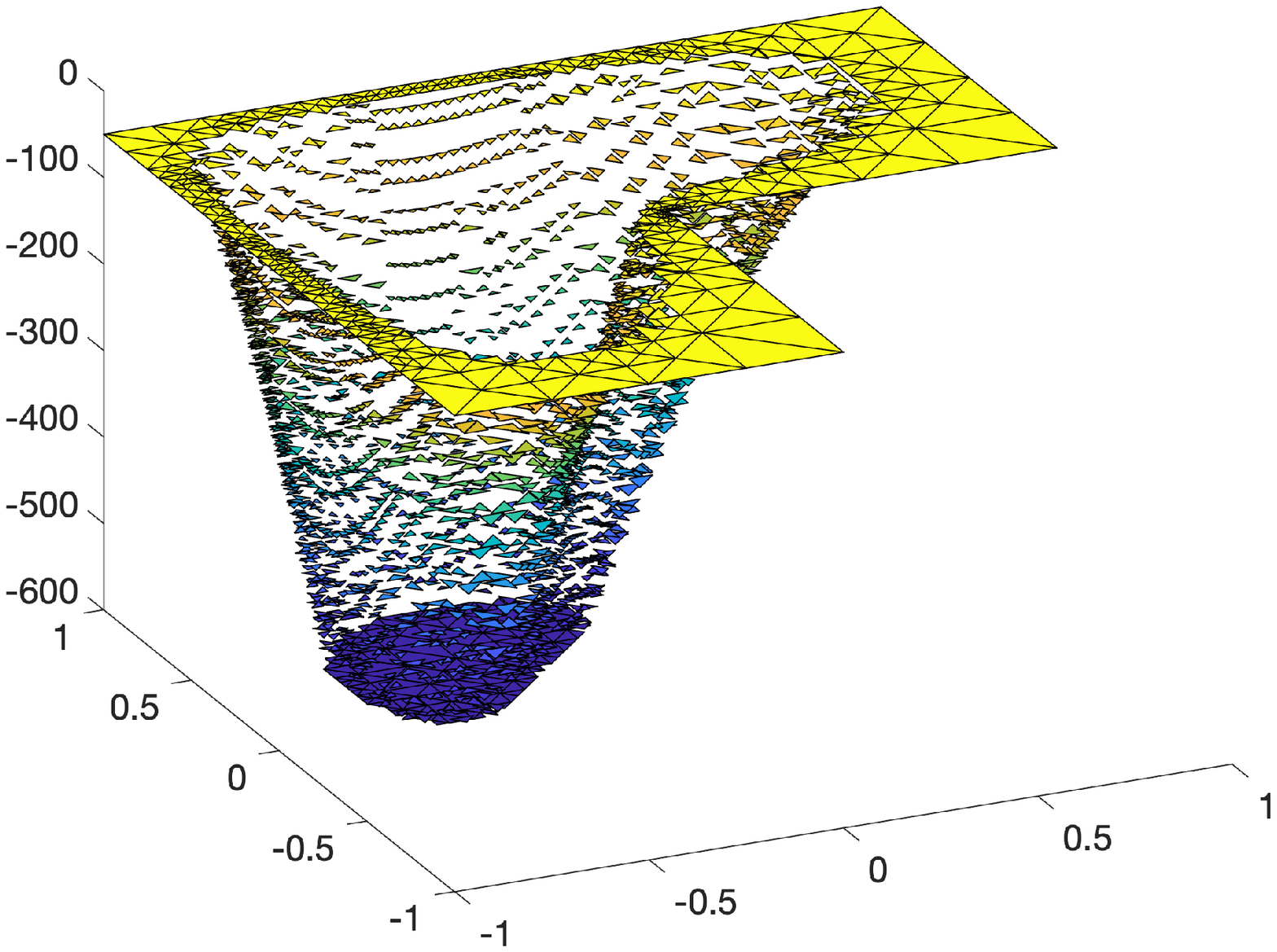}&
\includegraphics[width=7cm,height=6cm]{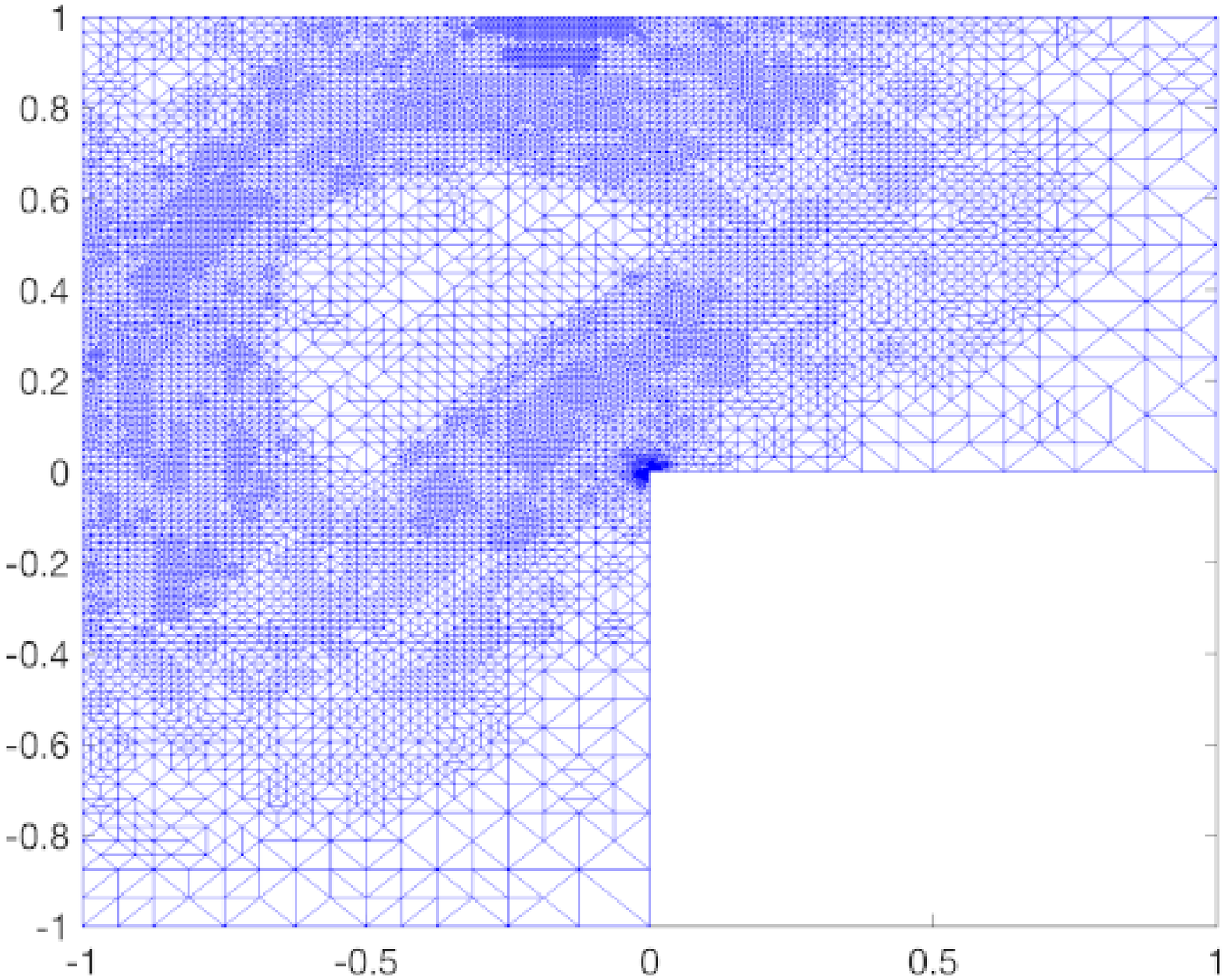}\\
(a) Discrete control $\bar u_h$   &   (b) Adaptive mesh-refinement\\
	\end{tabular}
	\vspace{-0.1in}
	       \caption{\small Discrete control solution $\bar u_h$ (left) and the adaptive mesh-refinement (right) (at level $\ell=24$) in Example \ref{eg2}}
\label{fig2}	
\end{figure}
\noindent   Figure \ref{fig2} shows that the significant adaptive refinement occurs near the control variable interface and the singularity point of the L-shaped domain.  
This is somewhat expected as the state and adjoint solutions have a singularity at the origin, and from Figure \ref{fig3}  it is observed that the control estimator dominates other estimators.
This supports the efficiency of the adaptive estimator in the theoretical estimates obtained in the previous section.  Figure  \ref{fig3}  and  Table \ref{table.Lshapeddomain2} also indicate that the errors and estimators have optimal convergence in the adaptive refinement.
\begin{figure}[ht!]
	\centering
	\includegraphics[width=4.9cm,height=5cm]{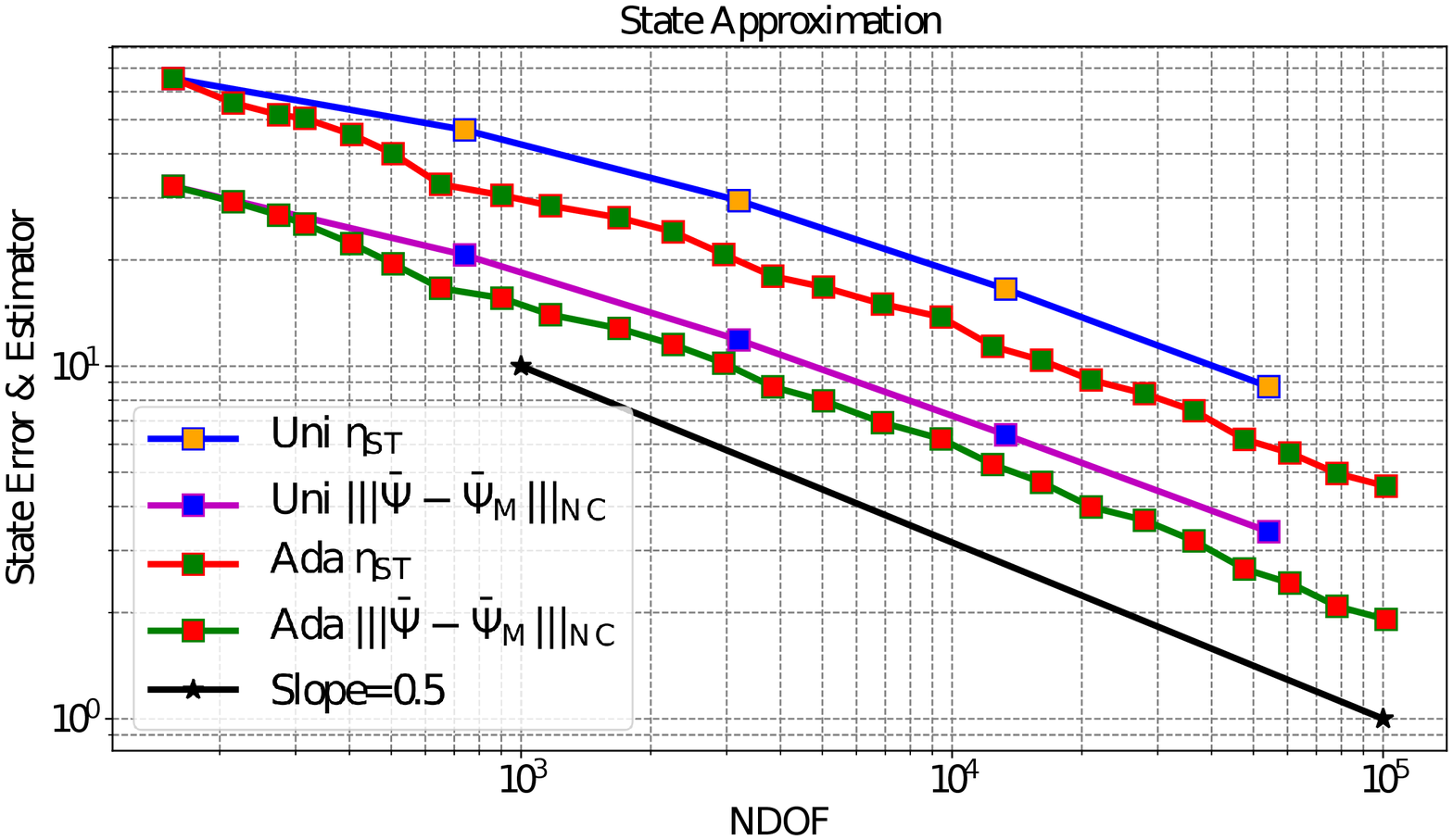}
		\includegraphics[width=4.9cm,height=5cm]{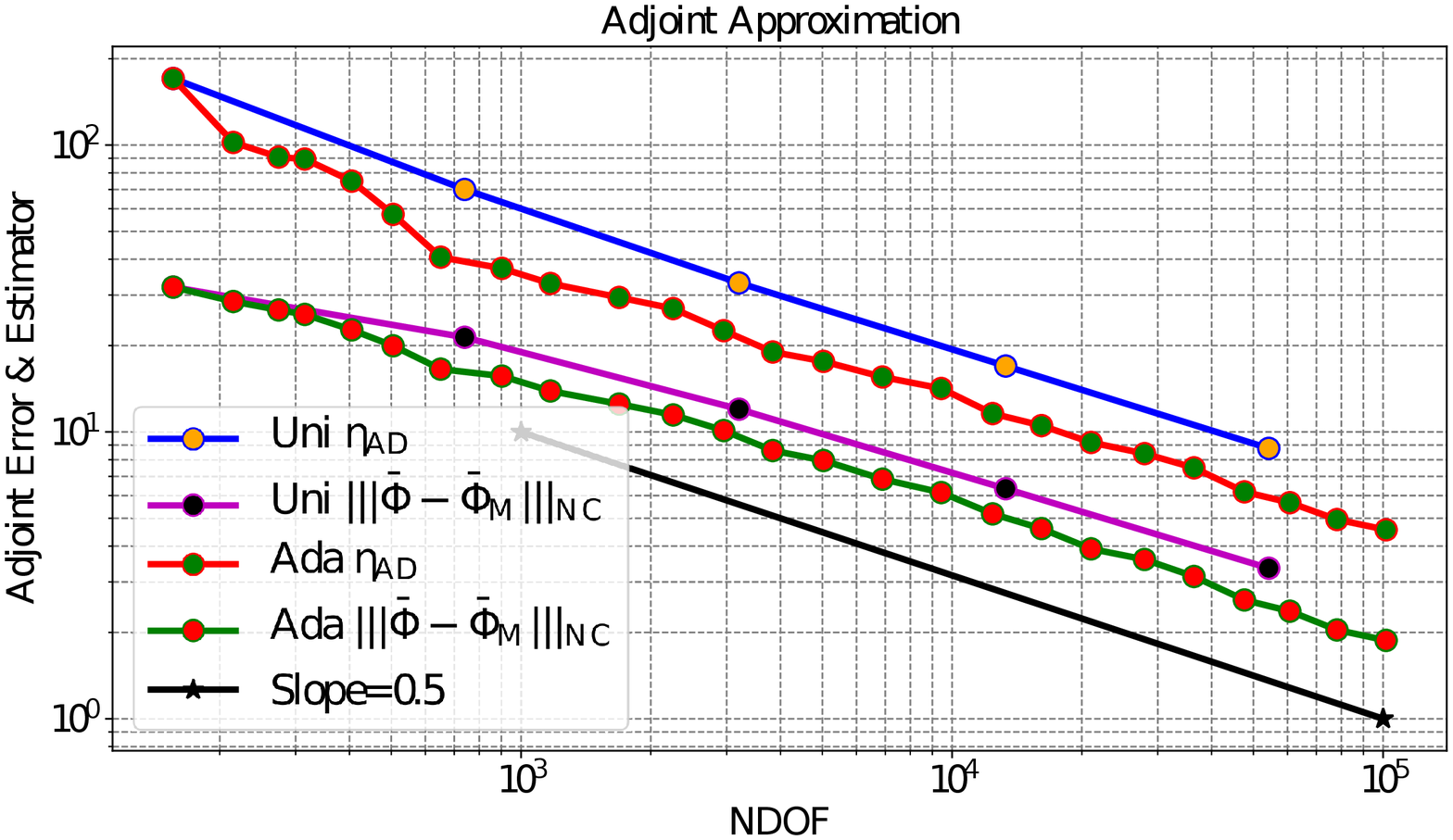}
			\includegraphics[width=4.9cm,height=5cm]{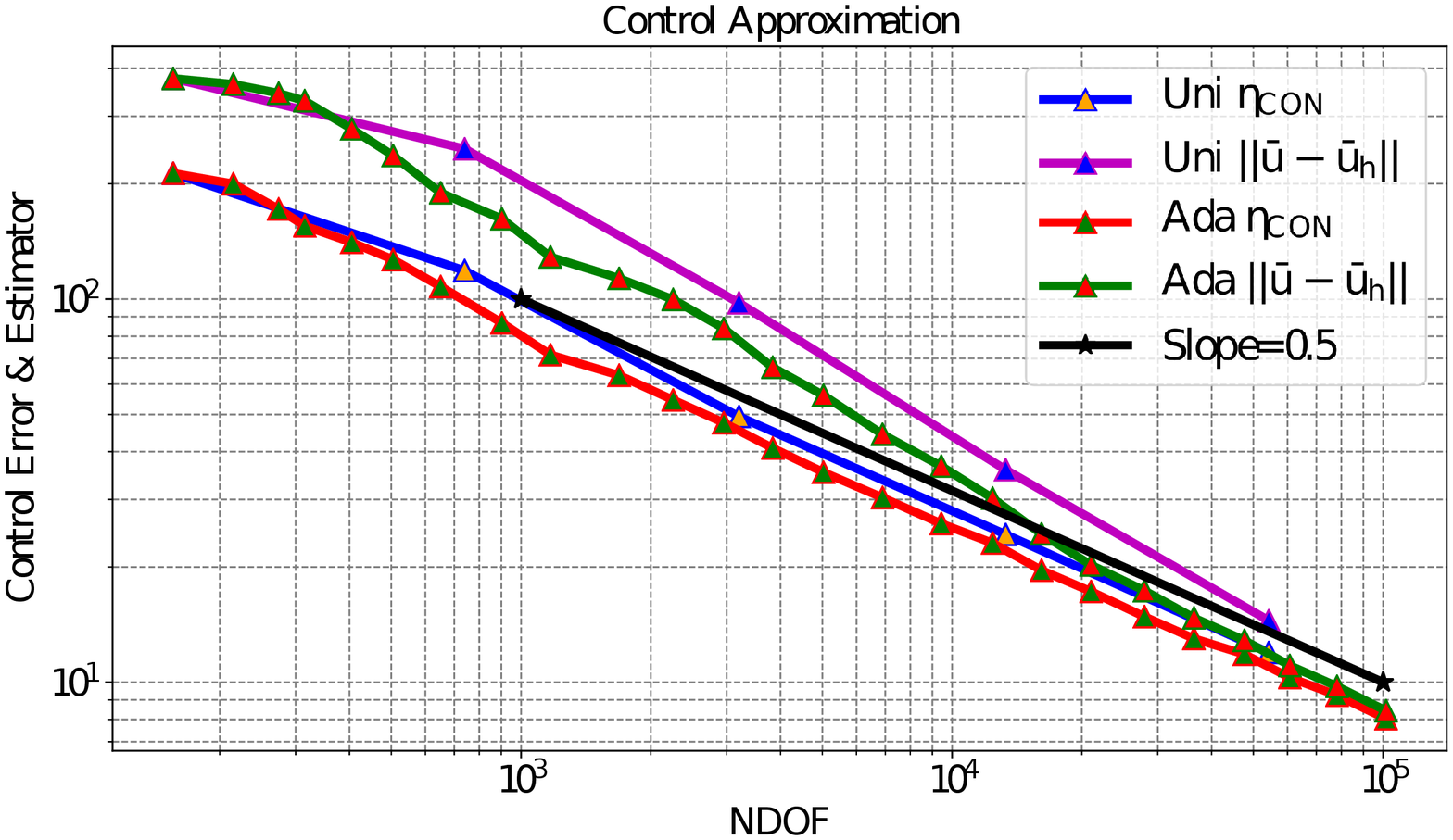}
		\vspace{-0.2in}
	\caption{\small Convergence plot of the approximation errors and estimators with adaptive and uniform refinement for state, adjoint and control variables in Example \ref{eg2} (State approximation (left), Adjoint approximation (Middle),  Control approximation (right), Uni=Uniform refinement, Ada=Adaptive refinement).}
	\label{fig3}
\end{figure}
\begin{table}[ht!]
	\footnotesize
	\begin{center}
	\begin{tabular}{|c|c||c|c||c|c ||c | c|| c|c|c|}
			\hline
		Iteration&	NDOF	& $\frac{\trinl\bPsi-\bPsi_{\M}\trinr_{\NC}}{\trinl\bPsi \trinr_{\NC}}	$ & Order& $\frac{\tnr{\bTheta-\bTheta_{\rm{M}}}_{\rm NC}}{\tnr{\bTheta}_{\rm NC}}$& Order& $\frac{\|\bar{u} - {\bar{u}}_h\|}{\|\bar{u}\|}$& Order& Total Error & Order&Ratio=T.Er/Et\\
			\hline 				\hline 

0 &        156 &   1.371575 & -- &   1.355646 & -- &   0.760376 & -- &   0.812881 & -- &   1.570478 \\
4 &        405 &   0.943121 &       0.51 &   0.961573 &       0.49 &   0.561676 &       0.67 &   0.595680 &       0.65 &   2.043008 \\
8 &       1170 &   0.592890 &       0.42 &   0.588820 &       0.46 &   0.260259 &       0.88 &   0.289033 &       0.81 &   1.949036 \\
12 &       3837 &   0.370235 &       0.59 &   0.365010 &       0.61 &   0.133979 &       0.90 &   0.154315 &       0.84 &   1.840881 \\
16 &      12417 &   0.222890 &       0.61 &   0.219544 &       0.62 &   0.061046 &       0.70 &   0.074987 &       0.68 &   1.502512 \\
20 &      36405 &   0.135407 &       0.51 &   0.132838 &       0.52 &   0.029756 &       0.62 &   0.038840 &       0.59 &   1.295190 \\
23 &      78146 &   0.088349 &       0.60 &   0.086474 &       0.60 &   0.019719 &       0.50 &   0.025611 &       0.53 &   1.215638 \\
\hline
		\end{tabular}
	\end{center}
	\vspace{-0.25in}
	\caption{\small Errors and orders of convergence  for state, adjoint and control variables  with adaptive refinement in Example \ref{eg2}}
	\label{table.Lshapeddomain2}
\end{table}
\begin{figure}[ht!]
	\begin{tabular}{cc}
	\includegraphics[width=8.5cm,height=6cm]{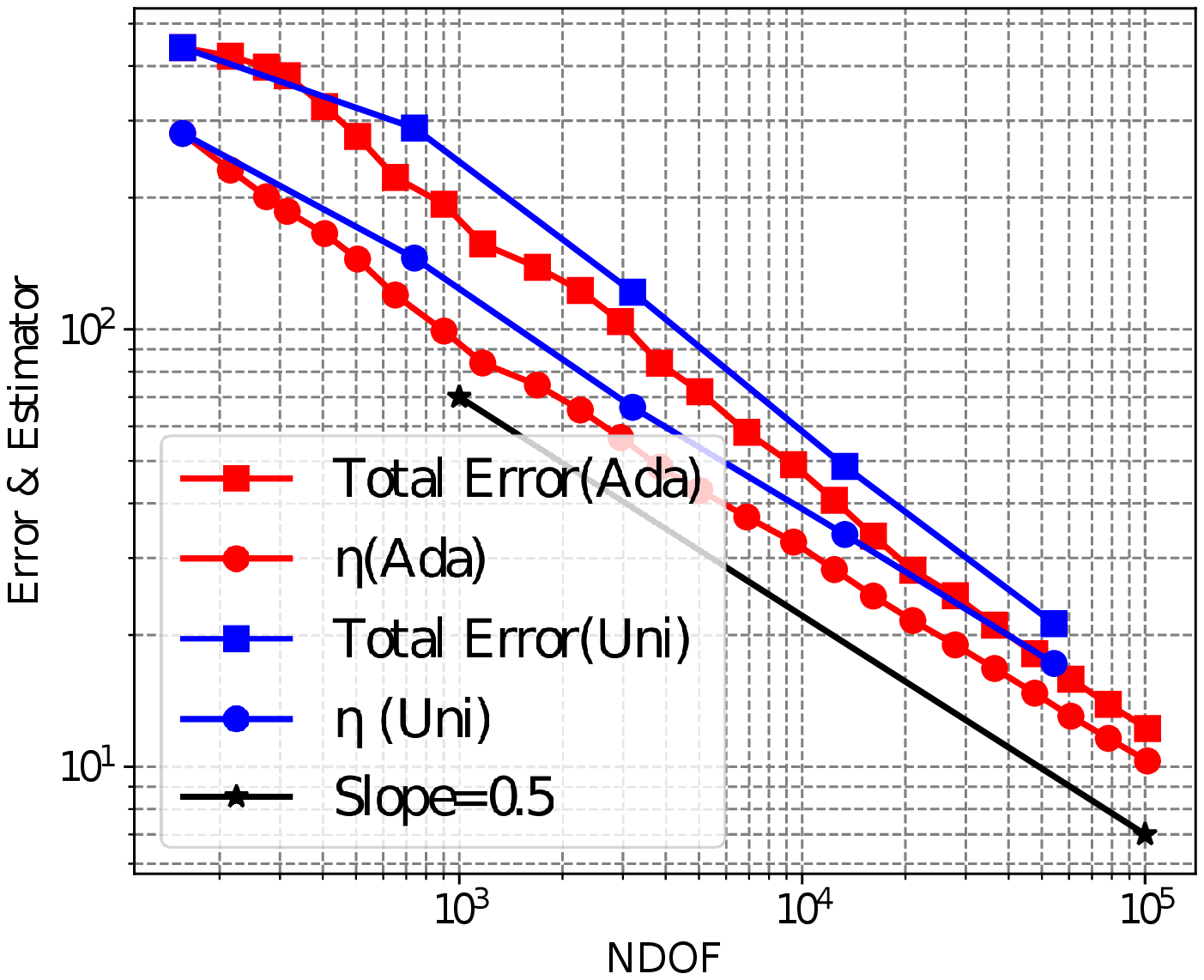} &
\includegraphics[width=6.5cm,height=6cm]{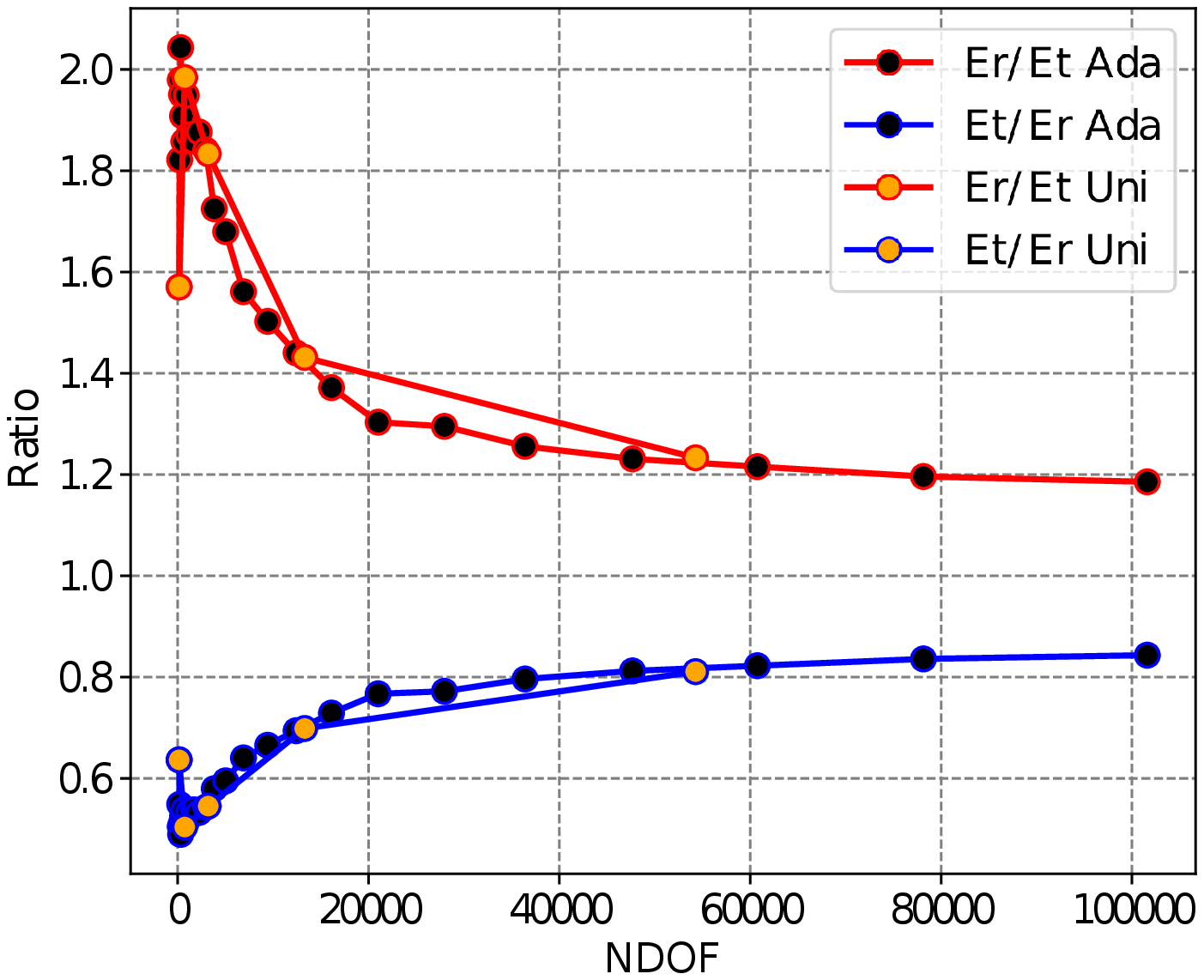}\\
(a) Total error and estimator  & (b) Efficiency and Reliability 
	\end{tabular}
	\caption{\small Convergence plot (left), and reliability and efficiency constants (right) over uniform and adaptive refinements (right) in Example \ref{eg2} ( Er=Total Error, Et=Complete Estimator, Uni=Uniform refinement, Ada=Adaptive refinement).}
	\label{fig4}
\end{figure}
\medskip

\noindent Figure \ref{fig4}.$(a)$ displays the convergence history of the total error and estimator; both achieve optimal convergence in adaptive refinement. Further, it can be observed that the adaptive refinements are doing better in terms of accuracy compared to the uniform refinements.  Figure \ref{fig4}.$(b)$ illustrates that reliability and efficiency constants are approaching a constant value with mesh refinement, which is numerical evidence for the efficiency and reliability of a~posteriori estimator derived in the theory section.
\vspace{-0.1in}
\subsection*{Acknowledgements}
Asha K. Dond would like to acknowledge support from Science \& Engineering Research Board (SERB), Government of  India under Start-up Research Grant, Project No. SRG/2020/001027. Neela Nataraj gratefully acknowledges SERB MATRICS grant MTR/2017/000199
titled {\it Finite element methods for nonlinear plate bending problems and optimal control problems governed
by nonlinear plates} and SERB POWER Fellowship
SPF/2020/000019. Devika Shylaja thanks National Board for Higher Mathematics, India for the financial support towards the research work (No: 0204/3/2020/R$\&$D-II/2476).
\bibliographystyle{plain}
\bibliography{vKeBib}

\appendix
\renewcommand{\thesection}{A\Alph{section}}
\renewcommand{\thesubsection}{A\Alph{section}.\arabic{subsection}}
\section*{Appendix}\label{sec.appendix}
\renewcommand\thesection{A.\arabic{section}}
\renewcommand{\theequation}{A.\arabic{equation}}
\pagestyle{empty}
\begin{appendices}

\setcounter{lemma}{0}
    \renewcommand{\thelemma}{\Alph{section}A.\arabic{lemma}}     
The detailed proofs of the {\it a priori} error estimates that are different from \cite{SCNNDS} are presented in this section. 

\subsection{A priori estimates}

\medskip

\noindent {\bf A linear mapping}

\medskip
\noindent For a given ${\bf g}=(g_1,g_2) \in {\bf V}'$,  let the {\it linear operator} $T \in {\mathcal L}(\bV', \bV)$  defined by 
$T {\bf  g}:= \boldsymbol{\xi} =(\xi_1, \xi_2) \in \bV$ solves the biharmonic system  $A(\boldsymbol{\xi}, \Phi) = \big\langle {\bf g}, \Phi \big\rangle \mbox{ for all } \Phi \in \bV$, that is,  
\begin{align} \label{bih}
	\Delta^2  \xi_1 = g_1 \,\, \mbox{ in } \Omega, \;  \Delta^2  \xi_2 = g_2  \mbox{ in } \Omega,   \; \xi_1=0,\,\frac{\partial \xi_1}{\partial \nu}=0\text{ and } \xi_2=0,\,\frac{\partial  \xi_2}{\partial \nu} =0
	\text{  on }\partial\Omega. 
\end{align}
Moreover, for ${\bf{g}}\in \bH^{-1}(\Omega)$, $\boldsymbol{\xi}\in \bV\cap\bH^{2+\gamma}(\Omega), \; \gamma \in (1/2,1]$, the  elliptic regularity \cite{BS} result stated next holds. 
\begin{align}\label{reguesti}\trinl\boldsymbol{\xi}\trinr_{{2}}\lesssim
	\trinl {\bf{g}}\trinr_{-1}, \quad \trinl\boldsymbol{\xi}\trinr_{{2+\gamma}}\lesssim
	\trinl {\bf{g}}\trinr_{-1}.
\end{align}
For  ${\bf  g} \in \bV_\M'$, define the bounded discrete operator $T_\NC: \bV_\M' \rightarrow\bV_{\M}$ by 
$T_\NC {\bf  g}:= \boldsymbol{\xi}_\M $ where $\boldsymbol{\xi}_\M \in \bV_\M$ solves the discrete problem
\begin{equation} \label{dis_bih}
	{A}_{\NC}(\boldsymbol{\xi}_\M, \Phi_\M) = \big \langle {\bf g}, \Phi_\M\big \rangle \;  \mbox{ for all } \Phi_\M \in \bV_\M.
\end{equation}
The  lemma stated next is utilized to prove the existence and uniqueness of the solution to \eqref{wv1}.
\begin{lemma}[\it An intermediate estimate] \label{aux}
	Let $\bar \Psi \in \bV \cap {\bH}^{2+\gamma}(\Omega)$ be a regular solution to \eqref{wform}. Then $\forall \epsilon >0 $, $\exists h_{1}>0$ such that 
	$\| T[\mathcal{B}'_{\rm NC}(\bar \Psi)] - T_{\rm NC} [\mathcal{B}_{\rm NC}'(\Psi)]\|_{{\mathcal L}(\bV+\bV_{\rm M})} < \epsilon  \fl \Psi \in B_{\rho_\epsilon}(\bar \Psi),$
	whenever $0 < h < h_1$.
\end{lemma}
\begin{proof}
	%
	%
	\noindent	For ${\bf z} \in \bV+\bV_\M$, \eqref{B.derivative} and Lemma \ref{Bnc bound}.$(b)$ show that $\mathcal{B}_{\NC}'(\bar{\Psi})({\bf z})\in \bV'$ and $\mathcal{B}_{\NC}'(\bar{\Psi})({\bf z})\in \bV_\M'$.  For  $\Psi\in \bV+\bV_{\M}$, the definitions of $T(\bullet)$ and $T_\NC(\bullet)$, and \eqref{reguesti} imply that ${\boldsymbol \theta}(\bar \Psi) =:T[\mathcal{B}'_{\NC}(\bar \Psi) ({\bf z})]\in \bV \cap \boldsymbol{H}^{2+\gamma}(\Omega)$ and ${\boldsymbol \theta}_\M( \Psi)=:T_\NC[\mathcal{B}_{\NC}'( \Psi) ({\bf z})]\in \bV_\M$ solve
	\begin{align}
		A({\boldsymbol \theta}(\bar \Psi), \Phi)&=\big \langle \mathcal{B}'_{\NC}(\bar  \Psi) {(\bf z)}, \Phi \big \rangle_{} \fl \Phi \in \bV, \label{a1} \\
		{A}_{\NC}({\boldsymbol \theta}_\M( \Psi), \Phi_\M)& = \langle \mathcal{B}'_{\NC}( \Psi) {(\bf z)}, \Phi_\M \rangle  \fl \Phi_\M \in \bV_\M. \label{a2}
	\end{align} 
	Let ${\boldsymbol \theta_\M}(\bar \Psi)$ and ${\boldsymbol \theta}^J_\M(\bar \Psi)\in \bV_\M$ solve the discrete problems 
	\begin{align}
		A_\NC({\boldsymbol \theta_\M}(\bar \Psi), \Phi_\M)&=\big \langle \mathcal{B}'_{\NC}(\bar  \Psi) {(\bf z)}, \Phi_\M \big \rangle \fl \Phi_\M \in \bV_\M, \label{a1d} \\
		{A}_{\NC}({\boldsymbol \theta}^J_\M( \bar \Psi), \Phi_\M)& = \langle \mathcal{B}'_{\NC}(\bar \Psi) {(\bf z)}, J\Phi_\M \rangle_{}  \fl \Phi_\M \in \bV_\M.\label{a2d}
	\end{align}   
	\noindent A triangle inequality yields  
	\begin{equation}\label{trian}
		\trinl {\boldsymbol \theta}(\bar \Psi) -{\boldsymbol \theta}_\M( \Psi) \trinr_{\NC} \le  
		\trinl {\boldsymbol \theta}(\bar \Psi) -{\boldsymbol \theta}_\M^J( \bar \Psi) \trinr_{\NC}  +
		\trinl {\boldsymbol \theta}_\M^J(\bar \Psi) -{\boldsymbol \theta}_\M( \bar \Psi) \trinr_{\NC} + 
		\trinl {\boldsymbol \theta}_\M(\bar \Psi) -{\boldsymbol \theta}_\M( \Psi) \trinr_{\NC}.
	\end{equation}
	Notice that ${\boldsymbol \theta}_\M^J( \bar \Psi)$ is the Morley nonconforming solution to \eqref{a1} for a modified right-hand side $\mathcal{B}'_{\NC}(\bar \Psi) {(\bf z)}\circ J \in \bV_\M'.$ 
	The best-approximation result from \cite[Theorem 3.2]{CCNN2021} shows
	\begin{align} \label{bap.biharmonic}
		\trinl  {\boldsymbol \theta}( \bar \Psi)-{\boldsymbol \theta}^J_\M( \bar \Psi) \trinr_{\NC} \le 
		\sqrt{1+\Lambda_\jc^2} \; \trinl  (1-I_\M) {\boldsymbol \theta}( \bar \Psi) \trinr_{\NC}.
	\end{align}
	This plus the interpolation estimate from Lemma \ref{Morley_Interpolation}.$(c)$, \eqref{reguesti}, \eqref{a1} and Lemma~\ref{Bnc bound}.$(d)$ imply
	\begin{align}\label{intermediate1}
		& 	\trinl  {\boldsymbol \theta}( \bar \Psi)-{\boldsymbol \theta}^J_\M( \bar \Psi) \trinr_{\NC} \lesssim 
		h^{\gamma} \sqrt{1+\Lambda_\jc^2}  \: \trinl {\boldsymbol \theta}( \bar \Psi) \trinr_{2+\gamma} \lesssim   h^{\gamma}  \trinl \mathcal{B}'_{\NC}(\bar  \Psi) {(\bf z)}\trinr_{-1} 
		\lesssim  h^{\gamma} \trinl \bPsi\trinr_{2+\gamma}\trinl {\bf z}\trinr_\NC.
	\end{align}
	The combination of \eqref{a1d} and \eqref{a2d}, and \eqref{B.derivative} show
	$$A_{\NC} ({\boldsymbol \theta_\M}(\bar \Psi)- {\boldsymbol \theta}^J_\M( {\bar \Psi}), \Phi_\M )=
	2B_\NC({\bar \Psi}, {\bf z},  (1 - J)  \Phi_\M ). $$
	The inverse inequality and Lemma \ref{hctenrich}.$(d)$ prove $\trinl (1 - J)  \Phi_\M \trinr_{0,\infty}\lesssim h^{-1}\trinl (1 - J)  \Phi_\M \trinr\lesssim h\trinl \Phi_\M \trinr_\NC$. This, Lemma \ref{Anc.bound}.$(c)$ with test function $\Phi_\M:= {\boldsymbol \theta_\M}(\bar \Psi)- {\boldsymbol \theta}^J_\M( {\bar \Psi})$ and Lemma \ref{Bnc bound}.$(a)$ imply
	\begin{align} \label{intermediate2}
		\trinl {\boldsymbol \theta}_{\M}^J(\bar \Psi) -{\boldsymbol \theta}_\M( \bar \Psi) \trinr_{\NC}\lesssim h \trinl {\bar \Psi} \trinr_{2} \trinl {\bf z}\trinr_{\NC}.
	\end{align}
	The combination of \eqref{a2} and \eqref{a1d}, and \eqref{B.derivative}, Lemma \ref{Anc.bound}.$(c)$ with test function ${\boldsymbol \theta_\M}(\Psi)- {\boldsymbol \theta}_\M( {\bar \Psi})$, and Lemma \ref{Bnc bound}.$(b)$ prove 
	$	\trinl {\boldsymbol \theta}_\M(\Psi) - {\boldsymbol \theta}_\M(\bar \Psi) \trinr_{\NC}\lesssim  \trinl {\bf z}\trinr_{\NC} \trinl \Psi-{\bar \Psi} \trinr_{\NC}.$ 
	A substitution of this and \eqref{intermediate1}-\eqref{intermediate2} in \eqref{trian} leads to the result that for any preassigned $\epsilon>0$, $h_{1}$ and the radius $\rho_{\epsilon}>0$ can be chosen small such that for all $\Psi \in B_{\rho_\epsilon}(\bar \Psi)$, 
	$\trinl T[\mathcal{B}'_{\NC}(\bar \Psi) ({\bf z})]-T_\NC[\mathcal{B}_{\NC}'( \Psi) ({\bf z})]\trinr_{\NC}<\epsilon\trinl\bf z\trinr_{\NC},$
	that leads to the desired estimate.
\end{proof}

\noindent The next lemma is a standard result in Banach spaces that helps to prove Lemma \ref{automorphism}.

\begin{lemma} \label{banach}
	Let $X$ be a Banach space,
	$A\in{\mathcal{L}}(X)$ be  invertible  and $B\in{\mathcal{L}}(X)$. If
	$\|A-B\|_{{\mathcal{L}}(X)}<1/\|A^{-1}\|_{{\mathcal{L}}(X)}$, then
	$B$ is invertible. If
	$\|A-B\|_{{\mathcal{L}}(X)}<1/(2\|A^{-1}\|_{{\mathcal{L}}(X)})$, then
	$\|B^{-1}\|_{{\mathcal{L}}(X)}\leq 2\|A^{-1}\|_{{\mathcal{L}}(X)}$. 
\end{lemma}
\noindent The uniform boundedness result for the inverse of the linear mapping $\mathcal{F}_{\Psi_u}$  with a bound independent of the discretization parameter $h$ {\it without assuming the extra regularity of $\bPsi$} is proved next. This result was used to derive the {\it a posteriori} error estimates for the adjoint variable.
\begin{proof}[Proof of Lemma \ref{automorphism}]
\cite[Lemma 4.3]{SCNNDS} shows that $\mathcal{F}_{\bPsi}$   is an automorphism on $\bV+\bV_{\rm{M}}$ if  $\bar{\Psi} \in \bV $ is a regular solution to \eqref{wform}. Also, for ${\boldsymbol \xi}+{\boldsymbol \xi}_\M\in \bV+\bV_\M$, the invertibility of $\mathcal{F}_{\bPsi}$ leads to
$	\mathcal{F}_{\bar\Psi}^{-1}({\boldsymbol \xi}+{\boldsymbol \xi}_\M)={\boldsymbol \xi}+{\boldsymbol \xi}_\M-(\mathcal{A}+\mathcal{B}'(\bar{\Psi})^*)^{-1}\mathcal{B}_{\NC}'(\bar{\Psi})^*({\boldsymbol \xi}+{\boldsymbol \xi}_\M).$ This,  and Lemma \ref{Bnc bound}.$(b)$ imply $\|\mathcal{F}_{\bPsi}^{-1}\|_{\mathcal{L}(\bV+\bV_{\rm M})}\lesssim  1+\|(\mathcal{A}+\mathcal{B}
'(\bar{\Psi})^*)^{-1}\|_{\mathcal{L}(\bV',\bV)}\trinl\bar{\Psi}\trinr_{2}$. Since $\cA^*=\cA$ and the operator norm of an operator and its adjoint are equal,  $\|(\mathcal{A}+\mathcal{B}
'(\bar{\Psi})^*)^{-1}\|=\|(\mathcal{A}+\mathcal{B}
'(\bar{\Psi}))^{-1}\|$ and hence, there exists a constant $C$ independent of $h$ such that $\frac{1}{\|\mathcal{F}_{\bPsi}^{-1}\|_{\mathcal{L}(\bV+\bV_{\rm M})}}\ge C.$

\smallskip

\noindent For $\Phi \in \bV+\bV_\M$, the definition of $\mathcal{F}_\Psi$ in \eqref{adjlinop}, the boundedness property of $T(\bullet)$ and $B_\NC(\bullet,\bullet,\bullet)$ in Lemma\ref{Bnc bound}.$(b)$ imply
\begin{align}\label{eq.fpsi}
\|\mathcal{F}_{\bPsi}(\Phi)-\mathcal{F}_{\Psi_u}(\Phi)\|_{\mathcal{L}(\bV+\bV_{\rm M})}\le \| T[{\mathcal B}_{\NC}'(\bPsi-\Psi_u)^* (\Phi)] \|\lesssim \trinl \bPsi-\Psi_u\trinr_2.
\end{align}
Theorem \ref{th2.5} shows $G(\bar u)=\bPsi$, $G( u)=\Psi_u$ and the uniform boundedness of $\trinl (\mathcal{A}+\mathcal{B}'(\Psi_{u}))^{-1}\trinr_{\mathcal{L}(\bV',\bV)}$ whenever $u\in {\mathcal O}({\bar u})$. Hence, for $u_t= {u} + t(\bar u- u)$ and $\Psi_{t}=G(u_{t})$, mean value theorem, Theorem \ref{th2.5} and  $u \in {\mathcal O}({\bar u})$ prove
\begin{align*} 
& \trinl \bPsi- \Psi_u \trinr_2  
=\trinl { \int_{0}^1} G'(u_t)({\mathbf C}(\bar{ \bf u}-{\bf {u}})))  \: {\rm dt} \trinr_{2} 
=  \trinl { \int_{0}^1} (\mathcal{A}+\mathcal{B}'(\Psi_{t}))^{-1}({\mathbf C}(\bar {\bf u}-{\bf {u}}))  \: {\rm dt} \trinr_{2} \lesssim \|\bar u -  {u}\|_{L^2(\omega)}.
\end{align*}
Since $u$ is sufficently close to $\bar u$, \eqref{eq.fpsi} leads to
\begin{align*}
\|\mathcal{F}_{\bPsi}-\mathcal{F}_{\Psi_u}\|_{\mathcal{L}(\bV+\bV_{\rm M})}\le C \le \frac{1}{\|\mathcal{F}_{\bPsi}^{-1}\|_{\mathcal{L}(\bV+\bV_{\rm M})}}.
\end{align*}
An application of Lemma \ref{banach} concludes the proof.
\end{proof}
\begin{proof}[Proof of Theorem \ref{thm.adj.energy}]
\noindent{\bf Step 1 (isolates a crucial term).}
	 Let $\brho_\M:= I_\M \Theta_{u} -\Theta_{u,\M} \in \bV_\M$. The triangle inequality leads to 
	\begin{align}\label{triang}
	\trinl \Theta_{u}- \Theta_{u,\M}\trinr_{1,2,h} 
	\le \trinl (1- I_\M) \Theta_{u}\trinr_{1,2,h} 
	+ \trinl    (1  - J)  \brho_\M  \trinr_{1,2,h} + \trinl   J  \brho_\M \trinr_{1}.
	\end{align}
	Lemma \ref{Morley_Interpolation}.$(a)$ shows that $\Theta_u-I_\M \Theta_u $ is orthogonal to $\Phi_\M$ for all $\Phi_\M \in V_\M$ and so Lemma \ref{Morley_Interpolation}.$(b)$ and the Pythagoras theorem proves that
	\begin{equation}\label{pythagoras}
	h^{-2}\trinl \Theta_{u}- I_\M \Theta_{u} \trinr_{1,2,h}^2 \le C_I^2\trinl \Theta_{u}- I_\M \Theta_{u}\trinr_{\NC}^2=C_I^2(\trinl \Theta_{u}-  \Theta_{u,\M}\trinr_{\NC}^2-\trinl \brho_\M \trinr_\NC^2)
	\le C_I^2 \trinl \Theta_{u}-  \Theta_{u,\M}\trinr_{\NC}^2.
	\end{equation}
	Lemma \ref{hctenrich}.$(d)$ with $v=0$ and the Pythagoras theorem in the above displayed inequality show
	\begin{equation} \label{pythagoras1}
	h^{-1}	\trinl    \brho_\M  - J  \brho_\M  \trinr_{1,2,h}  \le \Lambda_\jc \trinl  \brho_\M \trinr_{\NC}
	\lesssim \trinl \Theta_{u}-  \Theta_{u,\M}\trinr_{\NC},
	\end{equation}
where $`\lesssim'$ absorbs $\Lambda_\jc$ and $C_I$. \eqref{triang}--\eqref{pythagoras1} concludes the first step and shows	
	\begin{equation}\label{step1.adjoint}
	\trinl \Theta_{u}- \Theta_{u,\M}\trinr_{1,2,h} 
	\lesssim { h} \trinl \Theta_{u}-  \Theta_{u,\M}\trinr_{\NC}+ \trinl   J  \brho_\M \trinr_{1}.
	\end{equation}	
	{\bf Step 2 (estimates \trinl $J  \brho_\M \trinr_1$ in 
	\eqref{step1.adjoint}).} For a given ${\bf g} \in {\bH}^{-1}(\Omega)$,
	consider the dual problem that   seeks $\bchi_{\bg} \in \bV$ such that 
\begin{align}
& A(\bchi_{\bg}, \Phi) +2 B(\Psi_u, \bchi_{\bg}, \Phi)= \langle {\bf g}, \Phi \rangle \; \; \; \; 
\text{ for all }  \Phi \in {\bV}. \label{chig}
\end{align}
The existence of solution to \eqref{chig} and the regularity results stated below follows from Theorem \ref{th2.5}.
Note that $ \bchi_{\bg}  \in \bV \cap {\bf H}^{2+\gamma}(\Omega)$ and 
\begin{align} \label{regular}
\trinl  \bchi_{\bg}  \trinr_2 \lesssim \trinl {\bf g}\trinr_{-1} \text{   and   }  
\trinl \bchi_{\bg} \trinr_{2+\gamma} \lesssim \trinl {\bf g}\trinr_{-1}.
\end{align}	
 Choose ${\bf g}=- \Delta J \brho_\M \in L^2(\O)$ and $\Phi= J \brho_\M \in \bV$. This and elementary algebra eventually lead to 
\begin{align} \label{enriched.1}
&\| \nabla J \brho_\M\|^2
	 =  A_\NC(\bchi_{\bg}, (J-1) \brho_\M) + 2(B_\NC(\Psi_u, \bchi_{\bg}, (J-1)\brho_\M ) + B_\NC(\Psi_u, \bchi_{\bg},(I_\M-1) \Theta_{u}))   \nonumber \\
	& \quad + (A_\NC(\bchi_{\bg},(I_\M -1)\Theta_{u}) +A_\NC(\bchi_{\bg},\Theta_{u}-\Theta_{u,\M})) + 	2B_\NC(\Psi_u, \bchi_{\bg},\Theta_{u}-\Theta_{u,\M}) 
	  = \sum_{i=1}^4 T_i.
	\end{align}
	{\bf Step 3} estimates the  terms $T_1, \cdots, T_4$. Lemma \ref{Anc.bound}.$(a)$ shows that
	\begin{equation*}
	T_1:=A_\NC(\bchi_{\bg}, (J-1) \brho_\M) \lesssim h^{\gamma}\trinl \brho_\M\trinr_\NC  \trinl \bchi_{\bg} \trinr_{2+\gamma}\lesssim  h^{\gamma}\trinl \Theta_u-\Theta_{u,\M}\trinr_\NC  \trinl \bchi_{\bg} \trinr_{2+\gamma}
	\end{equation*}
	with \eqref{pythagoras1} in the end. Lemmas \ref{Bnc bound}.$(e)$, \ref{hctenrich}.$(d)$ with $v=\brho_\M$, \ref{Morley_Interpolation}. $(b)$, \eqref{pythagoras} and \eqref{pythagoras1} imply
	\begin{align*}
	\frac{1}{2}T_2& :=B_\NC(\Psi_u, \bchi_{\bg}, (J-1)\brho_\M ) 
	+B_\NC(\Psi_u, \bchi_{\bg},(I_\M-1) \Theta_{u})
	\lesssim h^2\trinl \Psi_u\trinr_{2+\gamma} \trinl \bchi_{\bg}\trinr_{2+\gamma}\trinl  \Theta_{u}-\Theta_{u,\M}\trinr_{\NC}.
	\end{align*}
	Simple manipulations lead to 
	\begin{align} \label{t4plust5}
	T_3&:=  A_\NC(\bchi_{\bg},(I_\M-1) \Theta_{u}) +
	A_\NC((1-I_\M) \bchi_{\bg},\Theta_{u}-\Theta_{u,\M})
	 + A_\NC((1-J) I_\M\bchi_{\bg},\Theta_{u}-\Theta_{u,\M}) \nonumber \\ & \qquad + 
	A_\NC(J I_\M\bchi_{\bg},\Theta_{u}-\Theta_{u,\M}). 
	\end{align}
	Lemma \ref{Morley_Interpolation}.$(a)$ shows $A_\NC(\Phi_\M,I_\M \Theta_{u}-\Theta_{u})=0=A_\NC(\bchi_{\bg}-I_\M\bchi_{\bg},\Phi_\M)$ for all $\Phi_\M \in V_\M$. 
	This shows that the  first two terms in  \eqref{t4plust5} is
	$$A_\NC(\bchi_{\bg}-I_\M \bchi_{\bg},I_\M \Theta_{u}-\Theta_{u}) 
	+  A_\NC(\bchi_{\bg}-I_\M\bchi_{\bg},\Theta_{u}-\Theta_{u,\M})=0.$$
The boundedness of $A_\NC(\bullet,\bullet)$, Lemma \ref{hctenrich}.$(d)$ with $v=\bchi_{\bg}$ and Lemma \ref{Morley_Interpolation}.$(c)$ result in an estimate for the third term in \eqref{t4plust5} as 
\begin{align} \label{a20}
	A_\NC((1-J) I_\M\bchi_{\bg},\Theta_{u}-\Theta_{u,\M})\le h^\gamma\trinl \bchi_{\bg}\trinr_{2+\gamma}\trinl  \Theta_{u}-\Theta_{u,\M}\trinr_{\NC}.
	\end{align}
	Lemma \ref{hctenrich}.$(c)$ shows $A_\NC(J I_\M\bchi_{\bg}-I_\M \bchi_{\bg},\Theta_{u,\M})=0$. This, \eqref{auxiadje_12} and \eqref{auxiadje_1_dis} lead to an expression for the last term in \eqref{t4plust5} as
	\begin{align}
	 & A_\NC(J I_\M\bchi_{\bg},\Theta_{u}-\Theta_{u,\M})=A_\NC(J I_\M\bchi_{\bg},\Theta_{u})-A_\NC( I_\M\bchi_{\bg},\Theta_{u,\M})\nonumber\\
	&\qquad =(\Psi_{u}-\Psi_{d},J I_\M\bchi_{\bg})-2{B}(\Psi_{u}, J I_\M\bchi_{\bg}, \Theta_u)-(\Psi_{u,\M}- \Psi_d, I_\M\bchi_{\bg})+2{B}_{\NC}(\Psi_{u,\M},I_\M\bchi_{\bg},  \Theta_{u,\M})\nonumber\\
	&\qquad =(\Psi_{u}-\Psi_{d},(J-1) I_\M\bchi_{\bg})-(\Psi_{u,\M}- \Psi_u, I_\M\bchi_{\bg})-2{B}(\Psi_{u}, J I_\M\bchi_{\bg}, \Theta_u)+2{B}_{\NC}(\Psi_{u,\M},I_\M\bchi_{\bg},  \Theta_{u,\M}). \label{t4plust5a}
	\end{align}
		Lemma \ref{hctenrich}.$(b)$ shows $\Pi_0 \bz=0$ for $z=(J-1) I_\M\bchi_{\bg}$. This Cauchy-Schwarz inequality, Lemmas \ref{hctenrich}.$(d)$ with $v=\bchi_{\bg}$, Lemma \ref{staest}.$(a)$ and \ref{Morley_Interpolation}.$(b)$-$(c)$ lead to the estimate for the first two terms of \eqref{t4plust5a} as
\begin{align}\label{t4plust5b}
\begin{split}
	(\Psi_{u}-\Psi_{d}-\Pi_0(\Psi_{u}-\Psi_{d}),(J-1) I_\M\bchi_{\bg}) & \lesssim h^{2+\gamma}{\rm{osc}}_0(\Psi_{u}-\Psi_{d})\trinl \bchi_{\bg} \trinr_{2+\gamma},  \\
	(\Psi_{u,\M}- \Psi_u, I_\M\bchi_{\bg}) & \lesssim \trinl \Psi_{u,\M}- \Psi_u\trinr\trinl \bchi_{\bg} \trinr_{2}. 
	\end{split}
	\end{align}
	The terms involving trilinear forms in \eqref{t4plust5a} are estimated now. The  orthogonality property of $J$ in Lemma \ref{hctenrich}.$(c)$ shows that $B_{\NC}(\Psi_{u,\M},\bchi_{\bg}-J I_\M \bchi_{\bg},\P_0\Theta_u)=0.$ This and a simple manipulation (omitting a factor 2) lead to an expression for the last two terms of \eqref{t4plust5a}  combined with $T_4$ as 
	\begin{align} \label{star}
	\begin{split}
& {B}_{\NC}(\Psi_{u}-\Psi_{u,\M},(1- J I_\M)\bchi_{\bg}, \Theta_u)+
	{B}_{\NC}(\Psi_{u,\M}, (1- J I_\M)\bchi_{\bg}, (1-\P_0) \Theta_u)\\
	&\qquad 
	+{B}_{\NC}(\Psi_{u,\M},I_\M\bchi_{\bg},  \Theta_{u,\M})-B_\NC(\Psi_u,\bchi_{\bg},\Theta_{u,\M}). 
	\end{split}
	\end{align}
The triangle inequality with $I_\M \bchi_g$, Lemmas \ref{Bnc bound}.$(b)$, \ref{Morley_Interpolation}.$(c)$ and \ref{hctenrich}.$(d)$ with $v=\bchi_g$ result in
\begin{align} \label{stara}
{B}_{\NC}(\Psi_{u}-\Psi_{u,\M},(1- JI_\M)\bchi_{\bg}, \Theta_u)
	 \lesssim h^\gamma \trinl \bchi_{\bg} \trinr_{2+\gamma}\trinl \Theta_u \trinr_{2+\gamma}\trinl \Psi_u-\Psi_{u,\M} \trinr_\NC.
	 \end{align}
	Lemmas \ref{Bnc bound}.$(a)$, \ref{Morley_Interpolation}.$(c)$ and \ref{hctenrich}.$(d)$  with $v=\bchi_g$ show
	\begin{align} \label{starb}
	 {B}_{\NC}(\Psi_{u,\M}, (1- J I_\M)\bchi_{\bg}, (1-\P_0) \Theta_u) \lesssim h^\gamma \trinl \bchi_{\bg} \trinr_{2+\gamma}\trinl \Psi_u \trinr_{2+\gamma}\trinl \Theta_u-\P_0 \Theta_u  \trinr_{0,\infty}.
	 \end{align}
The integral mean property of $I_\M$ in Lemma \ref{Morley_Interpolation}.$(a)$ shows that $B_{\NC}(\Psi_{u,\M},I_\M\bchi_{\bg}-\bchi_{\bg},\P_0\Theta_u)=0.$ 	 This and a simple manipulation show that the last two terms in \eqref{star} can be rewritten as
	 \begin{align}
	&\hspace{-0.4cm} B_{\NC}(\Psi_u-\Psi_{u,\M},(1-I_\M)\bchi_{\bg}, \Theta_{u,\M}) +B_{\NC}(\Psi_u,(1-I_\M)\bchi_{\bg}, \Theta_u-\Theta_{u,\M})
	 +B_{\NC}(\Psi_u-\Psi_{u,\M},(I_\M-1)\bchi_{\bg},\Theta_u) \nonumber \\
	 &  \hspace{-0.4cm} +B_{\NC}(\Psi_{u,\M},(I_\M-1)\bchi_{\bg},(1-\P_0)\Theta_u)+B_{\NC}(\Psi_u-\Psi_{u,\M},\bchi_{\bg}, \Theta_u-\Theta_{u,\M}) +B_{\NC}(\Psi_{u,\M}-\Psi_u,\bchi_{\bg}, \Theta_{u})=\sum_{i=1}^6 \mT_i. \label{t1011213}
	\end{align}
	The terms $\mT_1,\cdots,\mT_6$ are estimated next. The boundedness and interpolation estimates in 
	Lemmas \ref{Bnc bound}.$(a)-(b)$ and \ref{Morley_Interpolation}.$(c)$ prove
	\begin{align*}
	\mT_1 & \lesssim h^\gamma \trinl \bchi_{\bg} \trinr_{2+\gamma}\trinl \Theta_u \trinr_{2+\gamma}\trinl \Psi_u - \Psi_{u,\M} \trinr_\NC,\, \mT_2 \lesssim h^\gamma \trinl \bchi_{\bg} \trinr_{2+\gamma}\trinl \Psi_u \trinr_{2+\gamma}\trinl \Theta_u-\Theta_{u,\M} \trinr_\NC \\
	\mT_3&\lesssim h^\gamma\trinl \bchi_{\bg} \trinr_{2+\gamma}\trinl \Theta_u \trinr_{2+\gamma}\trinl \Psi_u-\Psi_{u,\M} \trinr_\NC,\,
	\mT_4\lesssim h^\gamma \trinl \bchi_{\bg} \trinr_{2+\gamma}\trinl \Psi_u \trinr_{2+\gamma}\trinl\Theta_u-\P_0 \Theta_u \trinr_{{0,\infty}}\\
	\mT_5 &\lesssim \trinl \bchi_{\bg}\trinr_{2+\gamma}\trinl \Psi_u-\Psi_{u,\M} \trinr_\NC \trinl \Theta_u-\Theta_{u,\M} \trinr_\NC.
	\end{align*}	
	Lemma \ref{Bnc bound adjointH1error} shows
	$\mT_6 \lesssim \left(h^\gamma\trinl \Psi_u-\Psi_{u,\M} \trinr_\NC+\trinl \Psi_u-\Psi_{u,\M} \trinr\right)\trinl \bchi_{\bg} \trinr_{2+\gamma}\trinl \Theta_u \trinr_{2+\gamma}.$
	A substitution of $\mT_1$-$\mT_6$ in \eqref{t1011213} and the resulting estimate with \eqref{stara} and \eqref{starb} in \eqref{star} leads to a bound for the terms involving the trilinear form $B_\NC(\bullet,\bullet,
	\bullet)$ as
	\begin{align*}
& 	 \trinl \bchi_{\bg}  \trinr_{2+\gamma}\big( \trinl \Theta_u \trinr_{2+\gamma}\left(h^\gamma\trinl \Psi_u-\Psi_{u,\M} \trinr_\NC+\trinl \Psi_u-\Psi_{u,\M} \trinr\right)  +\trinl \Psi_u-\Psi_{u,\M} \trinr_\NC \trinl \Theta_u-\Theta_{u,\M} \trinr_\NC \nonumber \\
 &  \qquad +h^\gamma \trinl\Psi_u \trinr_{2+\gamma} (\trinl \Theta_u-\Theta_{u,\M} \trinr_\NC+\trinl\Theta_u-\P_0\Theta_u\trinr_{0,\infty})\big).
	\end{align*}
	 This expression and \eqref{t4plust5b} is first substituted in \eqref{t4plust5a}, the resulting expression and \eqref{a20} is substituted in \eqref{t4plust5}  and utilized in   \eqref{enriched.1} with bounds for $T_1$ and $T_2$. In combination with 
	  $\trinl \bchi_{\bg} \trinr_{2+\gamma} \lesssim \trinl {\bf g}\trinr_{-1} \lesssim \| \nabla J \brho_\M\|$ from \eqref{regular}, $\trinl \Theta_u \trinr_{2+\gamma},\, \trinl \Psi_u \trinr_{2+\gamma} \lesssim 1$ this yields
	\begin{align*}
	\| \nabla J \brho_\M\|&\lesssim h^\gamma\trinl \Theta_u \trinr_{2+\gamma}\trinl\Psi_u-\Psi_{u,\M}\trinr_\NC +\trinl\Psi_u-\Psi_{u,\M} \trinr_\NC \trinl\Theta_u-\Theta_{u,\M}\trinr_\NC+ (1+\trinl \Theta_u \trinr_{2+\gamma})\trinl\Psi_u-\Psi_{u,\M} \trinr\\
	&\quad  +h^\gamma(1+\trinl\Psi_u \trinr_{2+\gamma}) \trinl\Theta_u-\Theta_{u,\M}\trinr_\NC+h^\gamma\trinl\Psi_u \trinr_{2+\gamma}\trinl\Theta_u-\P_0\Theta_u\trinr_{0,\infty} +h^{2+\gamma}{\rm{osc}}_0(\Psi_{u}-\Psi_{d}).
	\end{align*}
	This, $\trinl \Theta_u \trinr_{2+\gamma},\, \trinl \Psi_u \trinr_{2+\gamma} \lesssim 1$ and \eqref{step1.adjoint} lead to the desired estimate.
\end{proof}

\end{appendices}

\end{document}